\documentclass{amsart}

\usepackage[all]{xypic}
\usepackage{tikz}
\usetikzlibrary{arrows}
\usepackage{epsf}
\usepackage{verbatim} 
\usepackage{amsmath}
\usepackage{amsfonts}
\usepackage{amssymb}
\usepackage{mathrsfs}
\usepackage{amsthm}
\usepackage{newlfont}

 \newtheorem{thm}{Theorem}
 
 \newtheorem{cor}[thm]{Corollary}
 
  \newtheorem{conj}[thm]{Conjecture}
 
 \newtheorem{lem}[thm]{Lemma}
 
 \newtheorem{prop}[thm]{Proposition}
 
 \theoremstyle{definition}
 \newtheorem{defn}[thm]{Definition}
 
 \theoremstyle{definition}

 \theoremstyle{remark}
 \newtheorem{rem}[thm]{Remark}
 
 \theoremstyle{definition}
 \newtheorem{example}[thm]{Example}

 \numberwithin{thm}{section}
 \numberwithin{equation}{section}

 \font \cyr=wncyr10
 \define\Sh{\hbox{\cyr Sh}}

 \newcommand{\Mat}{\mathrm{Mat}}
 \newcommand{\JL}{\mathrm{JL}}
 
 \newcommand{\an}{\mathrm{an}}
 \newcommand{\ab}{\mathrm{ab}}
 \newcommand{\rig}{\mathrm{rig}}
 \newcommand{\Hom}{\mathrm{Hom}}
 
 \newcommand{\Spec}{\mathrm{Spec}}
 \newcommand{\Frob}{\mathrm{Frob}}

 \newcommand{\Aut}{\mathrm{Aut}}
 
 \newcommand{\End}{\mathrm{End}}
 \newcommand{\Pic}{\mathrm{Pic}}
 
 \newcommand{\ord}{\mathrm{ord}}

 \newcommand{\Gal}{\mathrm{Gal}}
 \newcommand{\GL}{\mathrm{GL}}
 \newcommand{\PGL}{\mathrm{PGL}}
 
 \newcommand{\sep}{\mathrm{sep}}

 \newcommand{\rank}{\mathrm{rank}}
 \newcommand{\un}{\mathrm{nr}}

 \newcommand{\tor}{\mathrm{tor}}
 
 \newcommand{\Stab}{\mathrm{Stab}}
 
 \newcommand{\new}{\mathrm{new}}
 \newcommand{\old}{\mathrm{old}}

 \newcommand{\Tr}{\mathrm{Tr}}

 \renewcommand{\mod}{\mathrm{mod}}
 
 \newcommand{\fl}{\mathfrak l}
 \newcommand{\fb}{\mathfrak b}
 
 \newcommand{\fp}{\mathfrak p}
 \newcommand{\fq}{\mathfrak q}
 \newcommand{\fr}{\mathfrak r}
 \newcommand{\fn}{\mathfrak n}
 \newcommand{\fm}{\mathfrak m}
 \newcommand{\fd}{\mathfrak d}

 \newcommand{\fE}{\mathfrak E}
 \newcommand{\fM}{\mathfrak M}
 \newcommand{\fI}{\mathfrak I}
 \newcommand{\cO}{\mathcal{O}}
 
 \newcommand{\cM}{\mathcal{M}}
 \renewcommand{\cH}{\mathcal{H}}

 \newcommand{\cE}{\mathcal{E}}
 
 \newcommand{\cC}{\mathcal{C}}

 \newcommand{\cG}{\mathcal{G}}
 \newcommand{\cJ}{\mathcal{J}}
 \newcommand{\cI}{\mathcal{I}}
 \newcommand{\cT}{\mathcal{T}}
 
 \newcommand{\cS}{\mathcal{S}}
 \renewcommand{\cL}{\mathcal{L}}
 
 \renewcommand{\cD}{\mathcal{D}}

\newcommand{\sT}{\mathscr{T}}

\newcommand{\sD}{\mathscr{D}}

 \newcommand{\gm}{\mathbb{G}}

 \newcommand{\C}{\mathbb{C}}
 \newcommand{\F}{\mathbb{F}}
 \newcommand{\Q}{\mathbb{Q}}
 \newcommand{\W}{\mathbb{W}}
 \newcommand{\Z}{\mathbb{Z}}
 \newcommand{\A}{\mathbb{A}}
 
 \newcommand{\p}{\mathbb{P}}
 \newcommand{\T}{\mathbb{T}}
 
 \newcommand{\N}{\mathbb{N}}
 \newcommand{\Nr}{\mathrm{Nr}}
 \newcommand{\disc}{\mathrm{disc}}

 \newcommand{\Ell}{\mathcal{E}\ell\ell}
 \newcommand{\eps}{\varepsilon}
 \newcommand{\G}{\Gamma}
 \newcommand{\To}{\longrightarrow}
 \newcommand{\bs}{\setminus}
 \newcommand{\Fi}{F_\infty}
 \newcommand{\bD}{\overline{\Delta}}

 \newcommand{\bG}{\overline{\Gamma}}

\begin{document}

\title[The Eisenstein ideal and Jacquet-Langlands isogeny]
{The Eisenstein ideal and Jacquet-Langlands isogeny over function fields}

\author{Mihran Papikian}
\address{Department of Mathematics, Pennsylvania State University, University Park, PA 16802, U.S.A.}
\email{papikian@psu.edu}
\author{Fu-Tsun Wei}
\address{Institute of Mathematics, Academia Sinica, 6F, Astronomy-Mathematics Building, No. 1, Sec. 4, Roosevelt Road, Taipei 10617, Taiwan}
\email{ftwei@math.sinica.edu.tw}

\thanks{The first author was supported in part by the Simons Foundation.} 

\thanks{The second author was partially supported by National Science Council and Max Planck Institute for Mathematics.}





\begin{abstract}
Let $\fp$ and $\fq$ be two distinct prime ideals of $\mathbb{F}_q[T]$. 
We use the Eisenstein ideal of the Hecke algebra of the Drinfeld modular curve $X_0(\fp\fq)$ 
to compare the rational torsion subgroup of the Jacobian $J_0(\fp\fq)$ 
with its subgroup generated by the cuspidal divisors, and to produce explicit 
examples of Jacquet-Langlands isogenies. Our results are stronger than what 
is currently known about the analogues of these problems over $\Q$.  
\end{abstract}

\maketitle

\tableofcontents


\section{Introduction} 

\subsection{Motivation} Let $\F_q$ be a finite field with $q$ 
elements, where $q$ is a power of a prime number $p$. Let $A=\F_q[T]$ 
be the ring of polynomials in indeterminate $T$ with coefficients 
in $\F_q$, and $F=\F_q(T)$ the field of fractions of $A$. 
The degree map $\deg: F\to \Z\cup \{-\infty\}$, which associates 
to a non-zero polynomial its degree in $T$ and $\deg(0)=-\infty$, defines 
a norm on $F$ by $|a|:=q^{\deg(a)}$. The corresponding place of $F$ 
is usually called the \textit{place at infinity}, and is denoted by $\infty$. 
We also define a norm and degree on the ideals of $A$ by $|\fn|:=\#(A/\fn)$ and $\deg(\fn):=\log_q|\fn|$. 
Let $\Fi$ denote the completion of $F$ at $\infty$, and $\C_\infty$ denote the 
completion of an algebraic closure of $\Fi$. Let $\Omega:=\C_\infty - \Fi$ be the \textit{Drinfeld half-plane}.  

Let $\fn\lhd A$ be a non-zero ideal. The level-$\fn$ \textit{Hecke congruence subgroup} of $\GL_2(A)$  
$$
\G_0(\fn):=\left\{\begin{pmatrix} a & b \\  c & d\end{pmatrix}\in \GL_2(A)\ \bigg|\ c\equiv 0\ \mod\  \fn \right\}  
$$
plays a central role in this paper. This group acts on $\Omega$ via linear fractional transformations. 
Drinfeld proved in \cite{Drinfeld} that the quotient $\G_0(\fn)\bs \Omega$ 
is the space of $\C_\infty$-points of an affine curve $Y_0(\fn)$ defined over $F$,  
which is a moduli space of rank-$2$ Drinfeld modules 
(we give a more formal discussion of Drinfeld modules and their moduli schemes in Section \ref{sDMC}).  
The unique smooth projective curve over $F$ containing $Y_0(\fn)$ as an 
open subvariety is denoted by $X_0(\fn)$. The \textit{cusps} of $X_0(\fn)$ are the 
finitely many points of the complement of $Y_0(\fn)$ in $X_0(\fn)$; 
the cusps generate a finite subgroup $\cC(\fn)$ of the Jacobian variety $J_0(\fn)$ of $X_0(\fn)$, 
called the \textit{cuspidal divisor group}. 
By the Lang-N\'eron theorem, the group of $F$-rational points of $J_0(\fn)$ is 
finitely generated, in particular, its torsion subgroup $\cT(\fn):=J_0(\fn)(F)_\tor$ is finite. 
It is known that when $\fn$ is square-free $\cC(\fn)\subseteq \cT(\fn)$. 

For a square-free ideal $\fn\lhd A$ divisible by an even number of primes,  
let $D$ be the division quaternion algebra over $F$ with discriminant $\fn$. 
The group of units $\G^\fn$ of a maximal $A$-order in $D$  
acts on $\Omega$, and the quotient $\G^\fn\bs \Omega$ 
is the space of $\C_\infty$-points of a smooth projective curve $X^\fn$ defined over $F$; this curve  
is a moduli space of $\sD$-elliptic sheaves introduced in \cite{LRS}. 
Let $J^\fn$ be the Jacobian variety of $X^\fn$. 

The analogy between $X_0(\fn)$ 
and the classical modular curves $X_0(N)$ over $\Q$ classifying elliptic curves with 
$\G_0(N)$-structures is well-known and has been extensively studied over the last 35 years. 
Similarly, the modular curves $X^\fn$ are the function field analogues of Shimura curves $X^N$
parametrizing abelian surfaces equipped with an action of the indefinite quaternion algebra over $\Q$ 
with discriminant $N$. 

\vspace{0.1in}

Let $\T(\fn)$ be the $\Z$-algebra generated by the Hecke operators 
$T_\fm$, $\fm\lhd A$, acting on the group $\cH_0(\sT, \Z)^{\G_0(\fn)}$  
of $\Z$-valued $\G_0(\fn)$-invariant 
cuspidal harmonic cochains on the Bruhat-Tits tree $\sT$ of $\PGL_2(\Fi)$. 
The \textit{Eisentein ideal} $\fE(\fn)$ of $\T(\fn)$ is the ideal generated by the elements $T_\fp-|\fp|-1$, 
where $\fp \nmid \fn$ is prime. 
In this paper we study the Eisenstein ideal in the case when $\fn=\fp\fq$ is a product of two distinct 
primes, with the goal of applying this theory to two important 
arithmetic problem: 1) comparing $\cT(\fn)$ with $\cC(\fn)$, and 2) constructing 
explicit homomorphisms $J_0(\fn)\to J^{\fn}$. 
Our proofs use the rigid-analytic uniformizations of $J_0(\fn)$ and $J^\fn$ over $\Fi$. 
It seems that the existence of actual geometric fibres at $\infty$ allows one to prove stronger 
results than what is currently known about either of these problems in the classical setting; 
this is specific to function fields since the analogue of $\infty$ for $\Q$ is the archimedean place.

\vspace{0.1in}

Our initial motivation for studying $\fE(\fp\fq)$ came from an attempt to prove a function field 
analogue of Ogg's conjecture \cite{Ogg} about the so-called Jacquet-Langlands isogenies. 
We briefly recall what this is about. A geometric consequence of the Jacquet-Langlands correspondence \cite{JL}
is the existence of Hecke-equivariant $\Q$-rational isogenies 
between the new quotient $J_0(N)^\new$ of $J_0(N)$ and the Jacobian $J^N$ of $X^N$; see \cite{RibetIsogeny}. 
(Here $N$ is a square-free integer with an even number of prime factors.)
The proof of the existence of aforementioned isogenies relies on Faltings' isogeny theorem, 
so provides no information about them beyond the existence. It is a major open problem in this area to 
make the isogenies more canonical (cf. \cite{Helm}). In \cite{Ogg},  
Ogg made several predictions about the kernel of an isogeny $J_0(N)^\new\to J^N$
when $N=pp'$ is a product of two distinct primes and $p=2,3,5,7,13$.  
As far as the authors are aware, Ogg's conjecture remains open except for the special cases when $J^N$ has dimension $1$ 
($N=14, 15, 21, 33, 34$) or dimension $2$ 
($N=26, 38, 58$). In these cases, $J^N$ and $J_0(N)^\new$ are either elliptic curves or, up to isogeny, decompose into a product 
of two elliptic curves given by explicit Weierstrass equations.  One can then find an isogeny $J_0(N)^\new\to J^N$ by 
studying the isogenies between these elliptic curves; see the proof of Theorem 3.1 in \cite{GoRo}. 
This argument does not generalize to $J^N$ of dimension $\geq 3$ because 
they contain absolutely simple abelian varieties of dimension $\geq 2$, 
and one's hold on such abelian varieties is decidedly more fleeting. 

Now returning to the setting of function fields, let $\fn\lhd A$ be a square-free ideal 
with an even number of prime factors.  The global Jacquet-Langlands correspondence over $F$, combined 
with the main results in \cite{Drinfeld} and \cite{LRS}, and Zarhin's isogeny theorem, 
implies the existence of a Hecke-equivariant $F$-rational isogeny $J_0(\fn)^\new\to J^\fn$. 
In Section \ref{sJL}, by studying the groups of connected components of the N\'eron models of 
$J_0(\fn)$ and $J^{\fn}$, we propose a function field analogue of Ogg's conjecture (see Conjecture \ref{conjJL}). 
This conjecture predicts that, when $\fn=\fp\fq$ is a product of two distinct primes with $\deg(\fp)\leq 2$, 
there is a Jacquet-Langlands isogeny whose 
kernel comes from cuspidal divisors and is isomorphic to a specific abelian group.  
Our approach to proving this conjecture starts with the  
observation that $\cC(\fn)$ is annihilated by the Eisenstein ideal $\fE(\fn)$ acting on $J_0(\fn)$, so 
we first try to show that there is a Jacquet-Langlands isogeny whose kernel is annihilated by $\fE(\fn)$, and 
then try to describe the kernel of the Eisenstein ideal $J[\fE(\fn)]$ in $J_0(\fn)$ explicitly enough to pin down the kernel of 
the isogeny. 
This naturally leads to the study of $J[\fE(\fn)]$ for composite $\fn$.  
On the other hand, $J[\fE(\fn)]$ also 
plays an important role in the  
analysis of $\cT(\fn)$, as was first demonstrated by Mazur  
in his seminal paper \cite{Mazur} in the case of classical modular Jacobian $J_0(p)$ of prime level.
These two applications of the theory of the Eisenstein ideal constitute the main theme of this paper. 


\subsection{Main results} 
The \textit{Shimura subgroup} $\cS(\fn)$ of $J_0(\fn)$ is the kernel of the 
homomorphism $J_0(\fn)\to J_1(\fn)$ induced by the natural morphism $X_1(\fn)\to X_0(\fn)$ of 
modular curves (see Section \ref{sSS}). 

Assume $\fp\lhd A$ is prime. Define $N(\fp)=\frac{|\fp|-1}{q-1}$ if $\deg(\fp)$ is odd, and define 
$N(\fp)=\frac{|\fp|-1}{q^2-1}$, otherwise. In \cite{Pal}, P\'al developed a theory of the Eisenstein ideal $\fE(\fp)$ 
in parallel with Mazur's paper \cite{Mazur}. 
In particular, he showed that $J[\fE(\fp)]$ is everywhere unramified of order $N(\fp)^2$, 
and is essentially generated by $\cC(\fp)$ and $\cS(\fp)$, both of which are cyclic of order $N(\fp)$. 
Moreover, $\cC(\fp)=\cT(\fp)$ and $\cS(\fp)$ is the largest $\mu$-type subgroup scheme of $J_0(\fp)$. 
These results are the analogues of some of the deepest results from \cite{Mazur}, whose proof first 
establishes that the completion of the Hecke algebra $\T(\fp)$ at any maximal ideal in the support of $\fE(\fp)$ 
is Gorenstein. 

\vspace{0.1in}

As we will see in Section \ref{sKEI}, even in the simplest composite level case
the kernel of the Eisenstein ideal $J[\fE(\fn)]$ has properties quite different from its prime level counterpart. 
For example, $J[\fE(\fn)]$ can be ramified, generally $\cS(\fn)$ has smaller order than $\cC(\fn)$, neither 
of these groups is cyclic, and $\cS(\fn)$ is not the largest $\mu$-type subgroup scheme of $J_0(\fn)$. 

First, we discuss our results about $\cC(\fn)$, $\cS(\fn)$, and $\cT(\fn)$: 

\begin{thm}\label{thmMAIN1}\hfill
\begin{enumerate}
\item We give a complete description of $\cC(\fp\fq)$ as an abelian group; see Theorem \ref{thm6.10}. 
\item For an arbitrary square-free $\fn$ we show that the group scheme $\cS(\fn)$ is $\mu$-type, 
and therefore annihilated by $\fE(\fn)$, and we give a complete description of $\cS(\fn)$ 
as an abelian group; see Proposition \ref{propSisEis} and Theorem \ref{thmSSS}. 
\item If $\ell\neq p$ is a prime number which does not divide $(q-1)\mathrm{gcd}(|\fp|+1, |\fq|+1)$, 
then the $\ell$-primary subgroups of $\cC(\fp\fq)$ and $\cT(\fp\fq)$ are equal; see Theorem \ref{thm8.3}. 
\end{enumerate}
\end{thm}
Usually, many of the primes dividing the order of $\cC(\fp\fq)$ 
satisfy the condition in (3), so, aside from a relatively small explicit set of primes, 
we can determine the $\ell$-primary subgroup $\cT(\fp\fq)_\ell$ of $\cT(\fp\fq)$. 
For example, (1) and (3) imply that if $\ell$ does not divide $(|\fp|^2-1)(|\fq|^2-1)$, then $\cT(\fp\fq)_\ell=0$. 
The most advantageous case for applying (3) is when $\deg(\fq)=\deg(\fp)+1$, since then $\mathrm{gcd}(|\fp|+1, |\fq|+1)$ 
divides $q-1$. In particular, if $q=2$ and $\deg(\fq)=\deg(\fp)+1$, then we conclude that 
the odd part of $\cT(\fp\fq)$ coincides with $\cC(\fp\fq)$. 
These results are qualitatively stronger than what is currently known about 
the rational torsion subgroup $J_0(N)(\Q)_\tor$ of classical modular Jacobians of composite square-free levels (cf. \cite{CL}). 

\begin{proof}[Outline of the Proof of Theorem \ref{thmMAIN1}] 
Although it was known that $\cC(\fn)$ is finite for any $\fn$ (see Theorem \ref{thm6.1}), there were no general 
results about its structure, besides the prime level case $\fn=\fp$. 
The curve $X_0(\fp)$ has two cusps, so $\cC(\fp)$ is cyclic; its  
order was computed by Gekeler in \cite{Uber}. 
The first obvious difference between the prime level and the 
composite level $\fn=\fp\fq$ is that $X_0(\fp\fq)$ has $4$ cusps, so $\cC(\fp\fq)$ 
is usually not cyclic and is generated by $3$ elements. 
To prove the result mentioned in (1), i.e., to compute the group structure of $\cC(\fp\fq)$, we follow the strategy in \cite{Uber},   
but the calculations become much more complicated. The idea is to use Drinfeld 
discriminant function to obtain upper bounds on the orders of cuspidal divisors,  
and then use canonical specializations of $\cC(\fp\fq)$ into the 
component groups of $J_0(\fp\fq)$ at $\fp$ and $\fq$ to obtain lower bounds on these orders. 

To deduce the group structure of $\cS(\fn)$ mentioned in (2) we use the rigid-analytic uniformizations 
of $J_0(\fn)$ and $J_1(\fn)$ over $\Fi$, and the ``changing levels'' result from \cite{GR}, 
to reduce the problem to a calculation with finite groups.  

The proof of (3) is similar to the proof of Theorem 7.19 in \cite{Pal}, 
although there are some important differences, too. Suppose $\ell$ is a prime that 
does not divide $q(q-1)$. Since $J_0(\fp\fq)$ has split toric reduction at $\infty$, 
the $\ell$-primary subgroup $\cT(\fp\fq)_\ell$ 
maps injectively into the component group $\Phi_\infty$ of $J_0(\fp\fq)$ at $\infty$. 
Using the Eichler-Shimura relations, one shows 
that the image of $\cT(\fp\fq)_\ell$ in $\Phi_\infty$ can be identified with a subspace of 
$\cH_0(\sT, \Z)^{\G_0(\fp\fq)}\otimes \Z/\ell^n\Z$ annihilated by the Eisenstein ideal $\fE(\fp\fq)$ for 
any sufficiently large $n\in \N$. Denote by $\cE_{00}(\fp\fq, \Z/\ell^n\Z)$ the subspace of 
$\cH_0(\sT, \Z)^{\G_0(\fp\fq)}\otimes \Z/\ell^n\Z$ annihilated by $\fE(\fp\fq)$. 
Then we have the inclusions 
$$
\cC(\fp\fq)_\ell\hookrightarrow \cT(\fp\fq)_\ell \hookrightarrow \cE_{00}(\fp\fq, \Z/\ell^n\Z). 
$$
The space $\cE_{00}(\fp\fq, \Z/\ell^n\Z)$ contains the 
reductions modulo $\ell^n$ of certain Eisenstein series. We prove that if $\ell$ does not divide 
$q(q-1)\mathrm{gcd}(|\fp|+1, |\fq|+1)$, then the whole $\cE_{00}(\fp\fq, \Z/\ell^n\Z)$ 
is generated by the reductions of these Eisenstein series (see Theorem \ref{thm3.9} and Lemma \ref{lem3.4}). This allows us to compute 
$\cE_{00}(\fp\fq, \Z/\ell^n\Z)$. It turns out that $\cE_{00}(\fp\fq, \Z/\ell^n\Z)\cong \cC(\fp\fq)_\ell$, 
and consequently $\cC(\fp\fq)_\ell=\cT(\fp\fq)_\ell$. To prove Theorem \ref{thm3.9}, we first prove
a version of the key Theorem 1 in the famous paper by Atkin and Lehner \cite{AL} 
for $\Z/\ell^n\Z$-valued harmonic cochains (see Theorem \ref{thmAL}). The fact that we need to work with $\Z/\ell^n\Z$ 
rather than $\C$ leads to technical difficulties, which results in the restriction $\ell\nmid q(q-1)\mathrm{gcd}(|\fp|+1, |\fq|+1)$. 
Note that in our definition the Hecke algebra $\T(\fp\fq)$ 
includes the operators $U_\fp$ and $U_\fq$. This is important since we need to deal systematically 
with ``old'' forms of level $\fp$ and $\fq$. The smaller algebra $\T(\fp\fq)^0$ 
generated by the Hecke operators $T_\fm$ with $\fm$ coprime to $\fp\fq$   
used by P\'al in \cite{Pal} and \cite{PalIJNT} is not sufficient for getting a handle on $\cE_{00}(\fp\fq, \Z/\ell^n\Z)$. 
\end{proof}

Now we concentrate on the case where we investigate the Jacquet-Langlands isogenies. 
We fix two primes $x$ and $y$ of $A$ of degree $1$ and $2$, respectively. 
This differs from our usual $\mathfrak{Fraktur}$ notation for ideals of $A$. This is 
done primarily to make it easy for the reader to distinguish the theorems which assume that the level 
is $xy$. 
Several sections in the paper are titled ``Special case'' and deal exclusively with the case $\fp\fq=xy$. 
Note that $X_0(\fp\fq)$ has genus $0$ if $\fp$ and $\fq$ are distinct primes with $\deg(\fp\fq)\leq 2$. 
The genus of $X_0(xy)$ is $q$, so this curve is the simplest example of a Drinfeld modular curve 
of composite level and positive genus. Also, by a theorem of Schweizer \cite{SchweizerHE}, $X_0(\fp\fq)$ is hyperelliptic 
if and only if $\fp=x$ and $\fq=y$, so one can think of this case as the hyperelliptic case.  

The cusps of $X_0(xy)$ can be naturally labelled $[x], [y], [1], [\infty]$; see Lemma \ref{lemCusps}. 
Let $c_x$ and $c_y$ denote the classes of divisors $[x]-[\infty]$ and $[y]-[\infty]$ in $J_0(xy)$. 
First, we show that (see Theorem \ref{thmCT})
$$
\cT(xy)=\cC(xy)=\langle c_x \rangle\oplus \langle c_y \rangle\cong \Z/(q+1)\Z\oplus \Z/(q^2+1)\Z. 
$$
The reason we can prove this stronger result compared to (3) of  
Theorem \ref{thmMAIN1} is that we can compute $\cE_{00}(xy, \Z/\ell^n\Z)$ without any restrictions 
on $\ell$, and we can deal with the $2$-primary torsion $\cT(xy)_2$ using the fact that $X_0(xy)$ is hyperelliptic.  

To simplify the notation, for the rest of this section denote $\T=\T(xy)$, $\fE=\fE(xy)$, $\cH:=\cH_0(\sT, \Z)^{\G_0(xy)}$, 
$\cH':=\cH(\sT, \Z)^{\G^{xy}}$, where this last group is the group of $\Z$-valued $\G^{xy}$-invariant 
harmonic cochains on $\sT$. We show that (see Corollary \ref{corT/E})
$$
\T/\fE\cong \Z/(q^2+1)(q+1)\Z,
$$
so the residue characteristic of any maximal ideal of $\T$ containing $\fE$ divides $(q^2+1)(q+1)$. 
The Jacquet-Langlands correspondence over $F$ implies that there is an isomorphism 
$\cH\otimes \Q\cong \cH'\otimes \Q$ which is compatible with the action of $\T$. 

\begin{thm}[See Theorems \ref{thmJL1} and \ref{thmJL2}]\label{thmMAIN2}\hfill 
\begin{enumerate}
\item If $\cH\cong \cH'$ as $\T$-modules, then there is an isogeny $J_0(xy)\to J^{xy}$ defined over $F$ 
whose kernel is cyclic of order $q^2+1$ and is annihilated by $\fE$. 
\item If $\cH\cong \cH'$ as $\T$-modules and for every prime 
$\ell|(q^2+1)$ the completion of $\T\otimes\Z_\ell$ at $\fM=(\fE, \ell)$ is Gorenstein, then 
there is an isogeny $J_0(xy)\to J^{xy}$ whose kernel is $\langle c_y \rangle\cong \Z/(q^2+1)\Z$.  
\end{enumerate}
\end{thm}
\begin{rem}
An isogeny $J_0(xy)\to J^{xy}$ with kernel $\langle c_y \rangle$ does not respect 
the canonical principal polarizations on the Jacobians since $\langle c_y \rangle$ is not 
a maximal isotropic subgroup of $J_0(xy)$ with respect to the Weil pairing. 
\end{rem}
\begin{proof}[Outline of the Proof of Theorem \ref{thmMAIN2}] Both $J_0(xy)$ and $J^{xy}$ have rigid-analytic uniformization 
over $\Fi$. The assumption that $\cH$ and $\cH'$ are isomorphic $\T$-modules allows us to 
identify the uniformizing tori of both Jacobians with $\T\otimes\C_\infty^\times$. Next, we show 
that the groups of connected components of the N\'eron models of $J_0(xy)$ and $J^{xy}$ at $\infty$ 
are annihilated by $\fE$. This allows us to identify the uniformizing lattices of the Jacobians with ideals in $\T$. 
These two observations, combined with a theorem of Gerritzen, imply (1). 
If in addition we assume that $\T_\fM$ is Gorenstein, then we get an explicit description of the kernel of the Eisenstein ideal 
from which (2) follows. 
\end{proof}

Proving that the assumptions in Theorem \ref{thmMAIN2} hold seems difficult. First, even though 
$\cH\otimes \Q$ and $\cH'\otimes \Q$ are isomorphic $\T$-modules, the integral isomorphism 
is much more subtle. It is related to a classical problem about the conjugacy classes of matrices in $\Mat_n(\Z)$; cf. \cite{LM}. 
Second, when $\ell|(q^2+1)$ the kernel of $\fM$ in $J_0(xy)$ is ramified, and Mazur's 
Eisenstein descent arguments for proving $\T_\fM$ is Gorenstein do not work in this 
ramified situation. (Both versions of Mazur's descent discussed in \cite[$\S\S$10,11]{Pal} rely on subtle 
arithmetic properties of $J_0(\fp)$ which are valid only for prime level.)

Nevertheless, both assumptions in Theorem \ref{thmMAIN2} can be verified computationally;  
Section \ref{sComputations} is devoted to these calculations. We were able to check 
the assumptions for several cases for each prime $q\leq 7$. In particular, 
we were able to go beyond dimension $2$, which is currently the only dimension 
where the Ogg's conjecture is known to be true over $\Q$.  
Section \ref{sComputations} is also of independent interest since it provides an 
algorithm for computing the action of Hecke operators on $\cH'$; this 
should be useful in other arithmetic problems dealing with $X^{xy}$. 
(An algorithm for computing the Hecke action on $\cH$ was already known from the 
work of Gekeler; see Remark \ref{rem10.2}.)

\subsection{Notation} 
Besides $\infty$, the other places of $F$ are in bijection with non-zero prime ideals of $A$.  
Given a place $v$ of $F$, we denote by $F_v$ the completion of $F$ at $v$, by $\cO_v$ 
the ring of integers of $F_v$, and by $\F_v$ the residue field of $\cO_v$. The 
valuation $\ord_v: F_v\to \Z$ is assumed to be normalized by $\ord_v(\pi_v)=1$, where 
$\pi_v$ is a uniformizer of $\cO_v$. The normalized absolute value on $\Fi$ is denoted by $|\cdot|$. 

Given a field $K$, we denote by $\bar{K}$ an algebraic closure of $K$ and $K^\sep$ 
a separable closure in $\bar{K}$. 
The absolute Galois group $\Gal(K^\sep/K)$ is denoted by $G_K$. 
Moreover, $F_v^\un$ and $\cO_v^\un$ 
will denote the maximal unramified extension of $F_v$ and its ring of integers, respectively. 

Let $R$ be a commutative ring with identity. We denote by $R^\times$ the group of 
multiplicative units of $R$. Let $\Mat_n(R)$ be the ring of $n\times n$ matrices over $R$, 
$\GL_n(R)$ the group of matrices whose determinant is in $R^\times$, and $Z(R)\cong R^\times$ 
the subgroup of  $\GL_n(R)$ consisting of scalar matrices. 

If $X$ is a scheme over a base $S$ and $S'\to S$ any base change, $X_{S'}$ 
denotes the pullback of $X$ to $S'$. If $S'=\Spec(R)$ is affine, 
we may also denote this scheme by $X_R$. By $X(S')$ we mean the $S'$-rational points 
of the $S$-scheme $X$, and again, if $S'=\Spec(R)$, we may also denote this set by $X(R)$. 

Given a commutative finite flat group scheme $G$ over a base $S$ (or just an 
abelian group $G$, or a ring $G$) and an integer $n$, $G[n]$ is the kernel of multiplication by 
$n$ in $G$, and $G_\ell$ is the maximal $\ell$-primary subgroup of $G$. The Cartier dual of $G$ is denoted by $G^\ast$.  

Given an ideal $\fn\lhd A$, by abuse of notation, we denote by the same symbol the unique 
monic polynomial in $A$ generating $\fn$. It will always be clear from the context 
in which capacity $\fn$ is used; for example, if $\fn$ appears in a matrix, column vector, or a 
polynomial equation, then the monic polynomial is implied. 
The prime ideals $\fp\lhd A$ are always assumed to be non-zero.


\section{Harmonic cochains and Hecke operators}\label{sAL} 

\subsection{Harmonic cochains} \label{sec2.1}

Let $G$ be an oriented connected graph in the sense of Definition 1 of $\S$2.1 in \cite{SerreT}. 
We denote by $V(G)$ and $E(G)$ its set of vertices and edges, respectively. 
For an edge $e\in E(G)$, let $o(e)$, $t(e)\in V(G)$ and $\bar{e}\in E(G)$ be its 
origin, terminus and inversely oriented edge, respectively. In particular, $t(\bar{e})=o(e)$ 
and $o(\bar{e})=t(e)$. We will assume that 
for any $v\in V(G)$ the number of edges with $t(e)=v$ is finite,  
and $t(e)\neq o(e)$ for any $e\in E(G)$ (i.e., $G$ has no loops).  A \textit{path} in $G$ is a 
sequence of edges $\{e_i\}_{i\in I}$ indexed by the set $I$ where $I=\Z$, $I=\N$ or 
$I=\{1,\dots, m\}$ for some $m\in \N$ such that $t(e_i)=o(e_{i+1})$ for every 
$i, i+1\in I$.  We say that the path is \textit{without backtracking} if $e_i\neq \bar{e}_{i+1}$ 
for every $i, i+1\in I$. We say that the path without backtracking $\{e_i\}_{i\in \N}$ is a \textit{half-line} 
if for every vertex $v$ of $G$ there is at most one index $n\in \N$ 
such that $v=o(e_n)$. 

Let $\G$ be a group acting on a graph $G$, i.e., $\G$ acts via
automorphisms. We say that $\G$ acts with \textit{inversion} if there is
$\gamma\in \G$ and $e\in E(G)$ such that $\gamma e=\bar{e}$. If
$\G$ acts without inversion, then we have a natural quotient graph
$\G\bs G$ such that $V(\G\bs G)=\G\bs V(G)$ and
$E(\G\bs G)=\G\bs E(G)$, cf. \cite[p. 25]{SerreT}.

\begin{defn}\label{defnHarmG}
Fix a commutative ring $R$ with identity. An $R$-valued \textit{harmonic cochain} on $G$ is a 
function $f: E(G)\to R$ that satisfies 
\begin{itemize}
\item[(i)] $$f(e)+f(\bar{e})=0\quad \text{for all $e\in E(G)$},$$
\item[(ii)] 
$$\sum_{\substack{e\in E(G) \\ t(e)=v}} f(e)=0\quad \text{for all $v\in V(G)$}.$$
\end{itemize}
Denote by $\cH(G, R)$ the group of $R$-valued harmonic cochains on $G$. 
\end{defn}

The most important graphs in this paper are the Bruhat-Tits tree $\sT$ of $\PGL_2(\Fi)$, and the 
quotients of $\sT$. We recall the definition and introduce some notation for later use. 
Fix a uniformizer $\pi_\infty$ of $\Fi$. 
The sets of vertices $V(\sT)$ and edges $E(\sT)$ are the cosets $\GL_2(\Fi)/Z(\Fi)\GL_2(\cO_\infty)$ 
and $\GL_2(\Fi)/Z(\Fi)\cI_\infty$, respectively, where $\cI_\infty$ is the Iwahori group:
$$
\cI_\infty=\left\{\begin{pmatrix} a & b\\ c & d\end{pmatrix}\in \GL_2(\cO_\infty)\ \bigg|\ c\in \pi_\infty\cO_\infty\right\}. 
$$
The matrix $\begin{pmatrix} 0 & 1\\ \pi_\infty & 0\end{pmatrix}$ 
normalizes $\cI_\infty$, so the multiplication from the right by this matrix on $\GL_2(\Fi)$ 
induces an involution on $E(\sT)$; this involution is $e\mapsto \bar{e}$. 
The matrices 
\begin{equation}\label{eq-setM}
E(\sT)^+=\left\{\begin{pmatrix} \pi_\infty^k & u \\ 0 & 1\end{pmatrix}\ \bigg|\
\begin{matrix} k\in \Z\\ u\in \Fi,\ u\ \mod\ \pi_\infty^k\cO_\infty\end{matrix}\right\}
\end{equation}
are in distinct left cosets of $\cI_\infty Z(\Fi)$, and there is a disjoint decomposition  (cf.\ \cite[(1.6)]{Improper})
$$
E(\sT)=E(\sT)^+\bigsqcup E(\sT)^+\begin{pmatrix} 0 & 1\\ \pi_\infty & 0\end{pmatrix}. 
$$
We call the edges in $E(\sT)^+$ \textit{positively oriented}. 

The group $\GL_2(\Fi)$ naturally acts on $E(\sT)$ by left multiplication. 
This induces an action on the group of $R$-valued functions on $E(\sT)$: 
for a function $f$ on $E(\sT)$ and $\gamma\in \GL_2(\Fi)$ we define the function $f|\gamma$ on $E(\sT)$ by 
$(f|\gamma)(e)=f(\gamma e)$. 
It is clear from the definition that $f|\gamma$ is harmonic if $f$ is harmonic, and 
for any $\gamma, \sigma\in \GL_2(\Fi)$ we have $(f|\gamma)|\sigma=f|(\gamma\sigma)$. 

Let $\G$ be a subgroup of $\GL_2(\Fi)$ which acts on $\sT$ without inversions. 
Denote by $\cH(\sT, R)^\G$ the subgroup of $\G$-invariant harmonic cochains, i.e.,  
$f|\gamma=f$ for all $\gamma\in \G$.
It is clear that $f\in \cH(\sT, R)^\G$ defines a function $f'$ on the quotient graph $\G\bs\sT$, and 
$f$ itself can be uniquely recovered from this function: If $e\in E(\sT)$ maps to $\tilde{e}\in E(\G\bs \sT)$ under the quotient map, 
then $f(e)=f'(\tilde{e})$. The conditions of harmonicity (i) and (ii) can be formulated 
in terms of $f'$ as follows. Since $\G$ acts without inversion, (i) is equivalent to 
\begin{itemize}
\item[(i$'$)] 
$$
f'(\tilde{e})+f'(\bar{\tilde{e}})=0\quad \text{for all $\tilde{e}\in E(\G\bs\sT)$}.
$$
\end{itemize}
Let $v\in V(\sT)$ and $\tilde{v}\in V(\G\bs\sT)$ be its image. The stabilizer group 
$$\G_v=\{\gamma\in \G\ |\ \gamma v=v\}$$  acts on the set $\{e\in E(\sT)\ |\ t(e)=v\}$, 
and the orbits correspond to $$\{\tilde{e}\in E(\G\bs \sT)\ |\ t(\tilde{e})=\tilde{v}\}.$$ 
Let $\G_e:=\{\gamma\in \G\ |\ \gamma e=e\}$; clearly $\G_e$ is a subgroup of $\G_{t(e)}$. 
The \textit{weight} of $e$ 
$$
w(e):=[\G_{t(e)}:\G_e]
$$
is the length of the orbit corresponding to $e$. Since $w(e)$ depends only on its image $\tilde{e}$ in $\G\bs \sT$, 
we can define $w(\tilde{e}):=w(e)$. Note that 
$\sum_{t(\tilde{e})=\tilde{v}} w(\tilde{e}) = q+1$. 
We stress that, in general, $w(e)$ depends on the orientation, i.e., $w(e)\neq w(\bar{e})$. 
With this notation, condition (ii) is equivalent to 
\begin{itemize}
\item[(ii$'$)] 
$$
\sum_{\substack{\tilde{e}\in E(\G\bs\sT) \\ t(\tilde{e})=\tilde{v}}} w(\tilde{e})f'(\tilde{e})=0\quad 
\text{for all $\tilde{v}\in V(\G\bs\sT)$},
$$
\end{itemize}
cf. \cite[(3.1)]{GR}. 

\begin{defn}\label{defnHarmG0} The group of $R$-valued \textit{cuspidal 
harmonic cochains} for $\G$, denoted $\cH_0(\sT, R)^\G$, is the 
subgroup of $\cH(\sT, R)^\G$ consisting of functions 
which have compact support as functions on $\G\bs\sT$, i.e., functions 
which have value $0$ on all but finitely many edges of $\G\bs\sT$.  
Let $\cH_{00}(\sT, R)^\G$ 
denote the image of $\cH_0(\sT, \Z)^\G\otimes R$ in $\cH_0(\sT, R)^\G$. 
\end{defn}

\begin{defn}\label{defnQG} It is known that the quotient 
graph $\G_0(\fn)\bs \sT$ is the edge disjoint union 
$$
\G_0(\fn)\bs \sT = (\G_0(\fn)\bs \sT)^0\cup \bigcup_{s\in \G_0(\fn)\bs \p^1(F)} h_s
$$
of a finite graph $(\G_0(\fn)\bs \sT)^0$ with a finite number of half-lines $h_s$, called \textit{cusps}; 
cf. Theorem 2 on page 106 of \cite{SerreT}. 
The cusps are in bijection with the orbits of the natural action of $\G_0(\fn)$ on $\p^1(F)$; cf. Remark 2 
on page 110 of \cite{SerreT}.
\end{defn}

To simplify the notation, we put 
$$
\cH(\fn, R):=\cH(\sT, R)^{\G_0(\fn)}$$ 
$$
\cH_0(\fn, R):=\cH_0(\sT, R)^{\G_0(\fn)} 
$$
$$
\cH_{00}(\fn, R)\text{ the image of }\cH_0(\fn, \Z)\otimes R \text{ in } \cH_0(\fn, R). 
$$
One can show that $\cH_0(\fn, \Z)$ and $\cH(\fn, \Z)$ are finitely generated free 
$\Z$-modules of rank $g(\fn)$ and  $g(\fn)+c(\fn)-1$, respectively, where $g(\fn)$ 
is the genus of $X_0(\fn)$ and $c(\fn)$ is the number of cusps. 

From the above description it is clear that $f$ is in $\cH_0(\fn, R)$ if and only if 
it eventually vanishes on each $h_s$. It is also clear that if $R$ is flat over $\Z$, then $\cH_0(\fn, R)=\cH_{00}(\fn, R)$. 
On the other hand, it is easy to construct examples where this equality does not hold. 

\begin{example}\label{example1} The quotient graph $\GL_2(A)\bs \sT$ is a half-line depicted in Figure \ref{Fig1}. 
\begin{figure}
\begin{tikzpicture}[->, >=stealth', semithick, node distance=1.5cm, inner sep=.5mm, vertex/.style={circle, fill=black}]

\node[vertex] (0) [label=below:$v_0$]{};
  \node[vertex] (1) [right of=0, label=below:$v_1$] {}; 
  \node[vertex] (2) [right of=1, label=below:$v_2$] {};
  \node[vertex] (3) [right of=2, label=below:$v_3$] {};
  \node[] (4) [right of=3] {};

\path[]
    (0) edge  (1) (1) edge (2) (2) edge (3) (3) edge[dashed] (4);   
\end{tikzpicture}
\caption{$\GL_2(A)\bs \sT$}\label{Fig1}
\end{figure}
Denote the edge with origin $v_i$ and terminus $v_{i+1}$ by $e_i$. The stabilizers 
of vertices and edges of $\GL_2(A)\bs \sT$ are well-known, cf. \cite[p. 691]{GN}. 
From this one computes $w(e_i) = q$ for all $i$, $w(\bar{e}_0)=q+1$, and $w(\bar{e}_i)=1$ for $i\geq 1$. 
Therefore, if $\varphi\in \cH(1, R)$, then 
$\varphi(e_i)=q^i\alpha$ $(i\geq 0)$ for some fixed $\alpha\in R[q+1]$. Now it is clear that 
$\cH(1,R)=R[q+1]$ and $\cH_0(1,R)=\cH_{00}(1,R)=0$. 
\end{example} 

\begin{example}\label{examplex}
\begin{figure}
\begin{tikzpicture}[->, >=stealth', semithick, node distance=1.5cm, inner sep=.5mm, vertex/.style={circle, fill=black}]

\node[vertex] (0) [label=above:$v_0$] {};
  \node[vertex] (1) [right of=0, label=above:$v_1$] {};
  \node[vertex] (2) [right of=1, label=above:$v_2$] {};
  \node[vertex] (3) [right of=2, label=above:$v_3$] {};
  \node[] (4) [right of=3] {};
   \node[vertex] (-1) [below of=1, label=below:$v_{-1}$] {};
  \node[vertex] (-2) [below of=2, label=below:$v_{-2}$] {};
  \node[vertex] (-3) [below of=3, label=below:$v_{-3}$] {};
  \node[] (-4) [below of=4] {};
  
  \path[]
  (-4) edge[dashed]  (-3) (-3) edge (-2) (-2) edge (-1) (-1) edge (0) (0) edge (1) (1)edge (2) (2)edge (3) (3)edge[dashed] (4);   
  
\end{tikzpicture}
\caption{$\G_0(x)\bs \sT$}\label{Fig2}
\end{figure}
The graph of $\G_0(x)\bs\cT$ is given in Figure \ref{Fig2}, where the vertex $v_i$ ($i\in \Z$) 
is the image of $\begin{pmatrix} T^{i} & 0\\ 0 & 1\end{pmatrix}\in V(\sT)$; the positive orientation 
is induced from $E(\sT)^+$.  
Denote by $e_i$ the edge with origin $v_{i-1}$ and terminus $v_{i}$. Since 
$\begin{pmatrix} 0 & 1\\ 1 & 0\end{pmatrix}v_{-i}=v_i$ and the stabilizers of $v_i$ ($i\geq 0$) in 
$\GL_2(A)$ are well-known (cf. \cite[p. 691]{GN}), one easily computes  
$$
w(e_i) = 
\begin{cases}
q & \text{if $i\geq 0$}\\
1 & \text{if $i\leq -1$}
\end{cases}
\qquad 
w(\bar{e}_i) = 
\begin{cases}
1 & \text{if $i\geq -1$}\\
q & \text{if $i\leq -2$}
\end{cases}
$$
Suppose $\varphi\in \cH(x, R)$ and denote $\alpha=\varphi(e_{-1})$. 
Since $w(e_i)\varphi(e_i)=w(\bar{e}_{i+1})\varphi(e_{i+1})$, we get 
$$
\varphi(e_i) = 
\begin{cases}
\alpha q^{i+1} & \text{if $i\geq -1$}\\
\alpha & \text{if $i=-2$}\\
\alpha q^{-i-3} & \text{if $i\leq -3$}. 
\end{cases}
$$
We conclude that $\cH(x, R)=R$, $\cH_0(x, R)=R_p$, and $\cH_{00}(x, R)=0$. (Recall that 
$R_p$ denotes the $p$-primary subgroup of $R$.)
\end{example}

\begin{figure}
\begin{tikzpicture}[->, >=stealth', semithick, node distance=1.5cm, inner sep=.5mm, vertex/.style={circle, fill=black}]

\node[vertex] (0) [label=above:$v_0$] {};
  \node[vertex] (1) [right of=0, label=above:$v_1$] {};
  \node[vertex] (2) [right of=1, label=above:$v_2$] {};
  \node[vertex] (3) [right of=2, label=above:$v_3$] {};
  \node[] (4) [right of=3] {};
   \node[vertex] (-1) [below of=1, label=below:$v_{-1}$] {};
  \node[vertex] (-2) [below of=2, label=below:$v_{-2}$] {};
  \node[vertex] (-3) [below of=3, label=below:$v_{-3}$] {};
  \node[vertex] (u) [below of=0, label=below:$u$] {};
  \node[] (-4) [below of=4] {};
  
  \path[]
  (-4) edge[dashed]  (-3) (-3) edge (-2) (-2) edge (-1) (-1) edge (0) (0) edge (1) (1)edge (2) (2)edge (3) (3)edge[dashed] (4) (-1) edge (u);   
  
\end{tikzpicture}
\caption{$\G_0(y)\bs \sT$}\label{Fig3}
\end{figure}
\begin{example}\label{exampley}

The graph $\G_0(y)\bs \sT$ 
is given in Figure \ref{Fig3}, where $v_i$ is the image of $\begin{pmatrix} T^{i} & 0\\ 0 & 1\end{pmatrix}\in V(\sT)$ and 
$u$ is the image of $\begin{pmatrix} T^{-2} & T^{-1} \\ 0 & 1\end{pmatrix}$.  
We denote the edge with 
origin $v_{i-1}$ and terminus $v_{i}$ by $e_i$, and the edge with terminus $u$ by $e_u$.  
One computes 
$$
w(e_i) = 
\begin{cases}
q & \text{if $i\geq 0$}\\
1 & \text{if $i\leq -1$}
\end{cases}
\qquad 
w(\bar{e}_i) = 
\begin{cases}
1 & \text{if $i\geq 0$}\\
q & \text{if $i\leq -1$}
\end{cases}
$$
$$
w(e_u)=q+1, \qquad w(\bar{e}_u)=q-1. 
$$
Let $\varphi\in \cH(y, R)$. Denote $\varphi(e_{0})=\alpha$ and $\varphi(e_u)=\beta$. Then $(q+1)\beta=0$ and 
$$
\varphi(e_i) = 
\begin{cases}
\alpha q^{i} & \text{if $i\geq 0$}\\
q^{-i-1}(\alpha+(q-1)\beta) & \text{if $i\leq -1$}. 
\end{cases}
$$
This implies that $\cH(y, R)\cong R\oplus R[q+1]$. 
For $\varphi$ to be cuspidal we must have $q^n \alpha=0$ and $q^n(q-1)\beta=0$ for some $n\geq 1$. 
Thus, $\alpha\in R_p$ and $\beta\in R[2]$ (resp. $\beta=0$) if $p$ is odd (resp. $2$). We 
get an isomorphism $\cH_0(y, R)\cong R_p\oplus R[2]$ if $p$ is odd and 
$\cH_0(y, R)\cong R_2$ if $p=2$. Note that $\cH_{00}(y, R)=0$. 
\end{example}

\begin{lem}\label{lem1.3} The following holds:
\begin{enumerate}
\item If $\fn\lhd A$ has a prime divisor of odd degree, assume 
$q(q-1)\in R^\times$. Otherwise, assume $q(q^2-1)\in R^\times$. Then 
$\cH_0(\fn, R)=\cH_{00}(\fn, R)$.
\item If $\fn=\fp$ is prime and $q(q-1)\in R^\times$, then $\cH_0(\fn, R)=\cH_{00}(\fn, R)$.
\end{enumerate} 
\end{lem}
\begin{proof} Our proof relies on the results in \cite{GN}, and is partly motivated by the proof  of Theorem 3.3 in \textit{loc. cit.} 
Let $\G:=\G_0(\fn)$. By 1.11 and 2.10 in \cite{GN}, 
the stabilizer $\G_v$ for any $v\in V(\sT)$ is finite, contains the scalar matrices $Z(\F_q)$, and $n(v):=\#\G_v/\F_q^\times$ 
either divides $(q-1)q^m$ for some $m\geq 0$, or is equal to $q+1$. 
Moreover, $n(v)=q+1$ is possible only if all prime divisors of $\fn$ have even degrees. 
Overall, we see that our assumptions in (1) imply that $n(v)$ is invertible in $R$ for any 
$v\in V(\sT)$. Since the stabilizer $\G_e$ of any $e\in V(\sT)$ is a subgroup of $\G_{t(e)}$ containing $Z(\F_q)$, 
we also have $n(e):=\#\G_e/\F_q^\times\in R^\times$. Note that $n(e)$ does not depend on the orientation of 
$e$ and depends only on its image $\tilde{e}$ 
in $\G\bs\sT$, so we can define $n(\tilde{e})=n(e)$.  

Let $\cH_0(\G\bs\sT,R)$ be the subgroup of $\cH(\G\bs\sT,R)$ consisting of compactly supported 
harmonic cochains on $\G \bs\sT$. 
There is an injective homomorphism 
\begin{align}\label{eqHomHarm}
\cH_0(\G\bs\sT, R) &\to \cH_0(\fn, R)\\
\nonumber\varphi &\mapsto \varphi^\dag
\end{align}
defined by $\varphi^\dag(\tilde{e})=n(\tilde{e})\varphi(\tilde{e})$. Indeed, since $n(\tilde{e})$ does not depend on the 
orientation of $e$, $\varphi^\dag$ clearly satisfies (i$'$). As for (ii$'$), we have 
$$
\sum_{\substack{\tilde{e}\in E(\G\bs\sT) \\ t(\tilde{e})=\tilde{v}}} w(\tilde{e})\varphi^\dag(\tilde{e})
=\sum_{\substack{\tilde{e}\in E(\G\bs\sT) \\ t(\tilde{e})=\tilde{v}}} \frac{n(\tilde{v})}{n(\tilde{e})}n(\tilde{e})\varphi(\tilde{e})
=n(\tilde{v}) \sum_{\substack{\tilde{e}\in E(\G\bs\sT) \\ t(\tilde{e})=\tilde{v}}} \varphi(\tilde{e}) = 0. 
$$

The map (\ref{eqHomHarm}) is also defined over $\Z$, and by \cite[Thm. 3.3]{GN} gives an isomorphism 
$\cH_0(\G\bs\sT, \Z) \overset{\sim}{\To} \cH_0(\fn, \Z)$. Next, there is an isomorphism 
$$\cH_0(\G\bs\sT, R)\cong \cH_0(\G\bs\sT, \Z)\otimes_\Z R,$$ 
which follows, for example, by observing that $H_1(\G\bs\sT, R)\cong \cH_0(\G\bs\sT, R)$ and applying the 
universal coefficient theorem for simplicial homology.  Hence 
$$
\cH_0(\G\bs\sT, R)\cong \cH_0(\G\bs\sT, \Z)\otimes_\Z R \cong \cH_0(\fn, \Z)\otimes_\Z R. 
$$

Let $g=\rank_\Z \cH_0(\G\bs\sT, \Z)$. Thinking of 
the elements of $\cH_0(\G\bs\sT, \Z)$ as $1$-cycles, it is easy to show by induction on $g$ 
that one can choose $e_1, \dots, e_g\in E(\G\bs\sT)$ and a $\Z$-basis $\varphi_1, \dots, \varphi_g$ of 
$\cH_0(\G\bs\sT, \Z)$ such that $\G\bs\sT - \{e_1, \dots, e_g\}$ is a tree, and  
$\varphi_i(e_j)=\delta_{ij}=$(Kronecker's delta), $1\leq i, j\leq g$. By slight abuse of notation, denote 
the image of  $\varphi_i^\dag$ in $\cH_{00}(\fn, R)$ by the same symbol. 
Let $\psi\in \cH_0(\fn, R)$. Then 
$$
\psi':=\psi-\sum_{i=1}^g \frac{\psi(e_i)}{n(e_i)} \varphi_i^\dag
$$
is supported on a finite subtree $S$ of $\G\bs \sT$. Let $v\in V(S)$ be a vertex such that 
there is a unique $e\in E(S)$ with $t(e)=v$. Note that $w(e)\in R^\times$. Condition (ii$'$) gives $w(e)\psi'(e)=0$, 
so $\psi'(e)=0$. This process can be iterated to show that $\psi'=0$. This implies 
that the natural map $\cH_0(\fn, \Z)\otimes_\Z R\to \cH_0(\fn, R)$ is surjective, which is part (1). 

To prove part (2), we can assume that $\deg(\fp)$ is even. 
A consequence of 2.7 and 2.8 in \cite{GN} is that 
there is a unique $v_0\in V(\G\bs \sT)$ 
with $n(v_0)=q+1$ and a unique $e_0\in E(\G\bs \sT)$ with $o(e_0)=v_0$. 
For any 
other $v\in V(\G\bs \sT)$, $n(v)$ divides $(q-1)q^m$. 
Since the stabilizer of any edge $e\in E(\G\bs\sT)$ is a subgroup of the stabilizers of both $t(e)$ and $o(e)$, we 
have $n(e)\in R^\times$. After this observation, we can repeat the argument used to prove (1) 
to reduces the problem to showing that $\psi\in \cH_0(\fp, R)$ supported on a finite tree $S$ 
is identically $0$. We can always choose $v\in V(S)$ to be a vertex different from $v_0$ but such that 
there is a unique $e\in E(S)$ with $t(e)=v$. Since $w(e)$ is a unit in $R$, we can also 
finish as in part (1). 
\end{proof}


The conclusion in Example \ref{exampley} that $\cH_0(y, R)\neq \cH_{00}(y, R)$ if $R[2]\neq 0$ is a 
special case of a general fact: 

\begin{lem}\label{lemRem11.9}
Assume $p$ is odd and invertible in $R$. Let $\fp\lhd A$ be
prime of even degree. 
If $R[2]\neq 0$, then $\cH_0(\fp, R)\neq \cH_{00}(\fp, R)$. 
\end{lem}
\begin{proof} Let $\G:=\G_0(\fp)$. As in Lemma \ref{lem1.3}, let $v_0$ 
be the unique vertex of $\G\bs \sT$ with $n(v_0)=q+1$, and let $e_0\in E(\G\bs \sT)$ 
be the unique edge with $o(e_0)=v_0$. Note that 
$w(\bar{e}_0)=q+1$. As we already mentioned 
in the proof of Lemma \ref{lem1.3}, for any 
other vertex $v$ in $\G\bs \sT$, $n(v)$ divides $(q-1)q^m$. 
Moreover, it is easy to see, for example by   
case (a) of Lemma 2.7 in \cite{GN}, that there is at least one vertex $v$ such that $n(v)$ is divisible by $q-1$. 
Consider all the paths without backtracking connecting $v_0$ to such a vertex, and  
fix a path of shortest length $\{e_0, e_1, \dots, e_m\}$.  Then 
$w(\bar{e}_i)$ ($1\leq i\leq m$) 
is invertible in $R$, but $w(e_m)$ is divisible by $q-1$. 
For a fixed non-zero $\alpha\in R[2]$, define $f$ on $E(\G\bs \sT)$ by 
$f(e_0)=\alpha$, $f(e_i)=\frac{w(e_{i-1})}{w(\bar{e}_i)}f(e_{i-1})$ $(1\leq i\leq m)$, 
$f(\bar{e}_j)=f(e_j)$ $(0\leq j\leq m)$, and $f(e)=0$ for all other edges. 
It is easy to see that $f\in \cH_0(\fp, R)$. On the other hand, any function $\varphi\in \cH_0(\fp, \Z)$ 
must be zero on $e_0$, since condition (ii$'$) for $v_0$ gives $(q+1)\varphi(\bar{e}_0)=0$. 
Therefore, $f\not\in \cH_{00}(\fp, R)$. 
\end{proof}

\begin{rem} The fact stated in Lemma \ref{lemRem11.9} is deduced in \cite{Pal} 
by different (algebro-geometric) methods. Our combinatorial proof 
seems to answer the question in Remark 11.9 in \cite{Pal}. 
\end{rem}


\subsection{Hecke operators and Atkin-Lehner involutions} 
Assume $\fn\lhd A$ is fixed. Given a non-zero ideal $\fm\lhd A$, define 
an $R$-linear transformation of the space of $R$-valued functions on $E(\sT)$ by 
$$
f|T_\fm=\sum f|\begin{pmatrix} a & b \\ 0 & d\end{pmatrix},
$$
where $f|\gamma$ for $\gamma \in \GL_2(F_\infty)$ is defined in Section \ref{sec2.1},
and the above sum is over $a,b,d\in A$ such that $a,d$ are monic, $(ad)=\fm$, $(a)+\fn=A$, and $\deg(b)< \deg(d)$. 
This transformation is the \textit{$\fm$-th Hecke operator}. Following a common convention, 
for a prime divisor $\fp$ of $\fn$ we often write $U_\fp$ instead of $T_\fp$. 

\begin{prop}\label{propTrecur} The Hecke operators preserve the spaces $\cH(\fn, R)$ and $\cH_0(\fn, R)$, and satisfy the 
recursive formulas: 
\begin{align*}
T_{\fm\fm'}&= T_\fm T_{\fm'}\quad \text{if}\quad  \fm+\fm'=A,\\
T_{\fp^i} &= T_{\fp^{i-1}}T_\fp-|\fp|T_{\fp^{i-2}}\quad \text{if}\quad  \fp\nmid \fn, \\
T_{\fp^i} &= T_\fp^i\quad \text{if}\quad  \fp| \fn. 
\end{align*}
\end{prop}
\begin{proof} The group-theoretic proofs of the analogous statement for the Hecke operators 
acting on classical modular forms work also in this setting; cf. \cite[$\S$4.5]{Miyake}. 
\end{proof}

\begin{defn}\label{defHA}
Let $\T(\fn)$ be the commutative subalgebra of $\End_\Z(\cH_0(\fn, \Z))$ with the same unity element 
generated by all Hecke operators. Let $\T(\fn)^0$ to be the subalgebra of $\T(\fn)$ generated 
by the Hecke operators $T_\fm$ with $\fm$ coprime to $\fn$. 
\end{defn}

For every ideal $\fm$ dividing $\fn$ with 
$\mathrm{gcd}(\fm, \fn/\fm)=1$, let $W_\fm$ be any matrix in $\Mat_2(A)$ of the form 
\begin{equation}\label{ALmatrix}
\begin{pmatrix} a\fm & b \\ c\fn & d\fm \end{pmatrix} 
\end{equation}
such that $a,b,c,d, \in A$ and the ideal generated by $\det(W_\fm)$ in $A$ is $\fm$. 
It is not hard to check that for $f\in \cH(\fn, R)$, $f|W_\fm$ does not depend on the choice 
of the matrix for $W_\fm$ and  $f|W_\fm\in \cH(\fn, R)$. Moreover, 
as $R$-linear endomorphisms of $\cH(\fn, R)$, $W_\fm$'s satisfy 
\begin{equation}\label{eqWs}
W_{\fm_1}W_{\fm_2}=W_{\fm_3}, \quad \text{where} \quad \fm_3=
\frac{\fm_1\fm_2}{\mathrm{gcd}(\fm_1, \fm_2)^2}. 
\end{equation}
Therefore, the matrices $W_\fm$ acting on the $R$-module $\cH(\fn, R)$ 
generate an abelian group $\W\cong (\Z/2\Z)^s$, called the group of \textit{Atkin-Lehner involutions}, 
where $s$ is the number of prime divisors of $\fn$. The following proposition, whose 
proof we omit, follows from calculations similar to those in \cite[$\S$2]{AL}. 

\begin{prop}\label{prop3lem} Let 
$$
B_\fm=\begin{pmatrix} \fm & 0 \\  0 & 1\end{pmatrix}. 
$$
\begin{enumerate}
\item
If $\fn$ is coprime to $\fm$ and $f\in \cH(\fn, R)$, then 
$$(f|B_\fm)|W_\fm=f,$$ 
where $W_\fm$ is the Atkin-Lehner involution acting on $\cH(\fn\fm, R)$. $($Note 
that by Lemma \ref{lemAL2}, $f|B_\fm\in \cH(\fn\fm, R)$.$)$
\item Let $\fm|\fn$ with $\gcd(\fm, \fn/\fm)=1$, and 
$\fb$ be coprime to $\fm$. If $f\in \cH(\fn, R)$, then 
$$
(f|B_\fb)|W_\fm=(f|W_\fm)|B_\fb,
$$ 
where on the left hand-side $W_\fm$ denotes the Atkin-Lehner involution 
acting on $\cH(\fn\fb, R)$ and on the right hand-side $W_\fm$ denotes the  
involution acting on $\cH(\fn, R)$. 
\item Let $f\in \cH(\fn, R)$. If $\fq$ is a prime ideal which divides $\fn$ 
but does not divide $\fn/\fq$, then $f|(U_\fq+W_\fq)\in \cH(\fn/\fq, R)$. 
\end{enumerate}
\end{prop}

The vector space $\cH_0(\fn, \Q)$ is equipped with a natural (Petersson) inner product 
$$
\langle f, g\rangle =\sum_{e\in E(\G_0(\fn)\bs \sT)} n(e)^{-1} f(e)g(e), 
$$
where $n(e)$ is defined in the proof of Lemma \ref{lem1.3}. The Hecke operator 
$T_\fm$ is self-adjoint with respect to this inner product if $\fm$ is coprime to $\fn$; 
one can prove this by an argument similar to the proof of Lemma 13 in \cite{AL}.  

\begin{defn}\label{defnNewH} Let $\fm$ be a divisor of $\fn$ and $\fd$ be a divisor of $\fn/\fm$. By Lemma \ref{lemAL2}, 
the map $\varphi \mapsto \varphi| B_\fd$ gives an injective homomorphism 
$$
i_{\fd, \fm}: \cH_0(\fm, \Q)  \to \cH_0(\fn, \Q). 
$$
We denote the subspace generated by the images of all $i_{\fd, \fm}$ ($\fm\neq \fn$) by $\cH_0(\fn, \Q)^\old$. 
The orthogonal complement of $\cH_0(\fn, \Q)^\old$ with respect to the Petersson product is the \textit{new} subspace 
of $\cH_0(\fn, \Q)$, and will be denoted by $\cH_0(\fn, \Q)^\new$. 
The new subspace of $\cH_0(\fn, \Q)$ is invariant under the action $\T(\fn)$ (this again can be proven as in \cite{AL}). 
We denote by $\T(\fn)^\new$ the quotient of $\T(\fn)$ through which $\T(\fn)$ acts on $\cH_0(\fn, \Q)^\new$. 
\end{defn}

As we mentioned, the cusps of $\G_0(\fn)$ are in bijection with the orbits 
of the action of $\G_0(\fn)$ on 
$$
\p^1(F)=\p^1(A)=\left\{\begin{pmatrix} a \\ b\end{pmatrix}\ \big|\ a, b\in A, \gcd(a, b)=1, a \text{ is monic}\right\}, 
$$
where $\G_0(\fn)$ acts on $\p^1(F)$ from the left as on column vectors. We leave the proof 
of the following lemma to the reader. 

\begin{lem}\label{lemCusps} Assume $\fn$ is square-free. 
\begin{enumerate}
\item 
For $\fm|\fn$ let $[\fm]$ 
be the orbit of $\begin{pmatrix} 1 \\ \fm \end{pmatrix}$ under the action of $\G_0(\fn)$. 
Then $[\fm]\neq [\fm']$ if $\fm\neq \fm'$, and the set $\{[\fm]\ |\ \fm|\fn\}$ 
is the set of cusps of $\G_0(\fn)$. In particular, there are $2^s$ cusps, where $s$ 
is the number of prime divisors of $\fn$. 
\item Since $W_\fm$ normalizes $\G_0(\fn)$, 
it acts on the set of cusps of $\G_0(\fn)$. There is the formula 
$$
W_\fm [\fn] = [\fn/\fm]. 
$$
\end{enumerate}
\end{lem}

The cusp $[\fn]$ is usually called the \textit{cusp at infinity}. We will denote it by $[\infty]$.  


\subsection{Fourier expansion} 
An important observation in \cite{Pal} is that the theory of Fourier expansions 
of automorphic forms over function fields developed in \cite{Weil} works over more general rings than 
$\C$. Here we follow Gekeler's reinterpretation \cite{Improper} of  
Weil's adelic approach as analysis on the Bruhat-Tits tree, but we will extend \cite{Improper} to the 
setting of these more general rings. 

\begin{defn}
Following \cite{Pal} we say that $R$ is a \textit{coefficient ring} if $p\in R^\times$
and $R$ is a quotient of a discrete valuation ring $\tilde{R}$ which contains $p$-th roots of unity. Note 
that the image of the $p$-th roots of unity of $\tilde{R}$ in $R$ is exactly the set of $p$-th roots 
of unity of $R$. For example, any algebraically closed field of characteristic different from $p$ 
is a coefficient ring. 
\end{defn}

Let 
\begin{align*}
\eta: \Fi &\to R^\times \\ 
\sum a_i\pi_\infty^i &\mapsto \eta_0\Big(\text{Trace}_{\F_q/\F_p}(a_1)\Big)
\end{align*}
where $\eta_0: \F_p\to R^\times$ is a non-trivial additive character fixed once and for all.   
Let $f$ be an $R$-valued function on $E(\sT)$, which is invariant under the 
action of 
$$
\G_\infty:=\left\{\begin{pmatrix} a & b \\ 0 & d\end{pmatrix}\in \GL_2(A)\right\}, 
$$
and is alternating (i.e., satisfies $f(e)=-f(\bar{e})$ for all $e\in E(\sT)$). 
The \textit{constant Fourier coefficient} of $f$ is the $R$-valued function $f^0$ on $\pi_\infty^\Z$ defined by 
$$
f^0(\pi_\infty^k)= \begin{cases}
q^{1-k} \sum_{\substack{u\in (\pi_\infty)/(\pi_\infty^k)}}f\left(\begin{pmatrix}\pi_\infty^k & u \\ 0 &1\end{pmatrix}\right) 
& \text{if $k\geq 1$} \\ 
f\left(\begin{pmatrix}\pi_\infty^k & 0 \\ 0 &1\end{pmatrix}\right)& \text{if $k\leq 1$}.  
\end{cases}
$$
For a divisor $\fm$ on $F$, the \textit{$\fm$-th Fourier coefficient} $f^\ast(\fm)$ of $f$ is 
$$
f^\ast(\fm) = q^{-1-\deg(\fm)}\sum_{u\in (\pi_\infty)/(\pi_\infty^{2+\deg(\fm)})} 
f\left(\begin{pmatrix}\pi_\infty^{2+\deg(\fm)} & u \\ 0 &1\end{pmatrix}\right)\eta(-m u), 
$$ 
if $\fm$ is non-negative, and $f^\ast(\fm)=0$, otherwise; here $m\in A$ is the monic polynomial 
such that $\fm=\mathrm{div}(m)\cdot \infty^{\deg(\fm)}$. 

\begin{thm}
Let $f$ be an $R$-valued function on $E(\sT)$, which is 
$\G_\infty$-invariant and alternating.
Then
$$
f \left(\begin{pmatrix} \pi_\infty^k & y \\ 0 &1\end{pmatrix}\right) = f^0(\pi_\infty^k)+ 
\sum_{\substack{0\neq m\in A \\ \deg(m)\leq k-2}} f^\ast(\mathrm{div}(m)\cdot \infty^{k-2})\cdot \eta(my). 
$$
In particular, $f$ is uniquely determined by the functions $f^0$ and $f^\ast$. 
\end{thm}
\begin{proof} This follows from \cite[$\S$2]{Pal} and \cite[$\S$2]{Improper}. 
\end{proof}

\begin{lem}\label{lemFH} Assume $f$ is alternating and $\G_\infty$-invariant. Then $f$ is 
a harmonic cochain if and only if 
\begin{itemize}
\item[(i)] 
$f^0(\pi_\infty^k)=f^0(1)\cdot q^{-k}$ for any $k\in \Z$; 
\item[(ii)] $
f^\ast(\fm \infty^k)=f^\ast(\fm)\cdot q^{-k}$ for any non-negative divisor $\fm$ and $k \in \Z_{\geq 0}$. 
\end{itemize}
\end{lem}
\begin{proof}
See Lemma 2.13 in \cite{Improper}. 
\end{proof}

\begin{lem}\label{thm3.3} For an ideal $\fm\lhd A$ and $f\in \cH(\fn, \Z)$ we have  
$$
(f|T_\fm)^\ast(\fr) = \sum_{\substack{a\ \mathrm{monic}\\ a|\gcd(\fm, \fr)\\  (a)+\fn=A}}\frac{|\fm|}{|a|}\cdot f^\ast\left(\frac{\fr \fm}{a^2}\right). 
$$
In particular, 
$$
(f|T_\fm)^\ast(1)=|\fm|f^\ast(\fm). 
$$
\end{lem}
\begin{proof}
See Lemma 3.2 in \cite{Pal}. 
\end{proof}

\begin{lem}\label{lem2.12}
Assume $\fn$ is square-free. A harmonic cochain $f\in \cH(\fn, R)$ is cuspidal if and only if 
$(f|W)^0(1)=0$ for all $W\in \W$. 
\end{lem}
\begin{proof}
By definition, $f$ is cuspidal if and only if it vanishes on all but finitely many edges 
of each cusp $[\fm]$. The positively oriented edges of the cusp $[\infty]$ are given by 
the matrices $\begin{pmatrix}\pi_\infty^k & 0 \\ 0 & 1 \end{pmatrix}$, $k\leq 1$. By definition 
of $f^0$ and Lemma \ref{lemFH}, 
$$
f\left(\begin{pmatrix}\pi_\infty^k & 0 \\ 0 &1\end{pmatrix}\right)=f^0(\pi_\infty^k) =q^{-k}f^0(1). 
$$
Since $q$ is invertible in $R$, we see that $f$ eventually vanishes on $[\infty]$ if and only if $f^0(1)=0$. 
Next, by Lemma \ref{lemCusps}, $f$ vanishes on $[\fn/\fm]$ if and only if $f|W_\fm$ 
vanishes on $[\infty]$, which is equivalent to $(f|W_\fm)^0(1)=0$. 
\end{proof}

\begin{thm}\label{thm2.14}
If $R$ is a coefficient ring, then the bilinear $\T(\fn)\otimes R$-equivariant pairing 
\begin{align*}
(\T(\fn)\otimes R)\times \cH_{00}(\fn, R) \to R\\
T, f\mapsto (f|T)^\ast(1)
\end{align*}
is perfect. 
\end{thm}

\begin{proof} Theorem 3.17 in \cite{Analytical} says that the pairing 
\begin{align}\label{GPairing}
\T(\fn)\times \cH_{0}(\fn, \Z) \to \Z\\
\nonumber T, f\mapsto (f|T)^\ast(1)
\end{align}
is non-degenerate and becomes a perfect pairing after tensoring with $\Z[p^{-1}]$. Since $p$ 
is invertible in $R$ by assumption, the claim follows. 
\end{proof}

It is not known if in general the pairing (\ref{GPairing}) is perfect. 
This is in contrast to the situation over $\Q$ where the analogous pairing 
between the Hecke algebra and the space of weight-$2$ cusp forms on $\G_0(N)$ with 
integral Fourier expansions is perfect (cf. \cite[Thm. 2.2]{RibetModp}). 
This dichotomy comes from the formula $(f|T_\fm)^\ast(1)=|\fm|f^\ast(\fm)$;  
in the classical situation the first Fourier coefficient of $f|T_m$ is just the $m$th Fourier 
coefficient of $f$.

\begin{prop}\label{propT(xy)} In the special case $\fn=xy$, 
the pairing (\ref{GPairing}) 
$$
\T(xy)\times \cH_0(xy, \Z)\to \Z 
$$
is perfect. 
Moreover, as $\Z$-modules, 
$$
\T(xy)^0=\T(xy)\cong \Z\oplus \bigoplus_{\substack{\deg(\fp)=1\\ \fp\neq x}}\Z T_{\fp}.
$$ 
\end{prop}

\begin{proof}
\begin{figure}
\begin{tikzpicture}[->, >=stealth', semithick, node distance=1.5cm, inner sep=.5mm, vertex/.style={circle, fill=black}]

\node (1){$[x]$};
  \node[vertex] (2) [below right of=1] {};
  \node[vertex] (3) [below right of=2] {};
  \node[vertex] (4) [below of=3] {};
  \node[vertex] (5) [below of=4] {};
  \node[vertex] (6) [below left of=5] {};
  \node (7) [below left of=6] {$[\infty]$};
  \node (15) [right of=3] {}; 
  \node[vertex] (8) [right of=15] {};
  \node[vertex] (9) [above right of=8] {};
  \node (10) [above right of=9] {$[y]$};
  \node[vertex] (11) [below of=8] {};
  \node[vertex] (12) [below of=11] {};
  \node[vertex] (13) [below right of=12] {};
  \node (14) [below right of=13] {$[1]$};
  \node (16) [below of=15] {$\vdots\ b_u$};
  
  \path[]
  (1) edge[dashed]  (2) 
  (2) edge node [right] {$c_3$} (3)
   (3) edge node [left] {$a_2=b_0$} (4)
   (4) edge node [left] {$a_1$} (5)
   (5) edge node [left] {$c_1$} (6)
   (6) edge[dashed]  (7) 
    (10) edge[dashed]  (9) 
    (9) edge node [left] {$c_4$} (8)
   (8) edge node [right] {$a_3$} (11)
   (11) edge node [right] {$a_4$} (12)
   (12) edge node [below] {$a_5$} (5)
   (8) edge node [above] {$a_6$} (3)
   (14) edge[dashed]  (13)
  (13) edge node [right] {$c_2$} (12)
  (11) edge[bend right]  (4)
  (11) edge[bend left]  (4);
\end{tikzpicture}
\caption{$\G_0(xy)\bs \sT$}\label{Fig4}
\end{figure}
Take $\alpha_x, \ \beta_x \in \F_q$ such that $y = x^2+\alpha_x x + \beta_x$.
Let $\varpi_x:= x^{-1}$, which is also a uniformizer at $\infty$. 
The quotient graph $\Gamma_0(xy)\backslash \sT$ is depicted in Figure \ref{Fig4} 
with positively oriented edges 
$$c_1 = \begin{pmatrix} \varpi_x&0\\0&1\end{pmatrix}, \quad c_2 = \begin{pmatrix}\varpi_x^3&0\\0&1\end{pmatrix},
\quad c_3 = \begin{pmatrix} \varpi_x^4 & \varpi_x \\ 0&1\end{pmatrix}, \quad 
c_4 = \begin{pmatrix} \varpi_x^5 & y^{-1}\\0&1\end{pmatrix};$$
$$a_1 = \begin{pmatrix} \varpi_x^2 & \varpi_x \\ 0&1\end{pmatrix}, \quad a_2 
= \begin{pmatrix} \varpi_x^3 & \varpi_x \\ 0&1\end{pmatrix}, \quad a_3 
= \begin{pmatrix} \varpi_x^4 & y^{-1} \\ 0&1\end{pmatrix}, \quad a_4 = \begin{pmatrix} \varpi_x^3 & \varpi_x^2 \\0&1\end{pmatrix};$$
$$a_5 = \begin{pmatrix} \varpi_x^2 & 0 \\ 0&1\end{pmatrix}, \quad a_6 
= \begin{pmatrix} \varpi_x^4 & \varpi_x - \beta_x \varpi_x^3 \\0&1\end{pmatrix}; 
$$
$$
b_u 
= \begin{pmatrix} \varpi_x^3 & \varpi_x + u \varpi_x^2 \\0&1\end{pmatrix}, \quad u \in \F_q.
$$
Note that in this notation $a_2=b_0$.  A small calculation shows that 
$$w(a_1) = w(\bar{a}_2) = w(\bar{a}_3) = w(a_4) = q-1,$$
and the weights of all other edges in $(\G_0(xy)\bs \sT)^0$ are $1$. 

It is easy to see that the map 
\begin{align*}
\cH_0(xy, \Z) &\to \bigoplus_{u\in \F_q}\Z\\
f &\mapsto (f(b_u))_{u\in \F_q}
\end{align*}
is an isomorphism, so the harmonic cochains $f_v\in \cH_0(xy, \Z)$, $v\in \F_q$, defined 
by $f_v(b_u)=\delta_{v,u}$=(Kronecker's delta) form a $\Z$-basis. 
Let $f\in \cH_0(xy, \Z)$ and $\kappa\in \F_q$. By Lemma \ref{thm3.3}
$$
q(f|T_{x-\kappa})^\ast(1) =q^2f^\ast(x-\kappa)=
\sum_{w\in \varpi_x\cO_\infty/\varpi_x^{3}\cO_\infty}
f\left(\begin{pmatrix} \varpi_x^3 & w\\ 0 & 1\end{pmatrix}\right)
\eta\left(-(x-\kappa)w\right) 
$$
\begin{align*}
= &f\left(\begin{pmatrix} \varpi_x^3 & 0\\ 0 & 1\end{pmatrix}\right) +
\sum_{\beta\in \F_q^\times} f\left(\begin{pmatrix} \varpi_x^3 & \beta \varpi_x^2\\ 0 & 1\end{pmatrix}\right)
\eta\left(-(\varpi_x^{-1}-\kappa)\beta \varpi_x^2\right) \\
& +\sum_{u\in \F_q} \sum_{\beta\in \F_q^\times} 
f\left(\begin{pmatrix} \varpi_x^3 & \beta(\varpi_x+u \varpi_x^2)\\ 0 & 1\end{pmatrix}\right)\eta\left(-(\varpi_x^{-1}-\kappa)\beta(\varpi_x+u \varpi_x^2)\right). 
\end{align*}
Since the double class of $\begin{pmatrix} \varpi_x^3 & w\\ 0 & 1\end{pmatrix}$ does not change if $w$ is replaced by 
$\beta w$ ($\beta\in \F_q^\times$), 
$f\left(\begin{pmatrix} \varpi_x^3 & 0\\ 0 & 1\end{pmatrix}\right)=f(c_2)=0$, and 
$\sum_{\beta\in \F_q^\times}\eta(\beta\varpi_x)=-1$, the above sum reduces to 
$$
 -f(a_4)+\sum_{u\in \F_q} f(b_u)(q\delta_{u, \kappa}-1). 
$$
Using (ii$'$), 
$$
(q-1)f(a_1)+f(a_5)=0, \quad (q-1)f(a_4)+f(\bar{a}_5)=0, \quad f(a_1)=\sum_{u\in \F_q}f(b_u). 
$$
Therefore, $f(a_4)=-\sum_{u\in \F_q}f(b_u)$ and we get 
$$
(f|T_{x-\kappa})^\ast(1)= f(b_\kappa). 
$$
In particular, $(f_v|T_{x-\kappa})^\ast(1)=\delta_{\kappa, v}$. This implies that the homomorphism 
\begin{equation}\label{eqTHZ}
\T(xy)\to  \Hom(\cH_0(xy, \Z), \Z)
\end{equation}
induced by the pairing (\ref{GPairing}) is surjective. Comparing the ranks of both sides, we conclude that this map is  
in fact an isomorphism, which is equivalent to the pairing being perfect. 
Let $M$ be the $\Z$-submodule of $\T(xy)$ generated by $\{T_{x-\kappa}\ |\ \kappa\in \F_q\}$. 
The composition of $M\hookrightarrow \T(xy)$ with (\ref{eqTHZ}) gives a surjection $M\to \Hom(\cH_0(xy, \Z), \Z)$. 
This implies that $M=\T(xy)$ and $M\cong \bigoplus_{\kappa\in \F_q}\Z T_{x-\kappa}$. 

An easy consequence of the definitions is that $f^\ast(1)=-f(a_1)$, cf. \cite[(3.16)]{Analytical}. 
If we denote $S=\sum_{\kappa\in \F_q} T_{x-\kappa}$, then 
\begin{equation}\label{eqS=-1}
(f|S)^\ast(1) =\sum_{\kappa\in \F_q} f(b_\kappa)=f(a_1)=-f^\ast(1).  
\end{equation}
The non-degeneracy of the pairing implies that $S=-1$. Therefore 
$$
\T(xy)=\Z\oplus \bigoplus_{\kappa\in \F_q^\times}\Z T_{x-\kappa}\subseteq \T(xy)^0,  
$$
which implies $\T(xy)=\T(xy)^0$. 
\end{proof}

\begin{rem} In \cite{PW2}, we have extended the statement of Proposition \ref{propT(xy)} 
to arbitrary $\fn \lhd A$ of degree $3$. More precisely, we proved that the pairing (\ref{GPairing})  
is perfect if $\deg(\fn)=3$. Moreover,  if $\fn$ has degree $3$ but is not a product of 
three distinct primes of degree $1$, then $\T(\fn)=\T(\fn)^0$.
Finally, if $\fn$ is a product of three distinct primes of degree $1$, then $\T(\fn)/\T(\fn)^0$ is finite but non-zero. 
\end{rem}

\subsection{Atkin-Lehner method} For $b\in A$, let $S_b=\begin{pmatrix} 1 & b \\ 0 & 1\end{pmatrix}$. 
Define a linear operator $U_\fp$ on the space of $R$-valued functions on $E(\sT)$ by 
$$
f|U_\fp=\sum_{\substack{b\in A\\ \deg(b)<\deg(\fp)}} f|B_\fp^{-1}S_b. 
$$
Note that the action of $B_\fm^{-1}$ on functions on $E(\sT)$ is the same as the action of the matrix 
$\begin{pmatrix} 1 & 0 \\  0 & \fm\end{pmatrix}$ (since the diagonal matrices act trivially), so this 
operator agrees with the Hecke operator $U_\fp$ when restricted to $\cH(\fn, R)$ for any $\fn$ divisible by $\fp$. 

\begin{lem}\label{lem2.15} Let $\fp$ and $\fq$ be two distinct prime ideals of $A$. 
If $f\in\cH(\sT, R)^{\G_\infty}$, then 
\begin{align*}
(f|B_\fp)|U_\fp &=|\fp|\cdot f,\\
(f|B_\fp)|U_\fq &= (f|U_\fq)|B_\fp. 
\end{align*}
\end{lem}
\begin{proof} 
We have 
$$
(f|B_\fp)|U_\fp= \sum_{\substack{b\in A\\ \deg(b)<\deg(\fp)}} (f|B_\fp)|B_\fp^{-1}S_b = 
\sum_{\substack{b\in A\\ \deg(b)<\deg(\fp)}} f|S_b.
$$
Since $S_b\in \G_\infty$, we have $f|S_b=f$ for all $b$, so the last sum is equal to $|\fp|f$. 
Next, for $b\in A$ representing a residue modulo $\fq$ we have 
$$
B_\fp B_\fq^{-1}S_b = \begin{pmatrix} \fp & b \fp \\ 0 & \fq \end{pmatrix}. 
$$ 
By the division algorithm there is $a\in A$ and $b'\in A$ with $\deg(b')<\deg(\fq)$ such that 
$b\fp= a\fq +b'$. Now 
$$
\begin{pmatrix} 1 & a \\ 0 & 1 \end{pmatrix}\begin{pmatrix} \fp & b \fp \\ 0 & \fq \end{pmatrix} = 
\begin{pmatrix} \fp & b' \\ 0 & \fq \end{pmatrix}=B_\fq^{-1}S_{b'} B_\fp. 
$$
As $b$ runs over the residues modulo $\fq$, $b'$ runs over the same set since $\fp\neq \fq$. Thus, using 
$\G_\infty$-invariance of $f$, we get $(f|B_\fp)|U_\fq = (f|U_\fq)|B_\fp$. 
\end{proof}

\begin{lem}\label{lem_new14}
For any non-zero ideal $\fm\lhd A$ and $f\in\cH(\sT, R)^{\G_\infty}$
$$
(f|B_\fm)^0(\pi_\infty^k)=f^0(\pi_\infty^{k-\deg(\fm)}), \quad (f|B_\fm)^\ast(\fn)=f^\ast(\fn/\fm). 
$$
\end{lem}
\begin{proof}
See Proposition 2.10 in \cite{Improper}. 
\end{proof}

Given ideals $\fn, \fm\lhd A$, denote 
$$
\G_0(\fn,\fm)=\left\{\begin{pmatrix} a & b \\  c & d\end{pmatrix}\in \GL_2(A)\ \big|\ c\in \fn, b\in \fm \right\}. 
$$

\begin{lem}\label{lemAL2} If $f\in \cH(\fn, R)$, then $f|B_\fm$ is $\G_0(\fn\fm)$-invariant and 
$f|B_\fm^{-1}$ is $\G_0(\fn/\gcd(\fn, \fm), \fm)$-invariant.  
\end{lem}
\begin{proof} This follows from a straightforward manipulation with matrices. 
\end{proof}

\begin{thm}\label{thmAL} Let $\fp$ and $\fq$ be two distinct primes such that $\fp\fq$ 
divides $\fn$, and $\fp\fq$ is coprime to $\fn/\fp\fq$.  
Let $\varphi\in \cH(\fn, R)$. Assume $\varphi^\ast(\fm)=0$ unless $\fp$ or $\fq$ divides $\fm$. Then there 
exist $\psi_1\in \cH(\fn/\fp, R)$ and $\psi_2\in \cH(\fn/\fq, R)$ such that 
$$
s_{\fp, \fq}\cdot \varphi=\psi_1|B_{\fp}+\psi_2|B_{\fq},
$$
where $s_{\fp, \fq}=\gcd(|\fp|+1, |\fq|+1)$. 
\end{thm}
\begin{proof} 
Take $\phi_2:= |q|^{-1} \cdot \varphi|U_{\fq} \in \cH(\fn,R)$. We have
$$
(\phi_2)^0(\pi_\infty^k) = \varphi^0(\pi_\infty^{k+\deg(\fq)}),\quad \phi_2^\ast(\fm)=\varphi^\ast(\fm\fq).
$$ 
Let $\varphi_1:= \varphi - \phi_2|B_{\fq} \in \cH(\fn\fq,R)$. Then by Lemma \ref{lem_new14},
$$
(\varphi_1)^0(\pi_\infty^k) =0 ,\quad \varphi_1^\ast(\fm)=\varphi^\ast(\fm) \text{ if } \fq\nmid \fm, \quad 
\varphi_1^\ast(\fm)=0 \text{ if }\fq|\fm.
$$
Let $\phi_1:= \varphi_1|B_\fp^{-1}$, which is $\G_0(\fn\fq/\fp,\fp)$-invariant by Lemma \ref{lemAL2}.
In particular, $\varphi_1^\ast(\fm) = 0$ unless $\fp| \fm$, which implies that $\phi_1$ is $\G_\infty$-invariant.
Since $\G_\infty$ and $\G_0(\fn\fq/\fp,\fp)$ generates $\G_0(\fn\fq/\fp)$, we get $\phi_1 \in \cH(\fn\fq/\fp,R)$ with
$$
(\phi_1)^0(\pi_\infty^k) =0 ,\quad \phi_1^\ast(\fm)=\varphi^\ast(\fm\fp) \text{ if } \fq\nmid \fm, \quad 
\phi_1^\ast(\fm)=0 \text{ if }\fq|\fm,
$$
and
$$\varphi = \phi_1|B_\fp + \phi_2|B_\fq.$$


By Proposition \ref{prop3lem}, $\psi_1:=\varphi|(U_\fp+W_\fp)\in \cH(\fn/\fp, R)$. 
Using Proposition \ref{prop3lem} and Lemma \ref{lem2.15}, 
$$
(\phi_1|B_\fp)|(U_\fp+W_\fp)=\phi_1|B_\fp|U_\fp+\phi_1|B_\fp|W_\fp=|\fp|\phi_1 + \phi_1=(|\fp|+1)\phi_1. 
$$
On the other hand, using the fact that $\phi_2\in \cH(\fn, R)$, we have 
$$
(\phi_2|B_\fq)|(U_\fp+W_\fp)=\phi_2|(U_\fp+W_\fp)|B_\fq. 
$$
If we denote $\psi:=\phi_2|(U_\fp+W_\fp)$, then we proved that 
$$
\psi_1=(|\fp|+1)\phi_1+\psi|B_\fq \in \cH(\fn/\fp, R). 
$$
Therefore,
$$
(|\fp|+1)\varphi = (|\fp|+1)\phi_1|B_\fp + (|\fp|+1)\phi_2|B_\fq
$$
$$
=((|\fp|+1)\phi_1+\psi|B_\fq)|B_\fp+((|\fp|+1)\phi_2 -\psi|B_\fp)|B_\fq=\psi_1|B_\fp+\psi_2|B_\fq,
$$
where $\psi_2:=(|\fp|+1)\phi_2 -\psi|B_\fp$. We already proved that $\psi_1\in \cH(\fn/\fp, R)$. 
Obviously $\psi_2|B_\fq\in \cH(\fn, R)$. 
By Lemma \ref{lemAL2}, $\psi_2$ is $\G_0(\fn/\fq, \fq)$-invariant. Since it is also $\G_\infty$-invariant, 
we conclude $\psi_2\in \cH(\fn/\fq, R)$. 

Finally, interchanging the roles of $\fp$ and $\fq$ we obtain $$(|\fq|+1)\varphi=\psi_1'|B_\fp+\psi_2'|B_\fq$$ 
with $\psi_1'\in \cH(\fn/\fp, R)$ and $\psi_2'\in \cH(\fn/\fq, R)$. This implies the claim of the theorem. 
\end{proof}


\section{Eisenstein harmonic cochains}\label{sec3}

\subsection{Eisenstein series}\label{ssES} In this section $R$ always denotes 
a coefficient ring, in particular, $p$ is invertible in $R$. 
We say that a harmonic cochain $\varphi \in \cH(\fn,R)$ is \textit{Eisenstein} 
if $ \varphi |T_{\fp} = (|\fp|+1)\varphi$ for every prime ideal $\fp \lhd A$ not dividing $\fn$. 
It is clear that the Eisenstein harmonic cochains form an $R$-submodule of $\cH(\fn, R)$ 
which we denote by $\cE(\fn, R)$. 

The Drinfeld half-plane 
$$\Omega=\p^1(\C_\infty)-\p^1(\Fi)=\C_\infty-\Fi$$ 
has a natural structure of a smooth connected rigid-analytic space 
over $\Fi$; see \cite[$\S$1]{GR}. 
The group $\G_0(\fn)$ acts on $\Omega$ via linear 
fractional transformations:
$$
\begin{pmatrix} a & b \\ c & d \end{pmatrix}z=\frac{az+b}{cz+d}. 
$$
This action is discrete, so the quotient 
\begin{equation}\label{eqUnifY}
Y_0(\fn)(\C_\infty) = \G_0(\fn)\bs \Omega
\end{equation}
has a natural structure of a rigid-analytic curve over $\Fi$, 
which is in fact an affine algebraic curve; cf. \cite[Prop. 6.6]{Drinfeld}. 
If we denote $\overline{\Omega}=\Omega\cup \p^1(F)$, then $$X_0(\fn)(\C_\infty)= \G_0(\fn)\bs \overline{\Omega}$$ 
is the projective closure of $Y_0(\fn)$. The points $X_0(\fn)(\C_\infty)-Y_0(\fn)(\C_\infty)$ 
are called the \textit{cusps} of $X_0(\fn)$, and they are in natural bijection with the cusps 
of $\G_0(\fn)\bs \sT$ in Definition \ref{defnQG}. 

The Hecke operator $T_\fp$ induces a correspondence on $X_0(\fn)(\C_\infty)$ 
\begin{equation}\label{eqHC}
T_\fp: z\mapsto \sum_{\substack{b\in A\\ \deg(b)<\deg(\fp)}} \frac{z+b}{\fp}+\fp z\quad \mod\ \G_0(\fn); 
\end{equation}
$U_\fp(z)$ is given by the same sum but without the last summand. 

Let $\cO(\Omega)^\times$ be the group of nowhere vanishing holomorphic functions on $\Omega$. 
The group $\GL_2(\Fi)$ act on $\cO(\Omega)^\times$ via $(f|\gamma)(z)=f(\gamma z)$. 
To each $f\in \cO(\Omega)^\times$ van der Put associated a 
harmonic cochain $r(f)\in \cH(\sT, \Z)$ so that the sequence  
\begin{equation}\label{eqvdPut}
0\to \C_\infty^\times \to \cO(\Omega)^\times\xrightarrow{r} \cH(\cT, \Z)\to 0
\end{equation}
is exact and $\GL_2(\Fi)$-equivariant. As is explained in \cite{GR}, the map 
$r$ plays the role of a logarithmic derivation. 

\begin{lem}\label{lemPlog}
Assume $\fn$ is square-free and $f\in \cO(\Omega)^\times$ is $\G_0(\fn)$-invariant. 
Then $r(f)$ is Eisenstein. 
\end{lem}
\begin{proof} 
Put
$$
e_\fn(z)=z\prod_{0\neq a\in \fn}\left(1-\frac{z}{a}\right) 
\quad\text{and}\quad \G_\infty^u=\left\{\begin{pmatrix} 1 & a \\ 0 & 1\end{pmatrix}\big| a\in \fn\right\}. 
$$
For $z\in \Omega$, let $|z|_i=\inf\{|z-s|\ |\ s\in \Fi\}$ be its ``imaginary'' absolute value.   
The subspace $\Omega_d=\{z\in \Omega\ |\ |z|_i\geq d\}$ of $\Omega$ is stable under $\G_\infty^u$, 
and for $d\gg 0$, the function $t(z)=e_A(z)^{-1}$ identifies $\G_\infty^u\bs \Omega_d$ 
with a small punctured disc $D_\eps^0=\{t\in \C_\infty\ |\ 0<|t|\leq \eps\}$. The function $f$ is $\G_\infty^u$-invariant, 
so can be considered as a holomorphic non-vanishing function on $D_\eps^0$. 
By the non-archimedean analogue of Picard's Big Theorem \cite[(1.3)]{vdPutEssential}, 
$f$ has at worst a pole at $t=0$, or equivalently, at the cusp $[\infty]$. Now let $[c]$ 
be any other cusp of $\G_0(\fn)$ and $\gamma\in \GL_2(A)$ be such that $\gamma[\infty]=[c]$. 
The function $f|\gamma\in \cO(\Omega)^\times$ is invariant under $\G'=\gamma^{-1}\G_0(\fn)\gamma$. 
The stabilizer of $[\infty]$ in $\G'$ contains $\G_\infty^u$, so the previous argument shows that 
$f|\gamma$ is meromorphic at $[\infty]$. Thus, $f$ is meromorphic at $[c]$.  
We conclude that $f$ descends to a rational function on 
$X_0(\fn)$ whose divisor is supported at the cusps. 

Now we use an idea from the proof of Lemma 6.2 in \cite{Pal}. The Hecke correspondence $T_\fp$ defines 
a map from the group of divisors on $X_0(\fn)$ supported at the cusps to itself (cf. (\ref{eqHC})): 
$$
T_\fp [\fd]=\begin{pmatrix} \fp & 0\\ 0 & 1\end{pmatrix} \begin{pmatrix} 1 \\ \fd\end{pmatrix} 
+\sum_{\substack{b\neq 0\\ \deg(b)<\deg(\fp)}} \begin{pmatrix} 1 & b\\ 0 & \fp\end{pmatrix} \begin{pmatrix} 1 \\ \fd\end{pmatrix} 
= \begin{pmatrix} \fp \\ \fd\end{pmatrix} +
\sum_{\substack{b\neq 0\\ \deg(b)<\deg(\fp)}} \begin{pmatrix} 1 +b\fd \\ \fp\fd\end{pmatrix}.
$$
The orbit of the cusp $[\fd]$ consists exactly of those $\begin{pmatrix} \alpha \\ \beta\end{pmatrix}\in \p^1(A)$ 
such that $\beta$ is divisible by $\fd$ and is coprime to $\fn/\fd$, so 
$
T_\fp [\fd] = (1+|\fp|) [\fd]$. 
Thus, $f|T_\fp$ and $f^{|\fp|+1}$, as rational functions on $X_0(\fn)$, have the 
same divisor. This implies that $(f|T_\fp)/f^{|\fp|+1}$ is a constant function. 
Applying $r$, we get 
$$
r(f|T_\fp) = (|\fp|+1)r(f). 
$$
Since $r$ is $\GL_2(\Fi)$-equivariant, $r(f|T_\fp)=r(f)|T_\fp$, which finishes the proof. 
\end{proof}

Lemma \ref{lemPlog} gives a natural source of $\Z$-valued Eisenstein harmonic cochains. 
Let $\Delta(z)$ be the Drinfeld discriminant function on $\Omega$ defined on page 183 of \cite{Discriminant}. This is 
a Drinfeld modular form of weight $(q^2-1)$ and type $0$ for $\GL_2(A)$, which 
vanishes nowhere on $\Omega$. Let 
$\Delta_\fn:=\Delta|B_\fn=\Delta(\fn z)$. By page 194 of \cite{Discriminant}, 
$\Delta/\Delta_\fn$ is a $\G_0(\fn)$-invariant function in $\cO(\Omega)^\times$. 
Hence $r(\Delta/\Delta_\fn)\in \cE(\fn, \Z)$. Define 
$$
\nu(\fn) =
\begin{cases}
1, & \text{if  $\deg(\fn)$ is even}\\ 
q+1, & \text{if  $\deg(\fn)$ is odd} 
\end{cases}
$$ 
and 
\begin{equation}\label{eqEn}
E_{\fn} = \frac{\nu(\fn)}{(q-1)(q^2-1)}r(\Delta/\Delta_{\fn}). 
\end{equation}
By \cite[(3.18)]{Discriminant}, $E_\fn$ is $\Z$-valued and 
primitive (i.e., $E_\fn$ is not a scalar multiple of another harmonic cochain in $\cH(\fn, \Z)$ 
except for $\pm E_\fn$). We call $E_\fn\in \cE(\fn, \Z)$ the \textit{Eisenstein series}. 
The Fourier 
expansion of $E_\fn$ can be deduced from \cite{Discriminant}: 
$$
E_{\fn}\left(\begin{pmatrix} \pi_{\infty}^k & y \\0 & 1\end{pmatrix}\right) = \nu(\fn)\cdot q^{-k+1} \cdot \left[
\frac{1-|\fn|}{1-q^2} + \sum_{0 \neq m \in A, \atop \deg(m) \leq k-2} \sigma_{\fn}(m) \eta(my)\right],
$$
where $\sigma_{\fn}(m):= \sigma(m) - |\fn| \cdot \sigma(m/\fn)$, and $\sigma$ is the divisor function
$$\sigma(m):= \begin{cases}\sum\limits_{\text{monic } m' \in A, \atop m' \mid m} |m'|, &\text{ if $m \in A$,} \\
0, & \text{ otherwise.}\end{cases}
$$
\begin{rem}
Note that for each prime $\fp \lhd A$ and $m \in A$, 
$$\sigma(m\fp) = \sigma(\fp) \sigma(m) - |\fp|\sigma(m/\fp).$$
Therefore the Fourier expansion of $E_{\fn}|T_{\fp}$ for each prime $\fp$ not dividing $\fn$ also tells us that $E_{\fn} \in \cE(\fn,\Z)$.
\end{rem}
\begin{example} 
Let $E_0$ be the $R$-valued function on $E(\sT)$ defined by
$$E_0\left(\begin{pmatrix} \pi_{\infty}^k & u \\ 0 &1 \end{pmatrix} \right) 
= - E_0 \left(\begin{pmatrix}  \pi_{\infty}^k & u \\ 0 &1 \end{pmatrix} 
\begin{pmatrix} 0 & 1 \\ \pi_{\infty} & 0 \end{pmatrix} \right) = q^{-k}.$$
This function is alternating and $\G_\infty$-invariant. For $\gamma=
\begin{pmatrix} a & b\\ c & d\end{pmatrix}\in \GL_2(A)$, let $\omega:=\ord_\infty(cu+d)$. 
By the calculations in \cite[p. 379]{Improper}
$$
(E_0|\gamma)\left(\begin{pmatrix} \pi^k & u \\ 0 & 1\end{pmatrix}\right)=
\begin{cases}
-q^{k-2\deg(c) -1} & \text{if $\omega\geq k-\deg(c)$}\\
q^{2\omega-k} & \text{if $\omega< k-\deg(c)$}. 
\end{cases}
$$
Now it is easy to see that if $\alpha\in R[q+1]$, then $\alpha E_0$ 
is the function in $\cH(1, R)$ discussed in Example \ref{example1}.  
The Hecke operator $T_\fp$ acts on $\alpha E_0$ by 
\begin{align*}
(\alpha E_0| T_\fp )\left(\begin{pmatrix} \pi^k & u \\ 0 & 1\end{pmatrix}\right) = 
&\alpha E_0\left(\begin{pmatrix} \pi^k\fp & u\fp \\ 0 & 1\end{pmatrix}\right)\\ &+\sum_{\deg(b)<\deg(\fp)} 
\alpha E_0\left(\begin{pmatrix} \pi^k/\fp & (u+b)/\fp \\ 0 & 1\end{pmatrix}\right)
\end{align*}
$$
=\alpha (-1)^{k-\deg(\fp)}+q^{\deg(\fp)}\alpha (-1)^{k+\deg(\fp)}=(1+|\fp|)\alpha 
E_0\left(\begin{pmatrix} \pi^k & u \\ 0 & 1\end{pmatrix}\right). 
$$
Therefore, $\alpha E_0$ is Eisenstein and $\cH(1, R)=\cE(1, R)\cong R[q+1]$. 
\end{example}

\begin{example}\label{examplex2}
Let $\varphi\in \cH_0(x, R)$. As we saw in Example \ref{examplex}, $\varphi$ 
is uniquely determined by $\varphi(e_0)=\alpha$.  
Note that $\varphi^0(1)=\varphi(e_{1})=q\alpha$. On the other hand, $E_x^0(1)=q$, 
so $\cH(x, R)=\cE(x, R)\cong R$ is generated by $E_x$. 
\end{example}

\begin{example}\label{exampley2}
Finally, we return to the setting of Example \ref{exampley}. The Eisenstein series $E_y$,  
as a function on $\G_0(y)\bs\sT$, can be explicitly described by 
$E_y(e_i)=E_y(e_{-i-1})=q^i$ for $i\geq 0$, and $E_y(e_u)=0$. 
Next, the function $\alpha E_0$, for any $\alpha\in R[q+1]$, can be considered 
as a function on $\G_0(y)\bs\sT$, and as such it is given by $\alpha E_0(e_u)=-\alpha$ 
and $\alpha E_0(e_i)=\alpha (-1)^{i+1}$ ($\forall i\in \Z$). Since 
any $f\in \cH(y, R)$ is uniquely determined by its values on $e_u$ and $e_0$, we see that 
$$
\cH(y, R) =\cE(y, R)= RE_y\oplus \{\alpha E_0\ |\ \alpha\in R[q+1]\}\cong R\oplus R[q+1]. 
$$
\end{example}


\subsection{Cuspidal Eisenstein harmonic cochains}\label{ssCEHC} We set 
$$
\cE_0(\fn,R):= \cE(\fn,R)\cap \cH_0(\fn,R), \quad \cE_{00}(\fn,R) := \cE(\fn,R) \cap \cH_{00}(\fn,R).
$$
Let $\fp\lhd A$ be a prime. Theorem 6.6 in \cite{Pal} 
states that $\cE_0(\fp, R)\cong R\left[\frac{|\fp|-1}{q-1}\right]$, if $\deg(\fp)$ is odd, and 
$\cE_0(\fp, R)\cong R\left[2\frac{|\fp|-1}{q^2-1}\right]$, if $\deg(\fp)$ is even. 
Note that Examples \ref{exampley} and \ref{exampley2} imply that $\cE_0(y, R)\cong R[2]$, 
which is a special case of this theorem. The main result of this section 
is a similar description of $\cE_0(\fp\fq, R)$ and $\cE_{00}(\fp\fq,R)$ under certain assumptions. Note that as a consequence 
of the Ramanujan-Petersson conjecture over function fields $\cE_0(\fn, \C)=0$ for any $\fn$. 
Therefore, $\cE_0(\fn, R)$ can be non-trivial only if $R$ has non-trivial 
additive torsion.  

\begin{defn}\label{defEI} The \textit{Eisenstein ideal} $\fE(\fn)$ of $\T(\fn)$ is the 
ideal generated by the elements $\{T_\fp-|\fp|-1\ |\ \fp \text{ is prime}, \fp \nmid \fn\}$. 
We say that a maximal ideal $\fM\lhd \T(\fn)$ is Eisenstein if $\fE(\fn)\subset \fM$.  
\end{defn}

\begin{lem}\label{prop3.1}
Let $\fp \lhd A$ be a prime, and $S$ be a set of prime ideals of $A$ of density one that 
does not contain $\fp$. A cochain $f \in \cH_{00}(\fp,R)$ is 
Eisenstein if and only if 
$$f|T_{\fq} = (|\fq|+1)f, \quad \forall \fq \in S.$$
\end{lem}
\begin{proof} 
Let $\fI(\fp)$ be the ideal of $\T(\fp)$ generated by the elements $T_\fq-|\fq|-1$, where $\fq\in S$. It is enough to 
show that $\fE(\fp)\otimes R=\fI(\fp)\otimes R$ in $\T(\fp)\otimes R$. 
The proof of the analogous statement over $\Q$ can be found in \cite[Lem. 4]{CL}. 
We briefly sketch the argument over $F$. 

Let $\fM\lhd \T(\fp)$ be a maximal ideal such that the
characteristic $\ell$ of $\T(\fp)/\fM$ is different from $p$. There is a
unique semi-simple representation
$$
\rho_{_\fM}:G_F\to \GL_2(\T(\fp)/\fM),
$$
which is unramified away from $\fp$ and $\infty$, and such that
for all primes $\fq\lhd A$, $\fq\neq \fp$, the following relations hold:
$$
\Tr\rho_{_\fM}(\Frob_\fq)=T_\fq\ \mod\ \fM,\quad
\det\rho_{_\fM}(\Frob_\fq)=|\fq|\ \mod\ \fM.
$$
The existence of such residual representations for $G_\Q$ 
is well-known. The corresponding statement over $F$ can be proved 
along the same lines (cf. \cite[Prop. 2.6]{PapikianMRL}); 
this relies on Drinfeld's fundamental results in \cite{Drinfeld}. If $\fM\supset \fI(\fp)$, then 
$T_\fq\equiv (1+|\fq|)\ (\mod\ \fM)$ for all $\fq\in S$. 
In view of the Chebotarev density and the Brauer-Nesbitt theorems, we conclude that 
$\rho_{_\fM}$ is the direct sum $\mathbf{1}\oplus \chi_\ell$ of the 
trivial and cyclotomic characters. 
But this means that $T_\fq\equiv (1+|\fq|)\ (\mod\ \fM)$ for all $\fq\neq \fp$, and  
therefore $\fM$ is Eisenstein. Now it suffices to show that 
$\fI(\fp)\otimes R \subseteq \fE(\fp)\otimes R$ 
is an equality in the completion $(\T(\fp)\otimes R)_\fM$ 
at any maximal Eisenstein $\fM$ of residue characteristic $\neq p$. (Recall that 
$p$ is invertible in $R$.)

A consequence of the proof of Theorem \ref{thm2.14} in 
\cite{Analytical}, the argument in the proof of Theorem \ref{thmAL}, 
and the fact that $\cH_{0}(1, R)=0$ is that $\T(\fp)^0\otimes R=\T(\fp)\otimes R$. 
On the other hand, 
by Theorem 5.13 in \cite{PalIJNT}, the ideal $\fE(\fp)_\fM$ in $\T(\fp)^0_\fM$
is principal, generated by $T_\fq-|\fq|-1$ for any ``good'' prime $\fq$. 
Since the density of these good primes is positive (see \cite[p. 186]{Pal}), there is a good prime in $S$.  
Therefore $(\fI(\fp)\otimes R)_\fM$ contains a generator of $(\fE(\fp)\otimes R)_\fM$ and must be equal to it. 
\end{proof}

Fix two distinct primes $\fp$ and $\fq$. 
Set
$$\nu(\fp,\fq) = \begin{cases} 1, & \text{ if $\deg (\fp)$ or $\deg (\fq)$ is even,} \\ q+1, & \text{ otherwise.} \end{cases}$$
Let $E_{(\fp,\fq)} \in \cE(\fp\fq,\Z)$ be the Eisenstein series defined by 
$$E_{(\fp,\fq)}\left(\begin{pmatrix} \pi_{\infty}^k & u \\0 & 1\end{pmatrix}\right) =
\nu(\fp,\fq)\cdot q^{-k+1} \cdot \left[
\frac{(1-|\fp|)(1-|\fq|)}{1-q^2} + \sum_{0 \neq m \in A, \atop \deg(m) \leq k-2} \sigma_{\fp\fq}'(m) \eta(mu)\right],$$
where
$$\sigma_{\fn}'(m):= \sum_{\text{monic } m' \in A, \ (m',\fn) = 1, \atop \text{ and } m' \mid m} |m'|.$$
It is clear that $E_{(\fp,\fq)} = E_{(\fq,\fp)}$, and comparing the Fourier expansions we get
\begin{eqnarray}\label{eqn3.1}
\frac{\nu(\fp)}{\nu(\fp,\fq)} \cdot E_{(\fp,\fq)} &=&  \Big(E_{\fp} - E_{\fp}|B_{\fq}\Big).
\end{eqnarray}

\begin{lem}\label{lem3.2}\hfill
\begin{enumerate}
\item Viewing $E_{\fp}$ as a harmonic cochain in $\cH(\fp\fq,\Z)$, we get
$$E_{\fp}|U_{\fp} = E_{\fp} = - E_{\fp}|W_{\fp} \quad \text{ and } \quad E_{\fp}|W_{\fq} = E_{\fp}|B_{\fq}.$$
\item $E_{(\fp,\fq)} = E_{(\fp,\fq)}|U_{\fp} = E_{(\fp,\fq)}|U_{\fq} = E_{(\fp,\fq)}|W_{\fp\fq} = - E_{(\fp,\fq)}|W_{\fp} = - E_{(\fp,\fq)}|W_{\fq}.$
\end{enumerate}
\end{lem}
\begin{proof}
The proof is straightforward.
\end{proof}

Let $\cE'(\fp\fq,R)$ be the $R$-submodule of $\cE(\fp\fq,R)$  spanned by 
$E_{\fp}$, $E_{\fq}$ and $E_{(\fp,\fq)}$, i.e., $\cE'(\fp\fq,R)= R E_{\fp} + R E_{\fq} + R E_{(\fp,\fq)}$. 
Denote 
$$
\cE_{0}'(\fp\fq,R)= \cE'(\fp\fq,R) \cap \cH_{0}(\fp\fq,R) \quad 
\text{and}\quad \cE_{00}'(\fp\fq,R)= \cE'(\fp\fq,R) \cap \cH_{00}(\fp\fq,R). 
$$

\begin{thm} \label{thm3.9} The following holds:
\begin{enumerate}
\item If $\nu(\fp,\fq)$ is invertible in $R$, then $\cE'(\fp\fq,R)$ is a free $R$-module of rank $3$.
\item If $q-1$ and $s_{\fp,\fq} = \gcd(|\fp|+1,|\fq|+1)$ are both invertible in $R$, then 
$$\cE_{00}(\fp\fq,R) = \cE_{00}'(\fp\fq,R).$$ 
\end{enumerate}
\end{thm}
\begin{proof}
By (\ref{eqn3.1}), $E_{\fp}|B_{\fq}$ and $E_{\fq}|B_{\fp}$ are both in $\cE'(\fp\fq,R)$, and
$$\cE'(\fp\fq,R) = R (E_{\fp}|B_{\fq}) + R(E_{\fq}|B_{\fp}) + RE_{(\fp,\fq)}.$$
Suppose $f = a E_{\fp}|B_{\fq} + b E_{\fq}|B_{\fp} + c E_{(\fp,\fq)} = 0$.
Then $f^*(1) = c E_{(\fp,\fq)}^*(1) = c \cdot \nu(\fp,\fq) = 0$.
Since $\nu(\fp,\fq)$ is invertible in $R$ under our assumption, we get $c$ = 0. 
Without loss of generality we can assume that either 
$\deg(\fp)$ is even, or $\deg(\fp)$ and $\deg(\fq)$ are both odd. 
Now $E_{\fp}^*(1) = (E_{\fp}|B_{\fq})^*(\fq) = \nu(\fp)$ is invertible in $R$.
Therefore from $f^*(\fq) = a \cdot \nu(\fp) = 0$ we get $a = 0$.
From the fact that $E_{\fp}$ is primitive, we have $b=0$. The proof of Part (1) is complete.

To prove Part (2), it suffices to show that $\cE_{00}(\fp\fq,R)$ is contained in $\cE'(\fp\fq,R)$. 
Given an Eisenstein harmonic cochain $f \in \cE_{00}(\fp\fq,R)$, let $$f_1:= f - f^*(1) \nu(\fp,\fq)^{-1} E_{(\fp,\fq)} \in \cE(\fp\fq,R).$$
Then for every ideal $\fm \lhd A$, $f_1^*(\fm) = 0$ unless $\fp$ or $\fq$ divides $\fm$.
By the method of Section \ref{sAL}, $(|\fp|+1) f_1 = \psi_1|B_{\fp} + \psi_2|B_{\fq}$, where $\psi_1 \in \cH(\fq,R)$ and $\psi_2 \in \cH(\fp,R)$. 
Moreover, by the proof of Theorem \ref{thmAL}, we can take 
$$\psi_1 = f_1|(U_{\fp}+W_{\fp}) \quad \text{ and } \quad \psi_2 = |\fq|^{-1} \Big[(|\fp|+1) f_1|U_{\fq} -
f_1|U_{\fq}(U_{\fp}+W_{\fp})B_{\fp}\Big].$$
Since Lemma \ref{lem3.2} (2) implies that $E_{(\fp,\fq)}|(U_{\fp}+W_{\fp}) = 0$, we get
$\psi_1 = f|(U_{\fp}+W_{\fp}) \in \cH_{00}(\fq,R)$ and 
\begin{eqnarray}
\psi_2 &=& |\fq|^{-1} \Big[(|\fp|+1) f|U_{\fq} -
f|U_{\fq}(U_{\fp}+W_{\fp})B_{\fp}\Big]  \nonumber \\
& & - |\fq|^{-1}(|\fp|+1) f^*(1) \nu(\fp,\fq)^{-1} E_{(\fp,\fq)} \in \cH(\fq,R). \nonumber
\end{eqnarray}
The constant term $\psi_2^0(1)$ is equal to
$$-|\fq|^{-1}(|\fp|+1)f^*(1) \cdot \frac{(1-|\fp|)(1-|\fq|)}{q(1-q^2)} = -|\fq|^{-1}(|\fp|+1)f^*(1) \cdot \frac{1-|\fq|}{\nu(\fp)} E_{\fp}^0(1)$$
and $(\psi_2|W_{\fp})^0(1) = -\psi_2^0(1)$.
Let $$\psi_2':= \psi_2 + |\fq|^{-1}(|\fp|+1)f^*(1) \cdot \frac{1-|\fq|}{\nu(\fp)} E_{\fp} \in \cH(\fp,R).$$
Since $q-1$ is invertible in $R$ by assumption, Lemma \ref{lem1.3} and \ref{lem2.12} show that 
$$\psi_2' \in \cH_0(\fp,R) = \cH_{00}(\fp,R).$$
Note that for every prime $\fp' \lhd A$ different from $\fp$ and $\fq$, we have
$$ \psi_1|T_{\fp'} = (|\fp'|+1)\psi_1 \quad \text{ and } \quad \psi_2'|T_{\fp'} = (|\fp'|+1)\psi_2'.$$
Lemma \ref{prop3.1} implies that $\psi_1 \in \cE_{00}(\fq,R)$ and $\psi_2' \in \cE_{00}(\fp,R)$.
By Theorem 6.6 in \cite{Pal}, we can find $a_1, \ a_2 \in R$  such that $\psi_1 = a_1 E_{\fq}$ 
and $\psi_2' = a_2 E_{\fp}$ (as $2$ is invertible in $R$ by our assumption). We conclude that
\begin{align*}
(|\fp|+1)f & = \psi_1|B_{\fp} + \psi_2|B_{\fq} + f^*(1) \nu(\fp,\fq)^{-1} E_{(\fp,\fq)}\\
&= a_2 E_{\fq}|B_{\fp} \begin{aligned}[t]& + \Big(a_2-|\fq|^{-1}(|\fp|+1)f^*(1) 
\cdot \frac{1-|\fq|}{\nu(\fp)} \Big)E_{\fp}|B_{\fq}  \\ 
&+ f^*(1) \nu(\fp,\fq)^{-1} E_{(\fp,\fq)} 
\end{aligned}
\end{align*}
is in $\cE'(\fp\fq,R)$.
Similarly, we also get $(|\fq|+1)f \in \cE'(\fp\fq,R)$ by interchanging $\fp$ and $\fq$.
Since $s_{\fp,\fq}$ is invertible in $R$, $f$ must be in $\cE'(\fp\fq,R)$, which completes the proof of Part (2).
\end{proof}

\begin{lem}\label{lem3.4} The following holds:  
\begin{enumerate}
\item Suppose $\nu(\fp,\fq)$ is invertible in $R$.
The $R$-module $\cE'_0(\fp\fq,R)$ is torsion and isomorphic to the submodule of $R^3$ consisting of elements
$(a,b,c)$ with
\begin{eqnarray}
&& \frac{(|\fp|-1)(|\fq|+1)}{q^2-1}\nu(\fp) a = 0, \nonumber \\
&& \frac{(|\fp|+1)(|\fq|-1)}{q^2-1}\nu(\fq) b = 0, \nonumber \\
&& \frac{1-|\fp|}{1-q^2} \nu(\fp) a + \frac{1-|\fq|}{1-q^2}\nu(\fq) b + \frac{(1-|\fp|)(1-|\fq|)}{1-q^2}\nu(\fp,\fq) c = 0. \nonumber
\end{eqnarray}
\item Suppose further that $2$ is also invertible in $R$. Then $\cE'_0(\fp\fq,R)$ is isomorphic to 
$$R\left[\frac{(|\fp|-1)(|\fq|+1)}{q^2-1} 
\nu(\fp)\right] \oplus R\left[\frac{(|\fp|+1)(|\fq|-1)}{q^2-1}\nu(\fq)\right] \oplus R\left[\frac{(|\fp|-1)(|\fq|-1)}{q^2-1}\right].$$
\end{enumerate}
\end{lem}
\begin{proof}
Let $f = a E_{\fp} + b E_{\fq} + c E_{(\fp,\fq)}$ with $a,\ b,\ c \in R$. By Lemma \ref{lem2.12}, $f \in \cE'_0(\fp\fq,R)$ if and only if 
$$f^0(1) = (f|W_{\fp})^0(1) = (f|W_{\fq})^0(1) = (f|W_{\fp\fq})^0(1) = 0.$$ This gives us the following equations:
\begin{eqnarray}\label{eqn3.2}
\frac{1-|\fp|}{1-q^2} \nu(\fp) a + \frac{1-|\fq|}{1-q^2}\nu(\fq) b + \frac{(1-|\fp|)(1-|\fq|)}{1-q^2}\nu(\fp,\fq) c & = & 0, \\ \label{eqn3.3}
-\frac{1-|\fp|}{1-q^2} \nu(\fp) a + |\fp|\frac{1-|\fq|}{1-q^2}\nu(\fq) b - \frac{(1-|\fp|)(1-|\fq|)}{1-q^2}\nu(\fp,\fq) c & = & 0, \\ \label{eqn3.4}
|\fq| \frac{1-|\fp|}{1-q^2} \nu(\fp) a - \frac{1-|\fq|}{1-q^2}\nu(\fq) b - \frac{(1-|\fp|)(1-|\fq|)}{1-q^2}\nu(\fp,\fq) c & = & 0, \\ \label{eqn3.5}
-|\fq|\frac{1-|\fp|}{1-q^2} \nu(\fp) a -|\fp| \frac{1-|\fq|}{1-q^2}\nu(\fq) b + \frac{(1-|\fp|)(1-|\fq|)}{1-q^2}\nu(\fp,\fq) c & = & 0. 
\end{eqnarray}

We remark that $$\text{Equation } (\ref{eqn3.5})= -\Big(\text{ Equation } (\ref{eqn3.2}) + (\ref{eqn3.3}) + (\ref{eqn3.4})\Big).$$
Considering the equation (\ref{eqn3.2})+(\ref{eqn3.3}), (\ref{eqn3.2})+(\ref{eqn3.4}), and (\ref{eqn3.2})+(\ref{eqn3.5}), we get
\begin{eqnarray} \label{eqn3.6}
&& \frac{(|\fp|+1)(|\fq|-1)}{q^2-1}\nu(\fq) b = 0, \\ \label{eqn3.7}
&& \frac{(|\fp|-1)(|\fq|+1)}{q^2-1}\nu(\fp) a = 0, \\ \label{eqn3.8}
&& \frac{(|\fp|-1)(|\fq|-1)}{q^2-1} \big(\nu(\fp)a+\nu(\fq)b+ 2\nu(\fp,\fq) c\big) = 0. 
\end{eqnarray}
Since the conditions of $(a,b,c)$ described by Equation (\ref{eqn3.2})$\sim$(\ref{eqn3.5}) is 
equivalent to those described by Equation (\ref{eqn3.2}), (\ref{eqn3.6}), and (\ref{eqn3.7}) combined, the proof of Part (1) is complete.

To prove Part (2), note that $2\nu(\fp,\fq)$ is 
invertible in $R$ by assumption. Therefore the conditions of $(a,b,c)$ described by Equation (\ref{eqn3.2})$\sim$(\ref{eqn3.5}) is equivalent to those described by Equation (\ref{eqn3.6})$\sim$(\ref{eqn3.8}).
Let $E':= \nu(\fp)E_{\fp} + \nu(\fq)E_{\fq} + 2\nu(\fp,\fq)E_{(\fp,\fq)}$.
By Theorem \ref{thm3.9} (1), we also have 
$\cE'(\fp\fq,R) = R E_{\fp} \oplus R E_{\fq} \oplus R E'$. Therefore Equation (\ref{eqn3.6})$\sim$(\ref{eqn3.8}) assures the result.
\end{proof}

In fact, when $q-1$ and $s_{\fp,\fq}$ are invertible in the coefficient ring $R$, one can show that 
$\cE_{00}(\fp\fq,R) = \cE'_{0}(\fp\fq,R)$; see Remark \ref{rem7.4}. 


\subsection{Special case}\label{sec3.1}
In this subsection we give a concrete description of $\cE_0(xy,R)$ and $\cE_{00}(xy,R)$ for an arbitrary coefficient ring $R$.
Recall that $E_0$ is the function on $E(\sT)$ satisfying
$$E_0\left(\begin{pmatrix} \pi_{\infty}^k & u \\ 0 &1 \end{pmatrix} \right) = - E_0 \left(\begin{pmatrix}  
\pi_{\infty}^k & u \\ 0 &1 \end{pmatrix} \begin{pmatrix} 0 & 1 \\ \pi_{\infty} & 0 \end{pmatrix} \right) = q^{-k}.$$

\begin{lem}\label{lem3.6}
Given $a \in R[q+1]$, we have $aE_0 = -aE_{\fn} \in \cH(1,R)$ if $\deg(\fn)$ is odd.
Moreover, $aE_0|B_{\fm} = (-1)^{\deg(\fm)} aE_0$.
\end{lem}
\begin{proof}
The proof is straightforward.
\end{proof}

According to Examples \ref{examplex2} and \ref{exampley2} we have: 
\begin{lem}\label{lem3.7}
$\cH(x,R) = R E_{x}$ and $\cH(y,R) = R E_{y} \oplus R[q+1] E_0$.
\end{lem}

Note that $\nu(x,y) = \nu(y) = 1$ and $\nu(x) = q+1$. Given $f \in \cE(xy,R)$, let $f' := f - f^*(1) E_{(x,y)}$.
Then $f'^*(\fm) = 0$ unless $x$ or $y$ divides $\fm$. By Theorem \ref{thmAL} and Lemma \ref{lem3.7}, there exists $a, \ b \in R$ and $b' \in R[q+1]$ such that
$$2 f' = (b E_{y} + b' E_0)|B_{x} + aE_{x}|B_{y}.$$
When $q$ is even, $2$ is invertible in $R$ and we have $f \in \cE'(xy,R)$.
Suppose $q$ is odd.
Note that $b'E_0|B_{x} = b'E_{x}$ by Lemma \ref{lem3.6}. Hence
$$2 f = (a+b') E_{x}|B_{y} + bE_{y}|B_{x} + 2 f^*(1) E_{(x,y)} \in \cE'(xy,R).$$
In fact, there exists $a'', \ b'' \in R$ such that $a+b' = 2 a''$ and $b = 2 b''$. Indeed, by Lemma \ref{lem_new14} we have
$$2f^*(x) = b + 2 f^*(1) \text{ and } 
2f^0(\pi_{\infty}) = 2q^{-1} f^0(1) = |y| (a+b') + |x| b + 2 f^*(1) (q-1).$$
Set $f'':= f-\big(a''E_{x}|B_{y} + b'' E_{y}|B_{x} + f^*(1)E_{(x,y)} \big) \in \cE(xy,R)$. Then the above discussion shows that $2f'' =0$.
Take $f''':= f'' - (f'')^0(1)E_0$. Then 
$$f''' \in \cE(xy,R)[2] \quad \text{ and } \quad (f''')^*(1) = (f''')^0(1) = 0.$$
The following lemma shows that $f''' \in \cE'(xy,R)$, which also implies that $f \in \cE'(xy,R)$.

\begin{lem}\label{lem3.8}
Suppose $q$ is odd.
Given $\varphi \in \cE(xy,R)[2]$ with $\varphi^*(1) = \varphi^0(1) = 0$,
there exists $\alpha \in R[2]$ such that
$$\varphi = \alpha (E_x + E_y|B_x).$$
\end{lem}
\begin{proof} We use the notation in the proof of Proposition \ref{propT(xy)}. 
Given $\varphi \in \cE(xy,R)[2]$ with $\varphi^0(1)$ = 0, we have that for each prime $\fp$ with $\fp \nmid x y$,
$$\varphi|T_{\fp} = (|\fp|+1)\varphi = 0.$$
Therefore for every prime $\fp$ with $\fp \nmid xy$ and a non-zero ideal $\fm \lhd A$,
\begin{eqnarray}\label{eqn3.10}
&&|\fp|\varphi^*(\fp \fm) + \varphi^*(\fm/\fp)= (\varphi|T_{\fp})^*(\fm) = 0.
\end{eqnarray}
By Equation (\ref{eqn3.10}) and the Fourier expansion of $\varphi$, we get
$$\varphi\left(\begin{pmatrix}\varpi_x^k&0\\0&1\end{pmatrix}\right) = 0, \quad \forall k \in \Z,$$
$$\varphi(a_1) = \varphi(a_2) = 0,$$
$$\varphi(b_u) = \varphi(a_4) = \varphi^*(x), \quad \forall u \in \F_q^{\times}, $$
$$\varphi(c_3) = \varphi^*(y), \quad \varphi(a_3) = \varphi^*(x) + \varphi^*(x^2), \quad \varphi(a_6) = \varphi^*(x^2).$$
The harmonicity of $\varphi$ gives us that
$$0 = \varphi(a_4) + \sum_{u \in \F_q^{\times}} \varphi(b_u) - \varphi(a_3) = \varphi^*(x^2),$$
$$0 = \varphi(a_6) + \varphi(c_3) + (1-q) \varphi(a_2) = \varphi^*(x^2) + \varphi^*(y).$$
$$0 = \varphi(a_6) + (q-1)\varphi(a_3) - \varphi(c_4) = \varphi^*(x^2) + \varphi(c_4).$$
Hence
$$\varphi(e) = 0 \quad \text{ for $e = c_1,c_2,c_3,c_4, a_1,a_2,a_5,a_6,$ and }$$
$$\varphi(a_3) = \varphi(a_4) = \varphi(b_u) = \varphi^*(x) \in R[2], \quad \text{ for $u \in \F_q^{\times}$.}$$
On the other hand, for $\alpha \in R[2]$, 
$$\alpha(E_{x}+E_{y}|B_{x}) (e) = 0 \quad \text{ for $e = c_1,c_2,c_3,c_4, a_1,a_2,a_5,a_6,$ and}$$
$$\alpha(E_{x}+E_{y}|B_{x}) (a_3) = \alpha(E_{x}+E_{y}|B_{x})(a_4) = \alpha(E_{x}+E_{y}|B_{x})(b_u) = \alpha, 
\text{ for $u \in \F_q^{\times}$.}$$
Therefore $\varphi = \varphi(a_3) \cdot (E_{x}+E_{y}|B_{x})$ and the proof is complete.
\end{proof}

From the above discussion, we conclude that
\begin{cor}\label{cor3.9}
$\cE(xy,R) = \cE'(xy,R)$ for every coefficient ring $R$.
In other words, every Eisenstein harmonic cochain of level $xy$ can be generated by Eisenstein series.
\end{cor}

By Lemma \ref{lem3.4} and Corollary \ref{cor3.9}, we immediately get

\begin{cor}\label{cor3.10}
The space $\cE_0(xy,R)$ is isomorphic to the torsion $R$-module
$$\{(a,b,c) \in R^3 \mid (q^2+1)a = (q+1)b = a+b+(1-q)c = 0\}.$$
In particular,
$$\cE_0(xy,R) \cong R\Big[\frac{(q-1)}{2}(q^2+1)(q+1)\Big]\oplus R[2].$$
\end{cor}

From the graph in Figure \ref{Fig4}, an alternating $R$-valued function $f$ on 
$E(\Gamma_0(xy)\backslash \sT)$ is in $\cH_{0}(xy,R)$ if and only if $f$ vanishes on the cusps $c_1, c_2, c_3, c_4$ and
\begin{eqnarray}
&&(q-1)f(a_1)+f(a_5) = 0, \quad (q-1)f(a_2)-f(a_6) = 0, \nonumber\\
&&(q-1)f(a_3)+f(a_6) = 0, \quad (q-1)f(a_4)-f(a_5) = 0, \nonumber\\
&&f(a_2) + \sum_{u \in \F_q^{\times}} f(b_u) = f(a_1), \quad 
f(a_3) - \sum_{u \in \F_q^{\times}} f(b_u) = f(a_4). \nonumber
\end{eqnarray}
Moreover, $f$ is in $\cH_{00}(xy,R)$ if and only if $f$ satisfies an extra equation:
$$f(a_1)+f(a_4) = 0.$$
In particular, every harmonic cochain $f \in \cH_{00}(xy,R)$ is determined uniquely by the values
$$f(a_1), \ f(b_u) \ \text{ for } u \in \F_q^{\times}.$$

Let $f = aE_{x}+ bE_{y}+cE_{(x,y)} \in \cE_0(xy,R)$. By Corollary \ref{cor3.10} we have
$$(q^2+1)a = (q+1)b = a+b+(1-q)c = 0.$$
It is observed that
$$\begin{tabular}{ll}
$E_{x}(a_1) = -1$, & $E_{x}(a_4) = -q$, \\
$E_{y}(a_1) = 0$, & $E_{y}(a_4) = -1$, \\
$E_{(x,y)}(a_1) = -1$, &$E_{(x,y)}(a_4) = -1.$
\end{tabular}$$
We then get
$$f(a_1)+f(a_4) = -\big((q+1) a + b + 2 c\big).$$
Hence $f \in \cE_{00}(xy,R)$ if and only if 
$$c \in R[(q^2+1)(q+1)], \ a = -q^{-1}(q+1) c, \ b = q^{-1}(q^2+1) c.$$

We conclude that
\begin{prop}\label{lem3.11}
The module $\cE_{00}(xy,R)$ is isomorphic to $R[(q^2+1)(q+1)]$. More precisely,
every harmonic cochain in $\cE_{00}(xy,R)$ must be of the form
$$c \cdot \Big(-(q+1) E_{x} + (q^2+1)E_{y} + q E_{(x,y)}\Big)$$
where $c \in R[(q^2+1)(q+1)]$.
\end{prop}

\begin{cor}\label{previousN}
For every natural number $n$ relatively prime to $p$ the module $\cE_{00}(xy,\Z/n\Z)$ 
is isomorphic to $\Z/n\Z[(q^2+1)(q+1)]$. 
\end{cor}
\begin{proof} 
The difference between this claim and Proposition \ref{lem3.11} is that $\Z/n\Z$ 
is not a coefficient ring in general. Still, one can deduce this from Proposition \ref{lem3.11} 
by arguing as in the proof of Corollary 6.9 in \cite{Pal}. First, one easily reduces to 
the case when $n$ is a power of some prime $\ell\neq p$, and applies Proposition \ref{lem3.11} 
with $R=\Z_\ell[\zeta_p]/n\Z_\ell[\zeta_p]$, where $\zeta_p$ is the primitive $p$th root of unity. 
The claim follows by observing that $R$ is a free $\Z/n\Z$-module. 
\end{proof}

\begin{cor}\label{corT/E}
$\T(xy)/\fE(xy)\cong \Z/(q^2+1)(q+1)\Z$. 
\end{cor}
\begin{proof} Let  $(\T(\fn)^0)^\new$ be the quotient of $\T(\fn)^0$ with which it acts on $\cH_0(\fn, \Q)^\new$; 
cf. Definition \ref{defnNewH}. By Lemma 4.2 and Theorem 4.5 in \cite{PalIJNT}, for any \textit{square-free} $\fn\lhd A$ 
the quotient ring $(\T(\fn)^0)^\new/\fE(\fn)$ is a finite cyclic group of order coprime to $p$; here 
with abuse of notation $\fE(\fn)$ denotes the ideal generated by the images of $T_\fp-|\fp|-1$ in $(\T(\fn)^0)^\new$. 
Since $\cH_0(xy, \Q)=\cH_0(xy, \Q)^\new$, we have $\T(xy)^0=(\T(xy)^0)^\new$. 
On the other hand, by Proposition \ref{propT(xy)}, $\T(xy)^0=\T(xy)$. Hence $\T(xy)/\fE(xy)\cong \Z/n\Z$ for some 
$n$ coprime to $p$. 

The perfectness of the pairing  (\ref{GPairing}) implies 
$$
\Hom_{\Z/n\Z}(\T(xy)\otimes_\Z\Z/n\Z, \Z/n\Z)\cong \cH_{00}(xy, \Z/n\Z). 
$$
Hence 
\begin{align*}
\cE_{00}(xy, \Z/n\Z) & \cong \Hom_{\Z/n\Z}(\T(xy)\otimes_\Z\Z/n\Z, \Z/n\Z)[\fE(xy)]\\
& \cong \Hom_{\Z/n\Z}((\T(xy)/\fE(xy))\otimes_\Z\Z/n\Z, \Z/n\Z)\\
&\cong \Hom_{\Z/n\Z}(\Z/n\Z\otimes_\Z\Z/n\Z, \Z/n\Z) \cong \Z/n\Z. 
\end{align*}

Applying Corollary \ref{previousN}, we conclude that $n$ must divide $(q+1)(q^2+1)$. Later in this paper 
we will prove that the component group $\Phi_\infty\cong \Z/(q^2+1)(q+1)\Z$ of $J_0(xy)$ 
is annihilated by $\fE(xy)$ (see Lemma \ref{lem5.5}). This implies that $n$ is divisible 
by $(q^2+1)(q+1)$. Therefore, $n=(q^2+1)(q+1)$. 
\end{proof}

\begin{rem}
In \cite{PW2}, we extended our calculation of $\T(\fn)/\fE(\fn)$ to arbitrary $\fn$ of degree $3$. 
Up to an affine transformation $T\mapsto aT + b$ with
$a \in \F_q^\times$ and $b \in \F_q$, there are 5 different cases, namely 
\begin{enumerate}
\item If $\fn=T^3$, then $\T(\fn)/\fE(\fn)\cong \Z/q^2\Z$;
\item If $\fn=T^2(T-1)$, then $\T(\fn)/\fE(\fn)\cong \Z/q(q^2-1)\Z$;
\item If $\fn$ is irreducible, then $\T(\fn)/\fE(\fn)\cong \Z/(q^2+q+1)\Z$;
\item If $\fn=xy$, then $\T(\fn)/\fE(\fn)\cong \Z/(q^2+1)(q+1)\Z$;
\item If $\fn=T(T-1)(T-c)$, where $c\in \F_q$, $c\neq 0,1$ (here we must have $q>2$), then 
$$
\T(\fn)/\fE(\fn)\cong \Z/(q+1)\Z\times  \Z/(q+1)\Z\times \Z/(q-1)^2(q+1)\Z.
$$ 
\end{enumerate}
\end{rem}


\section{Drinfeld modules and modular curves}\label{sDMC} 

In this section we collect some facts about Drinfeld modules and their moduli schemes 
that will be used later in the paper. 

Let $S$ be an $A$-scheme and $\cL$ a line bundle over $S$. Let $\cL\{\tau\}$ 
be the noncommutative ring $\oplus_{i\geq 0}\cL^{\otimes(1-q^i)}(S)\tau^i$, 
where $\tau$ stands for the $q$th power Frobenius mapping. The multiplication in this 
ring is given by $\alpha_i\tau^i\cdot \alpha_j\tau^j=(\alpha_i\otimes \alpha_j^{\otimes q^j})\tau^{i+j}$. 
A \textit{Drinfeld $A$-module of rank $r$ (in standard form)} over $S$ 
is given by a line bundle $\cL$ over $S$ together with a ring 
homomorphism $\phi^\cL: A\to \End_{\F_q}(\cL)=\cL\{\tau\}$, $a\mapsto \phi^\cL_a$, such that 
$\phi^\cL_a=\sum_{i=0}^{m(a)}\alpha_i(a)\tau^i$, where $m(a)=-r\cdot \ord_\infty(a)$, 
$\alpha_{m(a)}(a)$ is a nowhere vanishing section of $\cL^{\otimes(1-m(a))}$, 
and $\alpha_0$ coincides with the map $\partial: A\to H^0(S, \cO_S)$ giving the 
structure of an $A$-scheme to $S$; cf. \cite[p. 575]{Drinfeld}. The kernel of $\partial$ is called the \textit{$A$-characteristic} of $S$. 
A Drinfeld $A$-module over $S$ 
is clearly an $A$-module scheme over $S$, and a \textit{homomorphism} of Drinfeld modules 
is a homomorphism of these $A$-module schemes. A homomorphism 
of Drinfeld modules over a connected scheme is either the zero homomorphism, or it has finite kernel, in which 
case it is usually called an \textit{isogeny}. 
When $S$ is the spectrum of a field $K$, we will omit mention of $\cL$ and write 
$\phi: A\to K\{\tau\}$. 

Let $\fn\lhd A$ be a non-zero ideal. A \textit{cyclic subgroup of order $\fn$} of $\phi^\cL$ is an 
$A$-submodule scheme $C_\fn$ of $\cL$ which is finite and flat over $S$, and such that 
there is a homomorphism of $A$-modules $\iota: A/\fn\to \cL(S)$ giving an 
equality of relative effective Cartier divisors $\sum_{a\in A/\fn}\iota(a)=C_\fn$. Denote 
$\phi^\cL[\fn]=\ker(\cL\xrightarrow{\phi^\cL_n}\cL)$, where $n$ is a generator of the ideal $\fn$. 
It is clear that the $A$-submodule scheme $\phi^\cL[\fn]$ of $\cL$ 
does not depend on the choice of $n$, and $C_\fn\subset \phi^\cL[\fn]$. To each $C_\fn\subset \phi$ 
one can associate a unique, up to isomorphism, Drinfeld module 
$\phi':=\phi/C_\fn$ such that there is an isogeny $\phi\to \phi'$ 
whose kernel is $C_\fn$. 

Now assume $S=\Spec(K)$, where $K$ is a field. Explicitly, a 
homomorphism of Drinfeld modules $u:\phi\to \psi$ is
$u\in K\{\tau\}$ such that $\phi_a u=u\psi_a$ for all $a\in A$, and 
$u$ is an isomorphism if $u\in K^\times$. Let $\End(\phi)$ denote the centralizer of $\phi(A)$ in
$\bar{K}\{\tau\}$, i.e., the ring of all homomorphisms $\phi\to
\phi$ over $\bar{K}$. The automorphism group $\Aut(\phi)$ is the
group of units $\End(\phi)^\times$. It is known that $\End(\phi)$ is a free $A$-module 
of rank $\leq r^2$; cf. \cite{Drinfeld}. 
From now on we also assume that $r=2$. Note that $\phi$
is uniquely determined by the image of $T$:
$$
\phi_T=\partial(T)+g\tau+\Delta\tau^2,
$$
where $g\in K$ and $\Delta\in K^\times$. The \textit{$j$-invariant}
of $\phi$ is $j(\phi)=g^{q+1}/\Delta$. It is easy to check that if
$K$ is algebraically closed, then $\phi\cong \psi$ if and only if
$j(\phi)=j(\psi)$. It is also easy to check that 
\begin{equation}\label{eqAut}
\Aut(\phi)=\begin{cases}
\F_q^\times & \text{if $j(\phi)\neq 0$};\\
\F_{q^2}^\times & \text{if $j(\phi)= 0$}.
\end{cases}
\end{equation}
If $\fn$ is
coprime to the $A$-characteristic of $K$, then $\phi[\fn](\bar{K})\cong
(A/\fn)^2$. On the other hand, if $\fp=\ker(\partial)\neq 0$, then
$\phi[\fp](\bar{K}) \cong (A/\fp)$ or $0$; when $\phi[\fp](\bar{K}) = 0$, $\phi$ is
called \textit{supersingular}. The following is Theorem 5.9 in \cite{GekelerADM}: 

\begin{thm}\label{thmSSN} Let $\fp\lhd A$ be a prime ideal. 
The number of isomorphism classes of supersingular 
rank-$2$ Drinfeld $A$-modules over $\overline{\F}_\fp$ is 
$$
\begin{cases}
\frac{|\fp|-1}{q^2-1} & \text{ if $\deg(\fp)$ is even};\\
\frac{|\fp|-q}{q^2-1}+1 & \text{ if $\deg(\fp)$ is odd}.
\end{cases}
$$
The Drinfeld module with $j(\phi)=0$ is supersingular 
if and only if $\deg(\fp)$ is odd. 
\end{thm}

The functor from the category of $A$-schemes 
to the category of sets, which associates to an $A$-scheme 
$S$ the set of isomorphism classes of pairs $(\phi^\cL, C_\fn)$, 
where $\phi^\cL$ is a Drinfeld module of rank $2$ and $C_\fn$ is a cyclic subgroup of order 
$\fn$ has a coarse moduli scheme $Y_0(\fn)$. The scheme $Y_0(\fn)$ is affine, 
finite type, of relative dimension $1$ over $\Spec(A)$, and smooth over $\Spec(A[\fn^{-1}])$. 
This is well-known and can be deduced from the results in \cite{Drinfeld}. 
The rigid-analytic uniformization of $Y_0(\fn)$ over $\Fi$ is given by (\ref{eqUnifY}). 
The scheme $Y_0(\fn)$ has a canonical compactification over $\Spec(A)$:

\begin{thm}\label{thmX0}
There is a proper normal geometrically irreducible scheme $X_0(\fn)$ 
of pure relative dimension $1$ over $\Spec(A)$ which contains $Y_0(\fn)$ 
as an open dense subscheme. The complement $X_0(\fn)-Y_0(\fn)$ is a 
disjoint union of irreducible schemes. Finally, $X_0(\fn)$ is smooth over $\Spec(A[\fn^{-1}])$. 
\end{thm}
\begin{proof}
See \cite[$\S$9]{Drinfeld} and \cite[Prop. V.3.5]{Lehmkuhl}. 
\end{proof}

Denote the Jacobian variety of $X_0(\fn)_F$ by $J_0(\fn)$. 
Let $\fp\lhd A$ be prime. There are two natural degeneracy 
morphisms $\alpha, \beta: Y_0(\fn\fp)\to Y_0(\fn)$ with moduli-theoretic interpretation:
$$
\alpha:(\phi, C_{\fn\fp})\mapsto (\phi, C_\fn), \qquad \beta:(\phi, C_{\fn\fp})\mapsto (\phi/C_\fp, C_{\fn\fp}/C_\fp),
$$
where $C_\fn$ and $C_\fp$ are the subgroups of $C_{\fn\fp}$ of order $\fn$ and $\fp$, respectively. 
These morphisms are proper, and hence 
uniquely extend to morphisms $\alpha, \beta: X_0(\fn\fp)\to X_0(\fn)$. 
By Picard functoriality, $\alpha$ and $\beta$ induce two homomorphisms $\alpha_\ast, \beta_\ast: J_0(\fn)\to J_0(\fp\fn)$. 
The \textit{Hecke endomorphism} of $J_0(\fn)$ is $T_\fp:=\alpha^\ast\circ \beta_\ast$ , where $\alpha^\ast: J_0(\fp\fn)\to J_0(\fn)$ 
is the dual of $\alpha_\ast$. The $\Z$-subalgebra of $\End(J_0(\fn))$ generated 
by all Hecke endomorphisms is canonically isomorphic to $\T(\fn)$. This is a 
consequence of Drinfeld's reciprocity law \cite[Thm. 2]{Drinfeld}. 

The Jacobian $J_0(\fn)$ has a rigid-analytic uniformization over $\Fi$
as a quotient of a multiplicative torus by a discrete lattice. 
To simplify the notation denote $\G:=\G_0(\fn)$ and let $\bG$ be the maximal torsion-free abelian quotient of $\G$. 
In \cite{GR} and \cite{GekelerCDG}, Gekeler and Reversat associate a meromorphic theta 
function $\theta(\omega, \eta, \cdot)$ on $\Omega$ with each pair $\omega, \eta\in \overline{\Omega}:=\Omega\cup \p^1(F)$. 
The theta function $\theta(\omega, \eta, \cdot)$ satisfies a functional equation 
$$
\theta(\omega, \eta, \gamma z)=c(\omega, \eta, \gamma) \theta(\omega, \eta, z),\qquad \forall \gamma\in \G, 
$$
where $c(\omega, \eta, \cdot):  \G\to \C_\infty^\times$ is a homomorphism 
that factors through $\bG$. The divisor of $\theta(\omega, \eta, \cdot)$ is $\G$-invariant 
and, as a divisor on $X_0(\fn)(\C_\infty)$, equals $[\omega]-[\eta]$, where $[\omega]$ 
is the class of $\omega\in \overline{\Omega}$ in $\G\bs \overline{\Omega}=X_0(\fn)(\C_\infty)$. 

For a fixed $\alpha\in \G$, the function $u_\alpha(z)=\theta(\omega, \alpha \omega, z)$ 
is holomorphic and invertible on $\Omega$. Moreover, $u_\alpha$ is independent of the choice of 
$\omega\in \overline{\Omega}$, and depends only on the class $\bar{\alpha}$ of $\alpha$ in $\bG$. 
Let $c_\alpha(\cdot)=c(\omega, \alpha \omega, \cdot)$ be the multiplier of $u_\alpha$. It induces a   
pairing  
\begin{align*}
\bG\times\bG &\to \Fi^\times\\
(\alpha, \beta) &\mapsto c_\alpha(\beta)
\end{align*}
which is bilinear, symmetric, and 
\begin{align}\label{eqMoPaInf}
&\langle \cdot, \cdot \rangle: \bG\times\bG \to \Z\\
\nonumber &\langle \alpha, \beta\rangle =\ord_{\infty}(c_\alpha(\beta))
\end{align}
is positive definite. 
One of the main result of \cite{GR} is that there is an exact sequence 
\begin{equation}\label{eqGRs}
0\to\bG\xrightarrow{\alpha\mapsto c_\alpha(\cdot)}\Hom(\bG, \C_\infty^\times)\to J_0(\fn)(\C_\infty)\to 0. 
\end{equation}

One can define Hecke operators $T_\fp$ as endomorphisms of $\bG$ 
in purely group-theoretical terms as some sort of Verlagerung (see \cite[(9.3)]{GR}). 
These operators then also act on the torus $\Hom(\bG, \C_\infty^\times)$ 
through their action on the first argument $\bG$. By \cite[(3.3.3)]{GR} and \cite{GN}, 
there is a canonical isomorphism 
\begin{equation}\label{eqj}
j:\bG\xrightarrow{\sim}\cH_0(\fn,\Z)
\end{equation}
which is compatible with the action of Hecke operators. Through this 
construction, the Hecke algebra $\T(\fn)$ in Definition \ref{defHA} acts faithfully on $\bG$ 
and $\Hom(\bG, \C_\infty^\times)$. 
The sequence (\ref{eqGRs}) is compatible with the action of $\T(\fn)^0$ on its three terms; see \cite[(9.4)]{GR}.  

Assume $\fn$ is square-free. 
The matrix  (\ref{ALmatrix}) representing the 
Atkin-Lehner involution $W_\fm$ for $\fm|\fn$ is in the normalizer of $\G$ in $\GL_2(\Fi)$, 
and the induced involution of $X_0(\fn)_{\Fi}$ does not depend on the choice of this matrix. 
In terms on the moduli problem, the involution $W_\fm$ on $X_0(\fn)$ is given by 
$$
W_\fm: (\phi, C_\fn)\mapsto (\phi/C_\fm, (\phi[\fm]+C_{\fn/\fm})/C_\fm),
$$ 
where $C_\fm$ and $C_{\fn/\fm}$ are the subgroups of $C_\fn$ of order $\fm$ and $\fn/\fm$, respectively. 


\section{Component groups}

Let $\fn\lhd A$ be a non-zero ideal. Let $J:=J_0(\fn)$ and $\cJ$ denote the N\'eron model of $J$ over $\p^1_{\F_q}$. Let $\cJ^0$ 
denote the relative connected component of the identity of $\cJ$, that is, the largest open 
subscheme of $\cJ$ containing the identity section which has connected fibres.  
The \textit{group of connected components} (or \textit{component group}) of $J$ at a place $v$ 
of $F$ is $\Phi_v:=\cJ_{\F_v}/\cJ_{\F_v}^0$. This is a finite abelian group 
equipped with an action of the absolute Galois group $G_{\F_v}$. The homomorphism 
$\wp_v: J(F_v^\un)\to \Phi_v$ obtained from the composition 
$$
\wp_v: J(F_v^\un)=\cJ(\cO_v^\un)\to \cJ_{\F_v}(\overline{\F}_v)\to \Phi_v
$$
will be called the \textit{canonical specialization map}. 

 Assume $\fp\lhd A$ is a prime not dividing $\fm$. 
 Then the curve $X_0(\fm)_{\F_\fp}$ is smooth. We call a point $P\in X_0(\fm)(\overline{\F}_\fp)$ 
 \textit{supersingular} if it corresponds to the isomorphism class of a pair $(\phi, C_\fm)$ 
 with $\phi$ supersingular over $\overline{\F}_\fp$. For a Drinfeld $A$-module $\phi$ over $\overline{\F}_\fp$  
 given by $\phi_T=\partial(T)+g \tau+\Delta\tau^2$, let $\phi^{(\fp)}: A\to \overline{\F}_\fp\{\tau\}$ 
 be the Drinfeld module given by $\phi^{(\fp)}_T=\partial(T)+g^{|\fp|} \tau+\Delta^{|\fp|}\tau^2$. 
 Since $\partial(A)\subseteq \F_\fp$, we see that $\tau^{|\fp|}\phi_a=\phi^{(\fp)}_a\tau^{|\fp|}$ for all $a\in A$, so 
 $\tau^{|\fp|}$ is an isogeny $\phi\to \phi^{(\fp)}$. Denote the image of 
 $C_\fm$ in $\phi^{(\fp)}$ under $\tau^{|\fp|}$ by $C_\fm^{(\fp)}$. The map 
 from $X_0(\fm)(\overline{\F}_\fp)$ to itself given by $(\phi, C_\fm)\mapsto (\phi^{(\fp)}, C_\fm^{(\fp)})$ 
 restricts to an involution on the finite set of supersingular points; cf. \cite[Thm. 5.3]{GekelerADM}. 

\begin{thm}\label{thmpM}
Assume $\fn=\fp\fm$, with $\fp$ prime not dividing $\fm$. The curve $X_0(\fn)_{F_\fp}$ 
is smooth and extends to a proper flat scheme $X_0(\fn)_{\cO_\fp}$ 
over $\cO_\fp$ such that the special fibre $X_0(\fn)_{\F_\fp}$ 
is geometrically reduced and consists of two irreducible components, both isomorphic to 
$X_0(\fm)_{\F_\fp}$, intersecting transversally at the supersingular points. 
More precisely, a supersingular point $(\phi, C_\fm)$ on the first copy of  $X_0(\fm)_{\F_\fp}$ 
is glued to $(\phi^{(\fp)}, C_\fm^{(\fp)})$ on the second copy. The curve 
$X_0(\fn)_{\F_\fp}$ is smooth outside of the locus of supersingular points. 
Denote by $\Aut(\phi, C_\fm)$ the subgroup of automorphisms of $\phi$ which map $C_\fm$ 
to itself. For a supersingular 
point $P\in X_0(\fm)_{\F_\fp}$ corresponding to $(\phi, C_\fm)$, let $m(P):=\frac{1}{q-1}\#\Aut(\phi, C_\fm)$. 
Then, locally at $P$ for the \'etale topology, $X_0(\fn)_{\cO_\fp}$ is given by the equation $XY=\fp^{m(P)}$. 
\end{thm}
\begin{proof}
This is proven in \cite[$\S$5]{Uber} for $\fn=\fp$, but the proof easily extends to this 
more general case. 
\end{proof}

We will compute the component group $\Phi_\fp$ using a classical theorem of Raynaud, 
but first we need to determine the number of singular points 
on $X_0(\fn)_{\F_\fp}$, and the integers $m(P)\geq 1$ defined in Theorem \ref{thmpM}. We call $m(P)$ 
the \textit{thickness} of $P$. 

Let $\fm=\prod_{1\leq i\leq s}\fp_i^{r_i}$ be the prime decomposition of $\fm$. Define  
$$
R(\fm)=\begin{cases}
1 & \text{if $\deg(\fp_i)$ is even for all $1\leq i\leq s$}\\
0 & \text{otherwise}
\end{cases}
$$
$$
L(\fm)=\#\p^1(A/\fm)=\prod_{1\leq i\leq s}|\fp_i|^{r_i-1}(|\fp_i|+1). 
$$ 
If $\fm=A$, we put $s=0$ and $L(\fm)=R(\fm)=1$.
\begin{lem}\label{lem4.6} The number of supersingular points on $X_0(\fm)_{\F_\fp}$ 
is 
$$
S(\fp, \fm)=
\begin{cases}
\frac{|\fp|-1}{q^2-1}L(\fm) & \text{if $\deg(\fp)$ is even};\\
\frac{|\fp|-q}{q^2-1}L(\fm) + \frac{L(\fm)+q2^{s}R(\fm)}{q+1} & \text{if $\deg(\fp)$ is odd}. 
\end{cases}
$$
The thickness of a supersingular point on $X_0(\fm)_{\F_\fp}$ 
is either $1$ or $q+1$. Supersingular points with thickness $q+1$ can exist only if 
$\deg(\fp)$ is odd and their number is $2^{s}R(\fm)$. 
\end{lem}
\begin{proof} Let $\phi$ be a fixed Drinfeld module of rank $2$ over $\overline{\F}_\fp$. 
Since $\fm$ is assumed to be coprime to $\fp$, the number of distinct cyclic subgroups $C_\fm\subset \phi[\fm]$ 
is $L(\fm)$. If $\Aut(\phi)=\F_q^\times$, then all pairs $(\phi, C_\fm)$ are non-isomorphic, 
since $\F_q^\times$ fixes each of them. Therefore, all these pairs correspond to distinct points on 
$X_0(\fm)_{\F_\fp}$. We know from (\ref{eqAut}) that $\Aut(\phi)\neq \F_q^\times$ 
if and only if $j(\phi)=0$. Since Theorem \ref{thmSSN} gives the number of 
isomorphism classes of supersingular Drinfeld modules, the claim of the lemma follows if 
we exclude the case with $j(\phi)=0$. 

Now assume $j(\phi)=0$. Then $\Aut(\phi)\cong \F_{q^2}^\times$. We can identify 
the set of cyclic subgroups of $\phi$ of order $\fm$ with $\p^1(A/\fm)$. From this perspective, 
the action of $\Aut(\phi)$ on this set is induced from an embedding 
$$
\F_{q^2}^\times \hookrightarrow \GL_2(A/\fm)\cong \Aut(\phi[\fm]), 
$$
with $\GL_2(A/\fm)$ acting on $\p^1(A/\fm)$ in the usual manner. We can decompose 
$$\p^1(A/\fm)\cong \prod_{1\leq i\leq s}\p^1(A/\fp_i^{r_i})$$ with $\GL_2(A/\fm)$ 
acting on $\p^1(A/\fp_i^{r_i})$ via its quotient $\GL_2(A/\fp_i^{r_i})$. The image 
of $\F_{q^2}^\times$ in $\GL_2(A/\fm)$ is a maximal non-split torus in $\GL_2(\F_q)$. 
If the stabilizer $\Stab_{\F_{q^2}^\times}(P)$ of $P\in \p^1(A/\fp_i^{r_i})$ is strictly larger 
than $\F_{q}^\times$, then it is easy to see that in fact $\Stab_{\F_{q^2}^\times}(P)=\F_{q^2}^\times$. 
Moreover, this is possible if and only if $\F_{q^2}\hookrightarrow A/\fp_i^{r_i}$, in which 
case there are exactly $2$ fixed points in $\p^1(A/\fp_i^{r_i})$ 
under the action of $\F_{q^2}$; cf. \cite[p. 695]{GN}. Note that the existence of 
an embedding $\F_{q^2}\hookrightarrow A/\fp_i^{r_i}$ is equivalent to $\deg(\fp_i)$ being even. 
We conclude that $\F_{q^2}^\times$ acting on the set of $L(\fm)$ pairs $(\phi, C_\fm)$ has $2^sR(\fm)$ 
fixed elements, and the orbit of any other element has length $q+1$. This implies that there are 
$$
\frac{L(\fm)-2^sR(\fm)}{q+1}+2^{s}R(\fm)=\frac{L(\fm)+q2^sR(\fm)}{q+1}
$$
points on $X_0(\fm)_{\F_\fp}$ corresponding to $\phi$ with $j(\phi)=0$. Finally, 
by Theorem \ref{thmSSN}, the Drinfeld module with $j(\phi)=0$ is supersingular 
if and only if $\deg(\fp)$ is odd. 
\end{proof}

\begin{thm}\label{thmCG} Assume $\fn=\fp\fm$, with $\fp$ prime not dividing $\fm$.
If $\deg(\fp)$ is odd and $\fm=A$, then 
$$\Phi_\fp\cong \Z/\left((q+1)(S(\fp,A)-1)+1\right)\Z.$$
If $\deg(\fp)$ is even or $\fm$ has a prime divisor of odd degree, then 
$$
\Phi_\fp\cong \Z/S(\fp, \fm)\Z. 
$$ 
Finally, 
if $\deg(\fp)$ is odd, $\fm\neq A$, and all prime divisors of $\fm$ have even degrees, then 
$$
\Phi_\fp\cong \Z/\left((q+1)^2S(\fp, \fm)-q(q+1)2^{s}\right)\Z\bigoplus_{1\leq i\leq 2^{s}-2}\Z/(q+1)\Z. 
$$
$($The isomorphisms above are meant only as isomorphisms of groups, not group schemes, so 
$\Phi_\fp$ can have non-trivial $G_{\F_\fp}$-action.$)$
\end{thm}
\begin{proof}
This follows from Corollary 11 on page 285 in \cite{NM}, combined with Lemma \ref{lem4.6} and Theorem \ref{thmpM}. 
This result for $\fn=\fp$ is also in \cite{Uber}. 
\end{proof}

\begin{prop}\label{prop5.3} Assume $\fn=\fp\fq$ is a product of two distinct primes. Denote 
$$
n(\fp, \fq)=\frac{(|\fp|-1)(|\fq|+1)}{q^2-1}. 
$$
Let $\Phi_\fp(\F_\fp)$ be the subgroup of $\Phi_\fp$ fixed by $G_{\F_\fp}$. 
If $\deg(\fp)$ is odd and $\deg(\fq)$ is even, then 
$$
\Phi_\fp\cong \Z/(q+1)^2n(\fp, \fq)\Z,\qquad \Phi_\fp(\F_\fp)\cong \Z/(q+1)n(\fp, \fq)\Z. 
$$
Otherwise, $$\Phi_\fp(\F_\fp)=\Phi_\fp \cong \Z/n(\fp, \fq)\Z. $$  
\end{prop}
\begin{proof} To simplify the notation, denote $X=X_0(\fp\fq)_{\cO_\fp}$ and $k=\F_\fp$.  
Let $\widetilde{X}\to X$ be the minimal resolution of $X$. 
We know that $X_k$ 
consists of two irreducible components, both isomorphic to $X_0(\fq)_k$. If $\deg(\fp)$ is even or $\deg(\fq)$ is odd, then $X=\widetilde{X}$. 
On the other hand, if $\deg(\fp)$ is odd and $\deg(\fq)$ is even, then there are 
two points $P$ and $Q$ of thickness $q+1$. 
To obtain the minimal regular model one performs a sequence of $q$ blows-ups at $P$ and $Q$. 
With the notation of \cite[$\S$4.2]{PapikianJNT}, let $E_1, \dots, E_q$ 
and $G_1, \dots, G_q$ be the chain of projective lines resulting from these blow-ups at $P$
and $Q$, respectively. Let $Z$ and $Z'$ denote the two copies of $X_0(\fq)_k$, 
with the convention that $E_1$ and $G_1$ intersect $Z$. Let $B(\widetilde{X}_k)$ 
be the free abelian group generated by the set of geometrically irreducible components of $\widetilde{X}_k$. 
Let $B^0(\widetilde{X}_k)$ be the subgroup of degree-$0$ elements in $B(\widetilde{X}_k)$. 
The theorem of Raynaud that we mentioned earlier gives an explicit description of $\Phi_\fp$ 
as a quotient of $B^0(\widetilde{X}_k)$; cf. \cite[$\S$4.2]{PapikianJNT}. 
Let $\Frob_\fp: \alpha\to \alpha^{|\fp|}$ be the usual topological generator of $G_k$. 
The Frobenius $\Frob_\fp$ naturally acts 
on the geometrically irreducible components of $\widetilde{X}_k$, and therefore on $B^0(\widetilde{X}_k)$. 
Since $X_0(\fq)_k$ is defined over $k$, the action of $\Frob_\fp$ fixes $z:=Z-Z' \in B^0(\widetilde{X}_k)$. 
If $\deg(\fp)$ is even or $\deg(\fq)$ is odd, then $z$ generates $B^0(\widetilde{X}_k)$, so 
$G_k$ acts trivially on $\Phi_\fp$, which by Theorem \ref{thmCG} is isomorphic to $\Z/n(\fp, \fq)\Z$.  

From now on we assume $\deg(\fp)$ is odd and $\deg(\fq)$ is even. The claim $\Phi_\fp\cong \Z/(q+1)^2n(\fp, \fq)\Z$ 
follows from Theorem \ref{thmCG}. 
Let $\phi$ be given by $\phi_T=T+\tau^2$. Note that $j(\phi)=0$ and $\phi$ 
is defined over $k$, so $\Frob_\fp$ acts on the set of cyclic subgroups 
$C_\fq\subset \phi$. 
Denote $A'=\F_{q^2}[T]$. It is easy to see that $A'\subseteq \End(\phi)$. 
Since $\fq=\fq_1\fq_2$ splits in $A'$ into a product 
of two irreducible polynomials of the same degree, we have 
$$
\phi[\fq]\cong A'/\fq\cong A'/\fq_1\oplus A'/\fq_2\cong A/\fq\oplus A/\fq\cong \phi[\fq_1]\oplus \phi[\fq_2]. 
$$ 
The above decomposition is preserves under the action of $\F_{q^2}$. 
In particular, $\Aut(\phi, \phi[\fq_i])\cong \F_{q^2}^\times$. On the other hand, 
since $\fp$ has odd degree, this decomposition is not defined over $k$ and 
$\Frob_\fp(\phi[\fq_1])=\phi[\fq_2]$. We conclude that $P=(\phi, \phi[\fq_1])$, 
$Q=(\phi, \phi[\fq_2])$ and $\Frob_\fp(P)=Q$. Since the action of $\Frob_\fp$ preserves the 
incidence relations of the irreducible components of $\widetilde{X}_k$, 
we have $\Frob_\fp(E_i)=G_i$, $1\leq i\leq q$. 
Following the notation of Theorem 4.1 in \cite{PapikianJNT}, let $e_q=E_q-Z'$ and $g_q=G_q-Z'$. According 
to that theorem, the image of $e_q$ in $\Phi_\fp$ generates $\Phi_\fp$, 
and in the component group we have $g_q=-\left((q+1)(S(\fp, \fq)-2)+1\right)e_q$. 
Thus, the action of $\Frob_\fp$ on $\Phi_\fp$ in terms of the generator $e_q$ 
is given by $\Frob_\fp(e_q)=-\left((q+1)(S(\fp, \fq)-2)+1\right)e_q$. This implies that 
$ae_q\in \Phi_\fp(\F_\fp)$ if and only if 
$$
ae_q=\Frob_\fp(ae_q)=-a\left((q+1)(S(\fp, \fq)-2)+1\right)e_q,
$$  
which is equivalent to $a\left((q+1)S(\fp, \fq)-2q\right)e_q=0$. 
Hence $a$ must be a multiple of $q+1$, and  
$$
\Phi_\fp(\F_\fp)=\langle (q+1)\rangle\cong \Z/(q+1)n(\fp, \fq)\Z. 
$$
\end{proof}

The existence of rigid-analytic uniformization (\ref{eqGRs}) of $J$ 
implies that $J$ has totally degenerate reduction at $\infty$, i.e., 
$\cJ^0_{\F_\infty}$ is a split algebraic torus over $\F_\infty$. 
The problem of explicitly describing $\Phi_\infty$ is closely related to the problem of 
describing $\G_0(\fn)\bs\sT$ together with the stabilizers of its edges; cf. \cite[$\S$5.2]{PapikianJNT}. 
Since the graph $\G_0(\fn)\bs\sT$ becomes very complicated 
as $|\fn|$ grows, no explicit description of $\Phi_\infty$ is known in general. 
On the other hand, $\Phi_\infty$ has a description in terms of the uniformization of $J$. 
Let $\langle\cdot, \cdot \rangle: \bG\times \bG\to \Z$ 
be the pairing (\ref{eqMoPaInf}). 
This pairing is bilinear, symmetric, and positive-definite, so it induces an injection 
$\iota: \bG\to \Hom(\bG, \Z)$, $\gamma\mapsto \langle \gamma, \cdot\rangle$. 
Identifying $\bG$ with $\cH_0(\fn,\Z)$ via (\ref{eqj}), 
we get an injection $\iota: \cH_0(\fn,\Z)\to \Hom(\cH_0(\fn,\Z), \Z)$. 
The Hecke algebra $\T(\fn)$ acts 
on $\Hom(\cH_0(\fn,\Z), \Z)$ through its action on the first argument. The action of $\T(\fn)$ on $J$, 
by the N\'eron mapping property, canonically extends to an action on $\cJ$. Thus, $\T(\fn)$ 
functorially acts on $\Phi_\infty$. 

\begin{thm}\label{thmCGinf} The absolute Galois group $G_{\F_\infty}$ acts trivially on $\Phi_\infty$, 
and there is a $\T(\fn)^0$-equivariant exact sequence 
$$
0\to \cH_0(\fn,\Z)\xrightarrow{\iota}\Hom(\cH_0(\fn,\Z), \Z)\to \Phi_\infty\to 0. 
$$
\end{thm} 
\begin{proof}
The exact sequence is the statement of Corollary 2.11 in \cite{Analytical}. The $\T(\fn)^0$-equivariance 
of this sequence follows from the $\T(\fn)^0$-equivariance of (\ref{eqGRs}). 
The fact that $G_{\F_\infty}$ acts trivially on $\Phi_\infty$ is a consequence of $\cJ^0_{\F_\infty}$ 
being a split torus. 
\end{proof}


\section{Cuspidal divisor group} The \textit{cuspidal divisor group} $\cC(\fn)$ of $J:=J_0(\fn)$ 
is the subgroup of $J$ generated by the classes of divisors $[c]-[c']$, where $c, c'$  
run through the set of cusps of $X_0(\fn)_F$. 

\begin{thm}\label{thm6.1}
$\cC(\fn)$ is finite and if $\fn$ is square-free then $\cC(\fn)\subset J(F)_\tor$. 
\end{thm}
\begin{proof} This theorem is due to Gekeler \cite{GekelerIJM}, where 
the finiteness of $\cC(\fn)$ is proven for general congruence subgroups over 
general function fields. Since the proof is fairly short, 
for the sake of completeness, we give a sketch. The space $\cH_0(\fn, \C)$ has an interpretation 
as a space of automorphic cusp forms (cf. \cite{GR}), so 
the Ramanujan-Petersson conjecture over function fields (proved by Drinfeld) 
implies that the eigenvalues of $T_\fp$ are algebraic integers of absolute 
value $\leq 2\sqrt{|\fp|}$ for any prime $\fp\nmid \fn$. This implies that the operator 
$\eta_\fp=T_\fp-|\fp|-1$ is invertible on $\cH_0(\fn, \C)$. Hence 
$\eta_\fp: \cH_0(\fn, \Z)\to \cH_0(\fn, \Z)$ is injective with finite cokernel. 
Now using the $T_\fp$ equivariance of (\ref{eqGRs}) and (\ref{eqj}), we 
see that $\eta_\fp: J\to J$ is an isogeny.  Assume 
for simplicity that $\fn$ is square-free. In the proof of Lemma \ref{lemPlog} 
we showed that $\eta_\fp([c]-[c'])=0$. Hence $\cC(\fn)\subseteq \ker(\eta_\fp)$, 
which is finite. Finally, when $\fn$ is square-free, all the cusps of $X_0(\fn)_F$ 
are rational over $F$, so the cuspidal divisor group is in $J(F)$; see \cite[Prop. 6.7]{Invariants}. 
\end{proof}

\begin{prop}\label{prop6.3} With notation of $\S\ref{ssES}$, 
let $\mu(\fn)$ be the largest integer $\ell$ such 
that there exists an $\ell$-th root of $\Delta/\Delta_{\fn}$ in $\cO(\Omega)^{\times}$. 
Then $\mu(\fn) = (q-1)^2$ if $\deg \fn$ is odd and $\mu(\fn) = (q-1)(q^2-1)$ if $\deg \fn$ is even. 
Let $D_{\fn}$ be a $\mu(\fn)$-th root of $\Delta/\Delta_{\fn}$. 
There exists a character $\omega_{\fn}: \Gamma_0(\fn) \rightarrow \F_q^{\times}$ such that for each $\gamma \in \Gamma_0(\fn)$,
$$D_{\fn}(\gamma z)= \omega_{\fn}(\gamma) D_{\fn}(z).$$
\end{prop}
\begin{proof}
See Corollary 3.5 and 3.21 in \cite{Discriminant}. 
\end{proof}

If $\fn = \prod_i \fp_i^{r_i} \lhd A$ is the prime decomposition of $\fn$, 
define a character $\chi_{\fn}: \Gamma_0(\fn) \rightarrow \F_q^{\times}$ by the 
following: for $\gamma = \begin{pmatrix}a&b\\ c&d\end{pmatrix} \in \Gamma_0(\fn)$,
$$\chi_{\fn}(\gamma):= \prod_i N_i(d \bmod \fp_i)^{-r_i},$$
where $N_i: (A/\fp_i A)^{\times} \rightarrow \F_q^{\times}$ is the norm map.
Then $\omega_{\fn} = \chi_{\fn} \cdot \det^{\deg(\fn)/2}$ if $\deg(\fn)$ is even and $\omega_{\fn} = \chi_{\fn}^2 \cdot \det^{\deg(\fn)}$ if $\deg(\fn)$ is odd.
In particular, the order of $\omega_{\fn}$ is $q-1$ when $\fn$ is square-free (cf.\ \cite{Discriminant} Proposition 3.22).
By Proposition \ref{prop6.3}, we immediately get:

\begin{cor}\label{cor6.4}
${}$

\begin{itemize}
\item[(1)] 
$D_{\fn}':=D_{\fn}^{q-1}$ is a meromorphic function on the Drinfeld modular curve $X_0(\fn)$ satisfying
$$(D_{\fn}')^{\frac{\mu(\fn)}{q-1}} = \frac{\Delta}{\Delta_{\fn}}.$$
\item[(2)] Given two ideals $\fm$ and $\fn$ of $A$,
we set $D_{\fn,\fm}(z):= D_{\fn}(z) / D_{\fn}(\fm z)$, $\forall z \in \Omega$.
Then $$(D_{\fn,\fm})^{\mu(\fn)} = \frac{\Delta \Delta_{\fm \fn}}{\Delta_{\fn}\Delta_{\fm}}= (D_{\fm,\fn})^{\mu(\fm)}$$ 
and for every $\gamma \in \Gamma_0(\fm\fn)$, we have 
$$D_{\fn,\fm}(\gamma z) = D_{\fn,\fm}(z).$$
In other words, $D_{\fn,\fm}$ and $D_{\fm,\fn}$ can be viewed as meromorphic functions on $X_0(\fm\fn)$.
\end{itemize}
\end{cor}

\begin{rem}\label{rem6.5}
Take two coprime ideals $\fm$ and $\fn$ of $A$. By Corollary \ref{cor6.4} (1), the $(q-1)$-th roots of 
$(\Delta \Delta_{\fm})/(\Delta_{\fn} \Delta_{\fm\fn})$ always exist in the function field of $X_0(\fm\fn)_{\C_\infty}$. 
In fact, we can find a $(q^2-1)$-th root of $(\Delta \Delta_{\fm})/(\Delta_{\fn} \Delta_{\fm\fn})$ 
when $\deg \fn$ or $\deg \fm\fn$ is even. Indeed, we notice that
$$\frac{\Delta(z) \Delta_{\fm}(z)}{\Delta_{\fn}(z) \Delta_{\fm\fn}(z)}
=\frac{\Delta(z)}{\Delta_{\fn}(z)}\cdot \frac{\Delta(z')}{\Delta_{\fn}(z')} 
=\frac{\Delta(z)}{\Delta_{\fm\fn}(z)} \cdot \frac{\Delta(z'')}{\Delta_{\fm\fn}(z'')} \cdot \fm^{-(q^2-1)}.$$
Here $$z' = \fm z, \quad z'' = \frac{ \fm z+1}{b\fm\fn z + a \fm} = \begin{pmatrix}\fm&1\\ b\fm\fn&a\fm\end{pmatrix} \cdot z,$$ 
and $a, b \in A$ such that $a\fm-b\fn = 1$.
In particular, when $q$ is odd, 
$$D_{\fn}''(z):= \left(D_{\fn}(z) \cdot D_{\fn}(mz)\right)^{\frac{q-1}{2}}$$ is a $2(q^2-1)$-th (resp.\ $2(q-1)$-th) root of 
$(\Delta \Delta_{\fm})/(\Delta_{\fn} \Delta_{\fm\fn})$ in the function field of $X_0(\fm\fn)_{\C_\infty}$ when $\deg \fn$ is even (resp.\ odd).
\end{rem}

The following lemma is immediate from the definitions.

\begin{lem} Let $r$ be the van der Put derivative (\ref{eqvdPut}) and $E_\fn\in \cH(\fn,\Z)$ the Eisenstein series (\ref{eqEn}). Then 
$r(D_{\fn}) = E_{\fn}$, and for any two distinct prime ideals $\fp$ and $\fq$,
$$r(D_{\fp,\fq}) =\begin{cases} (q+1) \cdot E_{(\fp,\fq)},& \text{ if $\deg \fp$ is odd and $\deg \fq$ is even,}\\
E_{(\fp,\fq)}, & \text{ otherwise.} \end{cases}$$
\end{lem}

Take two distinct prime ideals $\fp$ and $\fq$,
let $\ell$ be the largest number such that there exists an $\ell$-th root $\xi$ of $(\Delta\Delta_{\fp\fq})/(\Delta_{\fp}\Delta_{\fq})$ in 
the function field of $X_0(\fp\fq)_{\C_\infty}$.
Comparing the first Fourier coefficient $r(\xi)^*(1)$ with 
$r\big((\Delta\Delta_{\fp\fq})/(\Delta_{\fp}\Delta_{\fq})\big)^*(1)$, we must have $\ell | (q-1)(q^2-1)$.
By Corollary 3.17 in \cite{Discriminant}, for every non-zero ideal $\fn$ of $A$
$$\frac{\Delta}{\Delta_{\fn}} = \text{const.} \frac{G_{\fn}^{(q-1)(q^2-1)}}{\Delta^{(q^{\deg \fn}-1)}},$$
where $G_{\fn}$ is a Drinfeld modular form on $\Omega$ such that for any $\gamma = \begin{pmatrix}a&b\\c&d\end{pmatrix} \in \Gamma_0(\fn)$ 
$$G_{\fn}(\gamma z) = \chi_{\fn}(\gamma) (cz+d)^{(|\fn|-1)/(q-1)} G_{\fn}(z).$$
Then
\begin{eqnarray}
\xi^{\ell} = \frac{\Delta(z) \Delta_{\fp\fq}(z)}{\Delta_{\fp}(z)\Delta_{\fq}(z)}
&=& \text{const.}\left(\frac{G_{\fp}(z)}{G_{\fp}(\fq z)}\right)^{(q-1)(q^2-1)} \cdot \left(\frac{\Delta_{\fq}(z)}{\Delta(z)}\right)^{|\fp|-1} \nonumber \\
&=&\text{const.} \left(\frac{G_{\fp}(z)}{G_{\fp}(\fq z)}\right)^{(q-1)(q^2-1)} \cdot D_{\fq}(z)^{-\mu(\fq) (|\fp|-1)}. \nonumber 
\end{eqnarray}
Since $G_{\fp}(z)/G_{\fp}(\fq z)$ is a meromorphic function on $X_0(\fp\fq)_{\C_\infty}$, 
$$D_{\fq}(z)^{\frac{\mu(\fq)(|\fp|-1)}{\ell}} = \text{const.} 
\left(\frac{G_{\fp}(z)}{G_{\fp}(\fq z)}\right)^{\frac{(q-1)(q^2-1)}{\ell}} \cdot \xi^{-1}$$ is also in the function field of $X_0(\fp\fq)_{\C_\infty}$.
Note that the character $\omega_{\fq}$ has order $q-1$.
Set
$$\mu(\fp,\fq):= \begin{cases} (q-1)(q^2-1), & \text{ if $\deg (\fp)$ or $\deg(\fq)$ is even,} \\ (q-1)^2, & \text{ otherwise.} \end{cases}$$
We then have: 

\begin{lem}\label{lem6.7}
 The largest number $\ell$ such that there exists an $\ell$-th roots of 
$\frac{\Delta\Delta_{\fp\fq}}{\Delta_{\fp}\Delta_{\fq}}$ in the function field of $X_0(\fp\fq)_{\C_\infty}$ is $\mu(\fp,\fq)$. 
\end{lem}

From now on we assume that $\fn=\fp\fq$ is a product of two distinct primes. 
In this case, $X_0(\fn)_{F}$ has $4$ cusps, which in the notation of Lemma \ref{lemCusps} 
are $[\infty], \ [1], \ [\fp], \ [\fq]$. Let $c_1, c_\fp, c_\fq\in J(F)$ be the classes of 
divisors $[1]-[\infty]$, $[\fp]-[\infty]$, and $[\fq]-[\infty]$, respectively. 
To simplify the notation, we put $\cC:=\cC(\fp\fq)$. The cuspidal divisor group $\cC$ is generated 
by $c_1$, $c_{\fp}$, and $c_{\fq}$.

\begin{prop} \label{prop6.2} 
The order of $c_1$ and $c_\fq$ is divisible by 
$$
\begin{cases}
\frac{(|\fp|-1)(|\fq|+1)}{q-1} & \text{if $\deg(\fp)$ is odd and $\deg(\fq)$ is even},\\
\frac{(|\fp|-1)(|\fq|+1)}{q^2-1} & \text{otherwise}. 
\end{cases}
$$
\end{prop}
\begin{proof} We use the notation in the proof of 
Proposition \ref{prop5.3}. Let $\wp_\fp: J(F_\fp)\to \Phi_\fp$ be the canonical specialization map. 
This map can be explicitly described as follows. Let $D=\sum_{i} n_i P_i$ 
be a degree-$0$ divisor, where all $P_i\in X(F_\fp)$. Denote by $D$ also 
the linear equivalence class of $D$ in $J(F_\fp)$. Since $X$ and $\widetilde{X}$ are proper, 
$X(F_\fp)=X(\cO_\fp)=\widetilde{X}(\cO_\fp)$. Since $\widetilde{X}$ is regular, each 
$P_i$ specializes to a unique irreducible component $c(P_i)$ of $\widetilde{X}_k$. 
Then $\wp_\fp(D)$ is the image of  $\sum_{i} n_i c(P_i)\in B^0(\widetilde{X}_k)$ in $\Phi_\fp$. 
By Theorem \ref{thmX0}, the cusps reduce to distinct points in the smooth locus of $X_k$. 
The Atkin-Lehner involution $W_\fp$ interchanges the two irreducible components 
$Z$ and $Z'$ of $X_k$. Since $W_\fp([\infty])=[\fq]$, the reductions 
of $[\infty]$ and $[\fq]$ lie on distinct components. On the 
other hand, $W_\fq$ acts on $X_k$ by acting on each 
component $Z$ and $Z'$ separately, without interchanging them. Since 
$W_\fq([\infty])=[\fp]$, the reductions of $[\infty]$ and $[\fp]$ lie on the same component of $X_k$. 
Let $Z'$ be the component containing $[\infty]$ and $[\fp]$. Let $z$ be the image of $Z-Z'$ in $\Phi_\fp$. 
Then $\wp_\fp(c_1)=\wp_\fp(c_\fq)=z$ and $\wp_\fp(c_\fp)=0$. By Theorem 4.1 in \cite{PapikianJNT}, 
$\Phi_\fp/\langle z\rangle\cong \Z/(q+1)\Z$. Now the claim follows from Theorem \ref{thmCG}. 
\end{proof}

\begin{rem} Suppose $\fn=\fp\fm$, with $\fp$ prime not dividing $\fm$. Let $\cC(\fn)(F)$ 
denote the $F$-rational cuspidal subgroup of $J_0(\fn)$.
Generalizing the argument in the proof of Proposition \ref{prop6.2} and using Theorem \ref{thmCG}, 
it is not hard to show that the exponent of the group $\Phi_\fp/\wp_\fp(\cC(\fn)(F))$ divides $(q+1)$. 
In particular, the map $$\wp_\fp: \cC(\fn)(F)_\ell\to (\Phi_\fp)_\ell$$ is surjective for $\ell\nmid (q+1)$. 
\end{rem}

 The orders of $\Delta$ at the cusps 
are known (cf. \cite[(3.10)]{Discriminant}): 
$$
\ord_{[\infty]}\Delta =1, \quad \ord_{[1]}\Delta =|\fp||\fq|, \quad \ord_{[\fp]}\Delta =|\fp|, \quad \ord_{[\fq]}\Delta =|\fq|. 
$$
Using the action of the Aktin-Lehner involutions on the cusps (Lemma \ref{lemCusps}) and on the 
functions $\Delta$, $\Delta_\fp$, and $\Delta_{\fp\fq}$, one obtains the divisors of the following functions on $X_0(\fn)_{\C_\infty}$ 
(cf. \cite{PapikianJNT}): 
\begin{align*}
\text{div}(\Delta/\Delta_\fp) &= (|\fp|-1)\Big(|\fq|([1]-[\infty])- |\fq|([\fp]-[\infty])+ ([\fq]-[\infty])\Big),  \\ 
\text{div}(\Delta/\Delta_\fq) &= (|\fq|-1)\Big(|\fp|([1]-[\infty])+ ([\fp]-[\infty]) - |\fp| ([\fq]-[\infty])\Big), \\ 
\text{div}(\Delta/\Delta_{\fp\fq}) &= (|\fp\fq|-1)([1]-[\infty])+ (|\fq|-|\fp|)([\fp]-[\infty]) + (|\fp|-|\fq|) ([\fq]-[\infty]).
\end{align*}
Hence
\begin{eqnarray}\label{eqn6.4}
\text{div}\left(\frac{\Delta\Delta_{\fp\fq}}{\Delta_\fp\Delta_\fq}\right)
& = & (|\fp|-1)(|\fq|-1)\Big([\infty]+[1]-[\fp]-[\fq]\Big), \\ \label{eqn6.5}
\text{div}\left(\frac{\Delta\Delta_{\fq}}{\Delta_\fp\Delta_{\fp\fq}}\right)
& = & (|\fp|-1)(|\fq|+1)\Big([\infty]+[1]-[\fp]+[\fq]\Big), \\ \label{eqn6.6}
\text{div}\left(\frac{\Delta\Delta_{\fp}}{\Delta_\fq\Delta_{\fp\fq}}\right)
& = & (|\fp|+1)(|\fq|-1)\Big([\infty]+[1]+[\fp]-[\fq]\Big).
\end{eqnarray}

From these equations we get:

\begin{lem}\label{lem6.8}
$$0 = (|\fp|^2-1)(|\fq|^2-1)c_1 = (|\fp|^2-1)(|\fq|-1)c_\fp = (|\fp|-1)(|\fq|^2-1)c_\fq.$$
In particular, the order of $\cC$ is not divisible by $p$.
\end{lem}

Let 
$$c_{(-,-)}:= c_1-c_\fp-c_\fq, \quad c_{(-,+)}:= c_1-c_\fp+c_\fq, \quad c_{(+,-)}:=c_1+c_\fp-c_\fq.$$
The subgroup $\cC'$ of $\cC$ generated by $c_{(-,-)}$, $c_{(-,+)}$, and $c_{(+,-)}$ has index $1$, $2$, or $4$.
Set
$$\epsilon(\fp,\fq):= \begin{cases} q-1, & \text{ if $q$ is even, $\deg \fp$ is odd and $\deg \fq$ is even,} \\
2(q-1), & \text{ if $q$ is odd, $\deg \fp$ is odd, and $\deg \fq$ is even,} \\
q^2-1, & \text{ if (i) $\deg \fp$ and $\deg \fq$ are both odd or} \\ 
 & \text{\quad (ii) $q$ is even and $\deg \fp$ is even,} \\
2(q^2-1), & \text{ otherwise,}
\end{cases}$$
and
$$N_{(-,-)} :=  \frac{(|\fp|-1)(|\fq|-1)}{\mu(\fp,\fq)}, $$
$$
N_{(-,+)} :=  \frac{(|\fp|-1)(|\fq|+1)}{\epsilon(\fp,\fq)}, \qquad
N_{(+,-)} :=  \frac{(|\fp|+1)(|\fq|-1)}{\epsilon(\fq,\fp)}. 
$$

\begin{prop}\label{prop6.9}
The orders of $c_{(-,-)}$, $c_{(-,+)}$, and $c_{(+,-)}$ are $N_{(-,-)}$, $N_{(-,+)}$, and $N_{(+,-)}$, respectively.
The group $\cC'$ is isomorphic to
$$\frac{\Z}{N_{(-,-)}\Z} \times \frac{\Z}{N_{(-,+)}\Z} \times \frac{\Z}{N_{(+,-)}\Z}.$$
\end{prop}



\begin{proof}
Lemma \ref{lem6.7} and Equation (\ref{eqn6.4}) tell us immediately that the order of $c_{(-,-)}$ is $N_{(-,-)}$.
In Remark \ref{rem6.5}, we have constructed $\epsilon(\fp,\fq)$-th root of $\frac{\Delta\Delta_\fq}{\Delta_\fp \Delta_{\fp\fq}}$ 
in the function field of $X_0(\fp\fq)_{\C_\infty}$. Therefore Equation (\ref{eqn6.5}) and (\ref{eqn6.6}) imply that 
$$N_{(-,+)}c_{(-,+)} = N_{(+,-)}c_{(+,-)} = 0.$$
Recall from the proof of Proposition \ref{prop6.2} that 
$\wp_\fp(c_1) = \wp_\fp(c_\fq) = z$ and $\wp_\fp(c_{\fp}) = 0$, where $\wp_\fp: J(F_{\fp}) \rightarrow \Phi_{\fp}$ is the canonical specialization map.
Then $\wp_\fp(c_{(-,+)}) = 2z$, and the order of $c_{(-,+)}$ must be divisible by $N_{(-,+)}$. 
Therefore the order of $c_{(-,+)}$ is $N_{(-,+)}$, which is also the order of $\wp_\fp(c_{(-,+)})$. 
By interchanging $\fp$ and $\fq$ we obtain that the order of $c_{(+,-)}$ is $N_{(+,-)}$. This proves the first claim. 

Now, suppose $c=\alpha_1 c_{(-,-)} + \alpha_2 c_{(-,+)} + \alpha_3 c_{(+,-)} = 0$ 
for $\alpha_1,\ \alpha_2, \ \alpha_3 \in \Z$. Then $\wp_\fp(c) = \alpha_2 \wp_\fp(c_{(-,+)}) = 0$, which implies that $N_{(-,+)} \mid \alpha_2$.
Similarly, we have $N_{(+,-)} \mid \alpha_3$. Hence $c = \alpha_1 c_{(-,-)} = 0$, and $N_{(-,-)} \mid \alpha_1$. 
This implies the second claim. 
\end{proof}

When $q$ is even, Lemma \ref{lem6.8} implies that $\cC = \cC'$. 
Suppose $q$ is odd.
Note that the quotient group $\cC/\cC'$ is generated by $c_\fp$ and $c_\fq$, where $2c_\fp \equiv 2c_\fq \equiv 0 \bmod \cC'$.
When $\deg(\fp)$ and $\deg(\fq)$ are both odd,
we can find meromorphic functions $\varphi_\fp$ and $\varphi_\fq$ on $X_0(\fp\fq)_{\C_\infty}$ such that 
$$\text{div}(\varphi_\fp) = \frac{(|\fp|^2-1)(|\fq|-1)}{(q-1)(q^2-1)} ([\fp]-[\infty]),$$ 
$$
\text{div}(\varphi_\fq) = \frac{(|\fp|-1)(|\fq|^2-1)}{(q-1)(q^2-1)} ([\fq]-[\infty]).$$
Indeed, let
$$\varphi_{\fp}(z):=\frac{ D_\fq(z)^{\frac{|\fp|-q}{q^2-1}} \cdot D_\fq(\fp z)^{\frac{|\fp| q -1}{q^2-1}}}{\Big(D_{\fp\fq}(z)\cdot D_{\fp\fq}(W_\fp z)\Big)^{\frac{|\fp|-1}{q-1}}}, \quad 
\varphi_{\fq}(z):=\frac{ D_\fp(z)^{\frac{|\fq|-q}{q^2-1}} \cdot D_\fp(\fq z)^{\frac{|\fq| q -1}{q^2-1}}}{\Big(D_{\fp\fq}(z)\cdot D_{\fp\fq}(W_\fq z)\Big)^{\frac{|\fq|-1}{q-1}}}$$
(Proposition \ref{prop6.3} implies that $\varphi_\fp$ and $\varphi_\fq$ are $\Gamma_0(\fp\fq)$-invariant.)
We conclude that the orders of $c_\fp$ and $c_\fq$ are odd, and therefore $\cC = \cC'$.

Suppose $q$ is odd and $\deg(\fp) \cdot \deg(\fq)$ is even.
Then Proposition \ref{prop6.2} tells us that the order of $c_{\fp}$ and $c_{\fq}$ are both even.
Thus from the canonical specialization maps $\wp_\fp$ and $\wp_\fq$, we have the following exact sequence:
$$0\longrightarrow \cC' \longrightarrow \cC \longrightarrow \frac{\Z}{2\Z} \times \frac{\Z}{2\Z}\longrightarrow 0.$$
Since $2c_{\fp} = c_{(+,-)} - c_{(-,-)}$ and $2c_{\fq} = c_{(-,+)}-c_{(-,-)}$, 
the order of $c_{\fp}$ is $2\cdot \text{lcm}(N_{(-,-)}, N_{(+,-)})$, and the order of $c_{\fq}$ is $2 \cdot\text{lcm}(N_{(-,-)},N_{(-,+)})$.
From the above discussion, we finally conclude that:

\begin{thm}\label{thm6.10} \hfill 
\begin{itemize}
\item[(1)] The order of $c_\fp$ is $\frac{(|\fp|^2-1)(|\fq|-1)}{(q-1)(q^2-1)}$ and the order of $c_{\fq}$ is $\frac{(|\fp|-1)(|\fq|^2-1)}{(q-1)(q^2-1)}$.
\item[(2)] The odd part of $\cC$ is isomorphic to
$$\frac{\Z}{N_{(-,-)}^{\mathrm{odd}}\Z} \times \frac{\Z}{N_{(-,+)}^{\mathrm{odd}}\Z} \times \frac{\Z}{N_{(+,-)}^{\mathrm{odd}}\Z},$$
where $N^{\mathrm{odd}}$ is the odd part of a positive integer $N$.
\item[(3)] When $q$ is even or $q\cdot \deg(\fp) \cdot \deg(\fq)$ is odd, we have $\cC = \cC'$.
\item[(4)] Suppose $q$ is odd and $\deg(\fp)\cdot \deg(\fq)$ is even. Let $\cC_2$ $($resp.\ $\cC'_2)$ be the $2$-primary part of $\cC$ $($resp.\ $\cC')$. Let $r_1,\ r_2,\ r_3 \in \Z_{\geq 0}$ with $r_1\geq r_2 \geq r_3$ such that 
$$\cC'_2  \cong \Z/2^{r_1}\Z \times \Z/2^{r_2} \Z \times \Z/2^{r_3}\Z.$$
Then $$\cC_2 \cong \Z/2^{r_1+1}\Z \times \Z/2^{r_2+1} \Z \times \Z/2^{r_3}\Z.$$
\end{itemize}
\end{thm}

\begin{example}\label{exampleCxy}
If $q$ is even, then $\cC'(xy)=\cC(xy)\cong \Z/(q+1)\Z\times \Z/(q^2+1)\Z$. If $q$ is odd, then 
$\cC'(xy)\cong \frac{\Z}{\frac{q+1}{2}\Z}\times \frac{\Z}{\frac{q^2+1}{2}\Z}$ and $\cC(xy)\cong\Z/(q+1)\Z\times \Z/(q^2+1)\Z$. 
\end{example}

\begin{example}\label{exampleC22}
Assume $\deg(\fp)=\deg(\fq)=2$. If $q$ is even, then 
$$
\cC'(\fp\fq)=\cC(\fp\fq)\cong 
\frac{\Z}{(q^2+1)\Z} \times \frac{\Z}{(q+1)(q^2+1)\Z}. 
$$
If $q$ is odd, 
$$
\cC'(\fp\fq)\cong \frac{\Z}{(q+1)\Z}\times \frac{\Z}{\frac{q^2+1}{2}\Z}\times \frac{\Z}{\frac{q^2+1}{2}\Z},
$$
$$
\cC(\fp\fq)\cong 
\frac{\Z}{(q^2+1)\Z} \times \frac{\Z}{(q+1)(q^2+1)\Z}.
$$
\end{example}


\section{Rational torsion subgroup}\label{RT}

\subsection{Main theorem} 
Let $\fn\lhd A$ be a non-zero ideal. To simplify the notation in this section we denote $J=J_0(\fn)$. 
Let $\cJ$ denote the N\'eron model of $J$ over $\p^1_{\F_q}$. Let 
$\cT(\fn)$ the torsion subgroup of $J(F)$. 

\begin{lem}\label{lem8.1} Let $\ell$ be a prime not equal to $p$. Then 
$\cT(\fn)_{\ell}$ is annihilated by the Eisenstein ideal $\fE(\fn)$, i.e., 
$(T_{\fp} - |\fp|-1)P = 0$ for every prime ideal $\fp$ not dividing $\fn$ and $P\in \cT(\fn)_{\ell}$.
\end{lem}
\begin{proof} It follows from Theorem \ref{thmX0} that $J$ has good reduction at $\fp\nmid \fn$.  
By the N\'eron mapping property, $\cT(\fn)_\ell$ extends to an \'etale subgroup scheme of $\cJ$. 
This implies that there is a canonical injective homomorphism 
$\cT(\fn)_\ell\hookrightarrow \cJ_{\F_\fp}(\F_\fp)$. The action of $\T(\fn)$ on $J$ canonically 
extends to an action on $\cJ$. 
Since the reduction map commutes with the action of $\T(\fn)$, it is enough to 
show that $(T_{\fp} - |\fp|-1)$ annihilates $\cT(\fn)_{\ell}$ over $\F_\fp$. Let $\Frob_\fp$ 
be the Frobenius endomorphism of $J_{\F_\fp}$. The Hecke operator $T_\fp$ 
satisfies the Eichler-Shimura relation: 
$$
\Frob_\fp^2-T_\fp\cdot \Frob_\fp+|\fp|=0. 
$$
Since $\Frob_\fp$ acts trivially on $\cT(\fn)_\ell$, the claim follows. 
\end{proof}

\begin{lem}\label{lemTPhi}
Suppose $\ell$ is a prime not dividing $q(q-1)$. 
There is a natural injective homomorphism $\cT(\fn)_\ell\hookrightarrow \cE_{00}(\fn, \Z/\ell^r\Z)$ 
for any $r \in \Z_{\geq 0}$ with $\ell^r \geq \#(\Phi_{\infty,\ell})$. 
\end{lem}
\begin{proof}
Since $J$ has split toric reduction at $\infty$, $\cJ^0_{\F_\infty}(\F_\infty)\cong \prod_{i=1}^g\F_q^\times$, 
where $g=\dim(J)$. Under the assumption that $\ell$ does not divide $q(q-1)$, we see that $\cT(\fn)_\ell$ 
has trivial intersection with $\cJ^0_{\F_\infty}$, so 
the canonical specialization $\cT(\fn)_\ell\to (\Phi_{\infty})_\ell$ is injective.  Since this 
map is $\T(\fn)$-equivariant, by Lemma \ref{lem8.1} we get an injection $\cT(\fn)_\ell\to (\Phi_{\infty})_\ell[\fE]$. 
Fix some $r \in \Z_{\geq 0}$ with $\ell^r \cdot (\Phi_{\infty})_\ell = 0$. Multiplying the 
sequence in Theorem \ref{thmCGinf} by $\ell^r$ and applying the snake lemma, we get an injection 
$(\Phi_{\infty})_\ell \hookrightarrow \cH_{00}(\fn,\Z/\ell^r\Z)$. 
Since this map is again $\T(\fn)$-equivariant, restricting to the 
kernels of $\fE(\fn)$, we get $(\Phi_{\infty})_\ell[\fE(\fn)] \hookrightarrow \cE_{00}(\fn,\Z/\ell^r\Z)$. 
Therefore, $\cT(\fn)_\ell\hookrightarrow \cE_{00}(\fn, \Z/\ell^r\Z)$ as was required to show. 
\end{proof}

\begin{thm}\label{thm8.3}
Suppose $\fn = \fp\fq$ is a product of two distinct primes $\fp$ and $\fq$. Let $s_{\fp, \fq}=\gcd(|\fp|+1, |\fq|+1)$. 
If $\ell$ does not divide $q(q-1)s_{\fp,\fq}$, then $\cC(\fn)_{\ell} = \cT(\fn)_{\ell}$. 
\end{thm}

\begin{proof} First of all, since $\fn$ is square-free, $\cC(\fn)$ is rational over $F$, so 
$\cC(\fn)_\ell\subseteq \cT(\fn)_\ell$; see Theorem \ref{thm6.1}. 
Next, note that $\ell$ is odd by our assumption, so by Proposition \ref{prop6.9} 
$$
\cC(\fn)_{\ell}  = \cC(\fn)^{\mathrm{odd}}_{\ell} \cong \Z/\ell^{r_{(-,-)}}\Z \times \Z/\ell^{r_{(-,+)}}\Z \times \Z/\ell^{r_{(+,-)}}\Z,
$$
where $r_{(\pm,\pm)}:= \ord_{\ell}(N_{(\pm,\pm)})$. 
Since $\ell$ is coprime to $q(q-1)s_{\fp,\fq}$, Theorem \ref{thm3.9} and Lemma \ref{lem3.4} give the inclusion 
\begin{eqnarray}
\cE_{00}(\fn,\Z/\ell^r\Z) &=& \cE'_{00}(\fn,\Z/\ell^r\Z) \nonumber \\
&\subseteq & \cE'_{0}(\fn,\Z/\ell^r\Z)
\cong \Z/\ell^{r_{(-,-)}}\Z \times \Z/\ell^{r_{(-,+)}}\Z \times \Z/\ell^{r_{(+,-)}}\Z. \nonumber
\end{eqnarray}
Finally, by Lemma \ref{lemTPhi} we have an injection 
$\cC(\fn)_{\ell} \hookrightarrow \cT(\fn)_{\ell} \hookrightarrow \cE_{00}(\fn,\Z/\ell^r \Z)$. Comparing the orders of these groups, 
we conclude that 
$$
\cC(\fn)_{\ell} = \cT(\fn)_{\ell} = \cE_{00}(\fn,\Z/\ell^r \Z)= \cE'_{0}(\fn,\Z/\ell^r\Z). 
$$
\end{proof}

\begin{rem}\label{rem7.4} The previous proof shows that $\cE_{00}(\fp\fq,\Z/\ell^r \Z)= \cE'_{0}(\fp\fq,\Z/\ell^r\Z)$. 
A generalization of this argument gives the equality $\cE_{00}(\fp\fq,R) = \cE'_{0}(\fp\fq,R)$ 
for any coefficient ring in which $q-1$ and $s_{\fp,\fq}$ are invertible.
\end{rem}

\begin{cor}
If $\ell$ does not divide $q(|\fp|^2-1)(|\fq|^2-1)$, then $\cT(\fp\fq)_\ell=0$. 
\end{cor}

\begin{lem}\label{lem8.2}
Assume $\fn=\fm\fp$ is square-free, $\fp$ is prime, and  
$\deg (\fm) \leq 2$. Then $\cT(\fn)_{p} = 0$.
\end{lem}
\begin{proof} If $\deg(\fm)\leq 2$, then 
$X_0(\fm)_{\F_\fp}\cong \p^1_{\F_\fp}$. It follows from Theorem \ref{thmpM} and \cite[p. 246]{NM} 
that $\cJ^0_{\F_{\fp}}$ is a torus. Since $J$ has toric reduction at $\fp$, the $p$-primary 
torsion subgroup $\cT(\fn)_p$ injects into $\Phi_\fp$; see Lemma 7.13 in \cite{Pal}.  
Finally, as is easy to see from Theorem \ref{thmCG}, if $\fn$ is square-free, the order of $\Phi_\fp$ is coprime to $p$. Thus, 
$\cT(\fn)_{p}=0$. 
\end{proof}

\subsection{Special case}
Here we focus on the case $\fn = xy$ and prove that $\cC(\fn)=\cT(\fn)$. 
To simplify the notation, let $\cC:=\cC(\fn)$ and $\cT:=\cT(\fn)$.  
By Theorem \ref{thm8.3} and Lemma \ref{lem8.2}, we know that $\cC_\ell=\cT_\ell$ for any $\ell \nmid (q-1)$. 
Let $$N = (q+1)(q^2+1).$$ By Corollary \ref{corT/E}, $\T(xy)/\fE(xy)\cong \Z/N\Z$, so $N\in \fE(xy)$. 
On the other hand, $\fE(xy)$ annihilates $\cT_\ell$ for $\ell\neq p$. Therefore, the exponent of $\cT_\ell$ 
divides $N$. Since $\mathrm{gcd}(q-1,N)$ divides $4$, 
$\cT_{\ell} = 0$ when $\ell \mid (q-1)$ is an odd prime. Therefore, we are reduced to showing that $\cC_2=\cT_2$ 
in the case when $q$ is odd. To prove this we will use the fact that $X_0(xy)_F$ is hyperelliptic. 

\vspace{0.1in}

Let $C$ be a hyperelliptic curve of genus $g$ over a field $F$ of characteristic not equal to $2$. Let $B\subset C(\bar{F})$ be the set 
of fixed points of the hyperelliptic involution of $C$. The cardinality of $B$ is $2g+2$. 
Let $J$ be the Jacobian variety of $C$. 
Let $\cG$ be the set of subsets of \textbf{even} cardinality of $B$ modulo the equivalence 
relation defined by $S_1\sim S_2$ if $S_1=S_2$ or $S_1=B- S_2$ (=  
the complement of $S_2$). Denote the equivalence class of $S\subset B$ by $[S]$. Define 
a binary operation on $\cG$ by 
$$
[S_1]\circ [S_2]=[(S_1\cup S_2)-(S_1\cap S_2)]. 
$$
Then $\cG$ is an abelian group isomorphic to $(\Z/2\Z)^{2g}$ (the identity is $[\emptyset]$). 
There is an obvious action of the absolute Galois group $G_F$ on $\cG$, induced from the action on $B$.  

\begin{thm} There is a canonical isomorphism 
$J[2]\cong \cG$ of Galois modules. 
\end{thm}
\begin{proof}
Follows from Lemma 2.4 in \cite{MumfordTataII}, page 3.32. 
\end{proof}

Now let $F=\F_q(T)$ with $q$ odd. Let $f(T)$ be a monic square-free polynomial 
of degree $3$. Let $\fn$ be the ideal in $A$ generated by $f(T)$. 
The Drinfeld modular curve $X_0(\fn)_F$ is hyperelliptic with the Atkin-Lehner involution $W_\fn$
being the hyperelliptic involution; see \cite{SchweizerHE}. Let $e\in \F_q^\times$ be a non-square. Denote 
$K_1=F(\sqrt{f(T)})$, $K_2=F(\sqrt{ef(T)})$, $\cO_1=A[\sqrt{f(T)}]$, 
$\cO_2=A[\sqrt{ef(T)}]$. Note that since $\infty$ does not split in $K_i/F$, $\cO_i$ 
is a maximal order in $K_i$ ($i=1,2$). 

\begin{thm}
The fixed points of $W_\fn$ on $X_0(\fn)_F$ correspond to the isomorphism classes of 
Drinfeld modules with complex multiplication by $\cO_1$ or $\cO_2$. 
\end{thm}
\begin{proof}
See (3.5) in \cite{Uber}.  
\end{proof}

\begin{thm}
Let $K$ be an imaginary quadratic extension of $F$, i.e., $\infty$ does not split in $K/F$. 
Let $\cO$ be the integral closure of $A$ in $K$. Let $\cH$ be the Hilbert class field of $K$, 
i.e., the maximal abelian unramified extension of $K$ in which $\infty$ splits completely. 
\begin{enumerate}
\item $\Gal(\cH/K)\cong \Pic(\cO)$.  
\item There is a unique irreducible monic polynomial $H_K(z)\in A[z]$ whose roots are the $j$-invariants of various 
non-isomorphic rank-$2$ Drinfeld $A$-modules over $\C_\infty$ with CM by $\cO$. 
\item The degree of $H_K(z)$ is the class number of $\cO$, and $\cH$ is its splitting field. 
\item The field $F'=F[z]/H_K(z)$ is linearly disjoint from $K$, and $\cH$ is the composition 
of $F'$ and $K$. 
\end{enumerate}
\end{thm}
\begin{proof}
Follows from Corollary 2.5 in \cite{GekelerADM}. 
\end{proof}

Denote $H_i(z)=H_{K_i}(z)$, $F_i=F[z]/H_i(z)$, $h_i=\#\Pic(\cO_i)$, and $\cH_i$ 
the Hilbert class field of $K_i$ ($i=1,2$). Let $B$ be the set 
of fixed points of $W_\fn$ on $X_0(\fn)$. Since the action of $G_F$ on $X_0(\fn)(\bar{F})$ 
commutes with the action of $W_\fn$, the set $B$ is stable under the action of $G_F$. 
The previous two theorems imply that $B$ can be identified with the set of 
roots of the polynomial $H_1(z)H_2(z)$ compatibly with the action of $G_F$. Let $J:=J_0(\fn)$. 

\begin{lem}\label{Prop2torCase1}
If $f(T)$ is irreducible, then $J(F)[2]=0$. 
\end{lem}
\begin{proof}
Let $X$ be the smooth, projective curve over $\F_q$ with function field $K_1$. It is 
well-known that there is an exact sequence 
$$
0\to \mathrm{Jac}_X(\F_q)\to \Pic(\cO_1)\to \Z/d_\infty\Z\to 0,
$$ 
where $d_\infty$ is the degree of $\infty$ on $X$. Since $f(T)$ has degree $3$, $\infty$ 
ramifies in $K_1/F$, so $d_\infty=1$ and $X$ is isomorphic to the elliptic curve $E_1$ defined by 
the equation $Y^2=f(T)$. Thus, $\Pic(\cO_1)\cong E_1(\F_q)$. Similarly, 
$\Pic(\cO_2)\cong E_2(\F_q)$, where $E_2$ is defined by $Y^2=ef(T)$. 
In particular, 
$$
h_1+h_2=\#E_1(\F_q)+\#E_2(\F_q)=2q+2. 
$$
The last equality follows from the observation that for any $\alpha\in \F_q$ either $f(\alpha)=0$, 
in which case we get one $\F_q$-rational point on both $E_1$ and $E_2$, or exactly one of 
$f(\alpha)$, $ef(\alpha)$ is a square in $\F_q^\times$, in which case we get two 
$\F_q$-rational on one of the elliptic curves and no points on the other. 

Now suppose $f(T)$ is irreducible. Then $f(T)$ has no $\F_q$-rational roots, and therefore $E_1(\F_q)[2]=O$. 
This implies that $h_1$ and $h_2$ are both odd. Denote the set of roots of $H_1$ (resp. $H_2$) 
by $B_1$ (resp. $B_2$), so that $B= B_1\cup B_2$. Note that $B_1$ and $B_2$ are stable 
under the action of $G_F$, but have no non-trivial $G_F$ stable subsets. Since $\# B_1$ and $\#B_2$ 
are odd, by Mumford's theorem the only possibility for having an $F$-rational $2$-torsion on $J$ 
is to have a subset $S\subset B$ of order $q+1$ such that $\sigma S=S$ or $\sigma S=B-S$ for any $\sigma\in G_F$.  
Denote $S_1=S\cap B_1$. Without loss of generality we can assume that $S_1\neq \emptyset, B_1$. 
We must have $\sigma S_1=S_1$ or $\sigma S_1=B_1-S_1$ for any $\sigma\in G_F$.  
But $\# S_1$ and $\#(B_1-S_1)$ have different parity, and therefore $\sigma S_1=S_1$ 
for any $\sigma\in G_F$, which is a contradiction. 
\end{proof}

\begin{rem} Lemma \ref{Prop2torCase1} also follows from Theorem 1.2 in \cite{Pal}. 
Indeed, if $f(T)$ is irreducible, then according to \textit{loc. cit.}, $J(F)_\tor\cong \Z/(q^2+q+1)\Z$, 
so $J(F)[2]=0$. 
\end{rem}

\begin{lem}\label{Prop2torCase12}
If $f(T)$ decomposes into a product $f_1(T)f_2(T)$, where $f_2(T)$ is irreducible 
of degree $2$, then $J(F)[2]\cong \Z/2\Z\oplus \Z/2\Z$. 
\end{lem}
\begin{proof} We retain the notation of the proof of Lemma \ref{Prop2torCase1}. 
The first part of the proof of that lemma implies that $\Pic(\cO_i)[2]\cong \Z/2\Z$, 
since both elliptic curves $E_1$ and $E_2$ have exactly one non-trivial $\F_q$-rational $2$-torsion 
point corresponding to $(\alpha, 0)$, where $f_1(\alpha)=0$. In particular, $h_1$ and $h_2$ 
are even, and $\cH_i/K_i$ has a unique quadratic subextension. We conclude that $[B_1]=[B_2]$ 
is $F$-rational. 

The other $F$-rational $2$-torsion points on $J$ are in bijection with the disjoint decompositions 
$B_1=R_1\cup R_2$  and $B_2=R_1'\cup R_2'$ such that $\# R_1=\# R_2$, $\# R_1'=\# R_2'$,  
and for any $\sigma \in G_F$ either 
$$
\sigma R_1 =R_1\text{ and }\sigma R_1' =R_1' 
$$
or 
$$
\sigma R_1 =R_2\text{ and }\sigma R_1' =R_2'. 
$$
In that case, $P=[R_1\cup R_1']=[R_2\cup R_2']$ and $Q=[R_1\cup R_2']=[R_2\cup R_1']$ 
give two distinct non-trivial $2$-torsion points on $J$ which are rational over $F$. 
Note that the subgroup generated by $P, Q$ and $[B_1]$ 
is isomorphic to $\Z/2\Z\oplus \Z/2\Z$.  

On the other hand, such disjoint decompositions are in bijection with the quadratic extensions $L$
of $F$ which simultaneously embed into both $F_1$ and $F_2$. Note that over a quadratic 
subextension $L$ of $F_1/F$ the polynomial $H_1$ decomposes into a product $G_1(z)G_2(z)$ 
of two irreducible polynomials in $L[z]$ of the same degree. In that case, $R_1$ and $R_2$ are the sets 
of roots of $G_1$ and $G_2$, respectively. Now if there are two distinct quadratic subextensions $L$ and $L'$ 
of $F_1$, then $\cH_1=F_1K_1$ also contains two distinct quadratic subextensions. As we indicated 
above, this is not the case, hence there is at most one $L$. 

Consider $L=F(\sqrt{f_2})$. Since $f_2$ is monic of degree $2$, $\infty$ splits in $L/F$. 
Since the only place that ramifies in $L/F$ is the place corresponding to $f_2$, which also 
ramifies in $K/F$ with the same ramification index, $LK$ is a quadratic subextension of $\cH_1/K$. 
Hence the composition $L F_1$ is a subfield of $\cH_1$, and if $L$ and $F_1$ are linearly disjoint over $F$, 
then by comparing the degrees we see that $LF_1=\cH_1$. Since $H_1(z)$ has even degree, $F_1$ 
embeds into the completion $\Fi$. The same is true for $L$. Thus, if $LF_1=\cH_1$, then 
$\cH_1$ embeds into $\Fi$, which is not the case as $K/F$ is imaginary quadratic. 
We conclude that $L$ and $F_1$ cannot be linearly disjoint, and therefore $L$ embeds into $F_1$. 
The same argument works also with $F_2$, and this finishes the proof of the lemma. 
\end{proof}

\begin{thm}\label{thmCT} 
$\cC=\cT\cong \Z/(q+1)\Z\times \Z/(q^2+1)\Z$. 
\end{thm}
\begin{proof} As we already discussed, it is enough to show that $\cC_2=\cT_2$ when $q$ is odd. By Example \ref{exampleCxy}, 
$$
\cC=\langle c_x\rangle\oplus \langle c_y\rangle  \cong \Z/(q+1)\Z\oplus \Z/(q^2+1)\Z, 
$$
and $c_1=c_x+c_y$. 
The component group $\Phi_\infty$ (for the case $\fn=xy$) and the canonical specialization map $\wp_\infty: \cC\to \Phi_\infty$ 
are computed in \cite[Thm. 5.5]{PapikianJNT}: 
$$
\Phi_\infty=\Phi_\infty(\F_\infty)\cong \Z/N\Z, \quad\wp_\infty(c_x)=q^2+1, \quad \wp_\infty(c_y)=-q(q+1).
$$ 
In particular, if we denote $\cC^0=\ker(\wp_\infty:\cC\to \Phi_\infty)$, then $\cC^0\cong \Z/2\Z$ when $q$ is odd.  
On the other hand, by Lemma \ref{Prop2torCase12}, $\cT[2]\cong \Z/2\Z\times \Z/2\Z$. Hence $\cC[2]=\cT[2]$.  

Let $\cT^0:=\ker(\wp_\infty:\cT\to \Phi_\infty)$. As we discussed at the beginning of this section, 
$\cT^0$ is a subgroup of $(\Z/(q-1)\Z)^{\oplus q}$ annihilated by $N$. We have 
$$
\gcd((q-1), N)= 
\begin{cases} 
2 & \text{ if }q\equiv 3\ \mod\ 4\\
4 & \text{ if }q\equiv 1\ \mod\ 4. 
\end{cases}
$$

Assume $q\equiv 3\ \mod\ 4$. Then $\cT^0\subset \cT[2]=\cC[2]$. This implies $\cT^0=\cC^0\cong \Z/2\Z$, and is generated by 
$$
c:=\frac{(q+1)}{2}c_x+\frac{(q^2+1)}{2}c_y. 
$$
In this case $\wp_\infty$ is injective on $\langle c_1\rangle$, and $\wp_\infty(c_1)$ generates $\wp_\infty(\cC)$. 
If $\cT_2\neq \cC_2$, then there is an element $t\in \cT$ such that $2\wp_\infty(t)=\wp_\infty(c_1)$. 
Thus, $2t=c_1+c$ or $2t=c_1$. 
We know from the proof of Proposition \ref{prop6.2} that 
$\wp_y(c_y)=0$, and $\wp_y(c_x)$ generates $\Phi_y(\F_y)\cong \Phi_y\cong \Z/(q+1)\Z$. 
If $2t=c_1$, then $2\wp_y(t)$ generates $\Phi_y$.  
Since $2$ divides $q+1$, the multiplication by $2$ map is not surjective on $\Phi_y$, so we get a contradiction. 
If $2t=c_1+c$, then $2\wp_y(t)=\wp_y(c_x)+\frac{(q+1)}{2}\wp_y(c_x)$, which is still a generator of $\Phi_y$, 
and we again arrive at a contradiction. 

Now assume $q\equiv 1\ \mod\ 4$. In this case 
$$
\cC_2\cong (\Z/(q+1)\Z)_2\times (\Z/(q^2+1)\Z)_2= \Z/2\Z\times \Z/2\Z \cong \cC[2] 
$$  
is generated by $\frac{q+1}{2}c_x$ and $\frac{q^2+1}{2}c_y$. 
If $\cT\neq \cC$, then there is $t\in \cT$ of order $4$ 
such that $2t\in \cC[2]$. If $2t=c$ or $2t=\frac{q+1}{2}c_x$, then applying $\wp_y$ we see that 
$2\wp_y(t)\neq 0$. On the other hand, $\wp_y(t)\in (\Phi_y)_2\cong \Z/2\Z$, 
so $2\wp_y(t)=0$, which is a contradiction. 
Finally, suppose $2t=\frac{q^2+1}{2}c_y$, which 
implies $2\wp_x(t)\neq 0$. 
Since $t$ is a rational point, $\wp_x(t)\in \Phi_x(\F_x)_2$. On the other hand, by Proposition \ref{prop5.3},  
$\Phi_x(\F_x)_2\cong (\Z/(q^2+1)\Z)_2\cong \Z/2\Z$, which again leads to a contradiction. 
(Note that $\Phi_x\cong \Z/(q+1)(q^2+1)\Z$, so here we cannot just rely on Theorem \ref{thmCG}.)
\end{proof}

\begin{cor} $J(F)\cong \Z/(q+1)\Z\times \Z/(q^2+1)\Z$. 
For any $\ell\neq p$, the $\ell$-primary part of the Tate-Shafarevich group $\Sh(J)$ is trivial. 
\end{cor}
\begin{proof} Denote by $L(J, s)$ the $L$-function of $J$; see \cite{KT} for the definition. 
Let $f\in \cH_0(\fn, \C)$ be an eigenform for $\T(\fn)$, normalized by $f^\ast(1)=1$. 
The $L$-function of $f$ is the sum 
$$
L(f, s)=\sum_{\fm \text{ pos. div.}} f^\ast(\fm)q^{-s\cdot\deg(\fm)}
$$ 
over all positive divisors on $\p^1_{\F_q}$; here $s\in \C$.  
Drinfeld's fundamental result \cite[Thm. 2]{Drinfeld} implies that $L(J, s)=\prod_{f} L(f, s)$, 
where the product is over normalized $\T(xy)$-eigenforms in $\cH_0(xy, \C)$. 
It is known that $L(f,s)$ is a polynomial in $q^{-s}$ of degree $\deg(\fn)-3$; cf. \cite[p. 227]{Tamagawa}. 
Thus, $L(J, s)=1$. From the main theorem in \cite{KT} we conclude that $J(F)=\cT$, $\Sh(J)$ 
is a finite group, and the Birch and Swinnerton-Dyer conjecture holds for $J$. The claim 
$J(F)\cong \Z/(q+1)\Z\times \Z/(q^2+1)\Z$ follows from Theorem \ref{thmCT}. The $\ell$-primary part 
of the Birch and Swinnerton-Dyer formula becomes the equality 
$$
1=\frac{(\# \Sh(J)_\ell)(\#\Phi_x(\F_x)_\ell)(\#\Phi_y(\F_y)_\ell)(\#\Phi_\infty(\F_\infty)_\ell)}{(\#\cT_\ell)^2}. 
$$
We know all entries of this formula  from Theorem \ref{thmCT} and its proof, except $\# \Sh(J)_\ell$. 
This implies that $\Sh(J)_\ell=0$. 
\end{proof}

\section{Kernel of the Eisenstein ideal}\label{sKEI}

\subsection{Shimura subgroup}\label{sSS} 
Let $\fn\lhd A$ be a non-zero ideal. Consider the subgroup $\G_1(\fn)$ of $\GL_2(A)$ consisting of matrices 
$$
\G_1(\fn)=\left\{\begin{pmatrix} a & b \\  c & d\end{pmatrix}\in \GL_2(A)\ \bigg|\ a\equiv 1\ \mod\ \fn, c\equiv 0\ \mod\ \fn \right\}.  
$$
The quotient $\G_1(\fn)\bs\Omega$ is the rigid-analytic 
space associated to a smooth affine curve $Y_1(\fn)_{\Fi}$ over $\Fi$; cf. (\ref{eqUnifY}). 
This curve is the modular curve of isomorphism classes of pairs $(\phi, P)$, where $\phi$ is a Drinfeld $A$-module 
of rank $2$ and $P\in \phi[\fn]$ is an element of exact order $\fn$.   
Let $Y_1(\fn)_{F}$ be the canonical model of $Y_1(\fn)_{\Fi}$ over $F$, and $X_1(\fn)_F$ 
be the smooth projective curve containing $Y_1(\fn)_F$ as a Zariski open subvariety. Denote by $J_1(\fn)$ 
the Jacobian variety of $X_1(\fn)_F$. 

To simplify the notation, in this section we denote $\G:=\G_0(\fn)$ 
and $\Delta:=\G_1(\fn)$. 
The map $w:\G\to (A/\fn)^\times$ given by $\begin{pmatrix} a & b \\  c & d\end{pmatrix} \mapsto a\ \mod\ \fn$ 
is a surjective homomorphism whose kernel is $\Delta$. Hence $\Delta$ is a normal subgroup of 
$\G$ and $\G/\Delta\cong (A/\fn)^\times$. One deduces from the action of $\G$ on $\Omega\cup \p^1(F)$ 
an action of $\G$ on $X_1(\fn)_{\C_\infty}$. The group $\Delta$ and the scalar matrices act trivially on 
$X_1(\fn)_{\C_\infty}$, hence we have an action of the group $(A/\fn)^\times/\F_q^\times$ on $X_1(\fn)_{\C_\infty}$. 
This implies that there is a natural morphism $X_1(\fn)_{\C_\infty}\to X_0(\fn)_{\C_\infty}$ 
which is a Galois covering with Galois group $(A/\fn)^\times/\F_q^\times$. This covering 
is in fact defined over $F$, as in terms of the moduli problems it 
is induced by $(\phi, P)\mapsto (\phi, \langle P\rangle)$, 
where $\langle P\rangle$ is the order-$\fn$ cyclic subgroup of $\phi$ generated by $P$. 
By the Picard functoriality we get a homomorphism $\pi: J_0(\fn)\to J_1(\fn)$ defined over $F$, whose 
kernel $\cS(\fn)$ is the \textit{Shimura subgroup} of $J_0(\fn)$. 

\begin{defn}
Let $Q$ be the subgroup of $(A/\fn)^\times$ generated by the elements which 
satisfy $a^2-ta+\kappa=0\ \mod\ \fn$ for some $t\in \F_q$ and $\kappa\in \F_q^\times$. 
Denote $U=(A/\fn)^\times/Q$.  
\end{defn}

\begin{thm}\label{thmSSQ} The Shimura subgroup $\cS(\fn)$, as a group scheme over $\Fi$, 
is canonically isomorphic to the Cartier dual $U^\ast$ of $U$ viewed as a constant group scheme. 
\end{thm}
\begin{proof} 
Denote by $\G^\ab$ the abelianization of $\G$ and let $\bG:=\G^\ab/(\G^\ab)_\tor$ 
be the maximal abelian torsion-free quotient of $\G$. The inclusion $\Delta\hookrightarrow \G$ 
induces a homomorphism $I:\bD\to \bG$ (which is not necessarily injective). There 
is also a homomorphism $V:\bG\to \bD$, the transfer map, such that $I\circ V:\bG\to \bG$ 
is the multiplication by $[\G:\Delta]/(q-1)$; see \cite[p. 71]{GR}. 

First, we note that the homomorphism $V: \bG\to \bD$ is injective with torsion-free cokernel. 
Indeed, by \cite[p. 72]{GR}, there is a commutative diagram 
$$
\xymatrix{
\bG \ar[r]^-{j_\G} \ar[d]^V &  \cH_0(\cT, \Z)^\G \ar[d] \\
\bD \ar[r]^-{j_\Delta} & \cH_0(\cT, \Z)^\Delta
}
$$
where the right vertical map is the natural injection. This last homomorphism 
obviously has torsion-free cokernel. Since by \cite{GN} $j_\G$ and $j_\Delta$ are 
isomorphisms, the claim follows.

Next, by the results in Sections 6 and 7 of \cite{GR}, there is a commutative 
diagram 
$$
\xymatrix{
0\ar[r] & \bG \ar[r] \ar[d]^V &  \Hom(\bG, \C_\infty^\times) \ar[r]\ar[d]^{I^\ast} & J_0(\fn)\ar[r]\ar[d]^{\pi} & 0 \\
0\ar[r] & \bD \ar[r] & \Hom(\bD, \C_\infty^\times) \ar[r] & J_1(\fn) \ar[r]& 0,
}
$$
where the top row is the uniformization in (\ref{eqGRs}), the 
bottom row is a similar uniformization for $J_1(\fn)$ constructed in \cite{GR} for 
an arbitrary congruence group, and the middle vertical map $I^\ast$ is the dual of $I:\bD\to \bG$. 
This diagram gives an exact sequence of group schemes 
$$
0\to (\bG/I\bD)^\ast \to \cS(\fn) \to \bD/V\bG.  
$$
Since there are no non-trivial maps from an abelian variety to a discrete group, 
$\cS(\fn)$ is finite. On the other hand, by the previous paragraph, $\bD/V\bG$ is torsion-free, so 
$$
\cS(\fn)\cong (\bG/I\bD)^\ast. 
$$

Denote by $\G^c$ the commutator subgroup of $\G$. The fact that $\G/\Delta$ is abelian 
implies $\G^c\subset \Delta$. Hence $\G/\Delta\cong \G^\ab/(\Delta/\G^c)$. It follows 
from Corollary 1 on page 55 of \cite{SerreT} that $(\G^\ab)_\tor$ is generated by the images of finite 
order elements of $\G$ in $\G^\ab$. Since $\bG$ 
is the quotient of $\G^\ab$ by the torsion subgroup $(\G^\ab)_\tor$, we conclude that $\bG/I\bD$ 
is the quotient of $\G/\Delta$ by the subgroup generated by the images of finite order elements of $\G$. 

An element $\gamma\in \GL_2(A)$ has finite order if and only if $\Tr(\gamma)\in \F_q$. 
Therefore, if $\gamma\in \G$ has finite order then $w(\gamma)$ 
satisfies the equation $a^2-ta+\kappa=0$, where $t=\Tr(\gamma)\in \F_q$ and $\kappa=\det(\gamma)\in \F_q^\times$. 
Conversely, suppose $\bar{a}\in (A/\fn)^\times$ satisfies $\bar{a}^2-t\bar{a}+\kappa=0$ for some $t\in \F_q$ and $\kappa\in \F_q^\times$. 
Fix some $a\in A$ reducing to $\bar{a}$ modulo $\fn$. Since $a(t-a)\equiv \kappa\ (\mod\ \fn)$, there exists $c\in A$ 
such that $a(t-a)=\kappa+c\fn$. The matrix $\gamma=\begin{pmatrix} a & 1 \\ c\fn & t-a\end{pmatrix}$ has determinant $\kappa$ 
and trace $t$. It is clear that $\gamma\in \G$ is a torsion element, and $w(\gamma)=\bar{a}$. 
We conclude that $\bG/I\bD\cong (A/\fn)^\times/Q$.  
\end{proof}

\begin{rem}
The previous theorem is the function field analogue of Theorem 1 in \cite{LO}. 
\end{rem}

\begin{lem}\label{lem8.4} Assume $\fn$ is square-free, so that the order of $(A/\fn)^\times$ is not 
divisible by $p$ and $\cS(\fn)$ is \'etale over $\Fi$.  
The Shimura subgroup $\cS(\fn)$ extends to a finite flat subgroup scheme of $\cJ^0_{\cO_\infty}$.  
\end{lem}
\begin{proof} 
From the proof of Theorem \ref{thmSSQ}, $\cS(\fn)$ is canonically 
a subgroup of $\Hom(\bG, \C_\infty^\times)$. It is easy to see that $\cS(\fn)$ 
actually lies in $\Hom(\bG, (\cO_\infty^{\un})^\times)$. On the 
other hand, as is implicit in the proof of Corollary 2.11 in \cite{Analytical}, there is a canonical isomorphism 
$\cJ^0(\cO_\infty^\un)\cong \Hom(\bG, (\cO_\infty^{\un})^\times)$. 
\end{proof}

\begin{prop}\label{propSisEis}
Assume $\fn$ is square-free. The Shimura subgroup $\cS(\fn)$, as a group scheme over $F$, 
is an \'etale group scheme whose Cartier dual is canonically 
isomorphic to $U$ viewed as a constant group scheme.  
The endomorphisms $T_\fp-|\fp|-1$ and $W_\fn+1$ of $J_0(\fn)$ annihilate $\cS(\fn)$; 
here $\fp\lhd A$ is any prime not dividing $\fn$. 
\end{prop}
\begin{proof}
The covering $\pi:X_1(\fn)_F\to X_0(\fn)_F$ can ramify only at the elliptic points and the cusps 
of $X_0(\fn)_F$. (By definition, an elliptic point on $X_0(\fn)_{\C_\infty}$ is the image of $z\in \Omega$ 
whose stabilizer in $\GL_2(A)$ is strictly larger than $\F_q^\times$.) The proof of 
Theorem \ref{thmSSQ} shows that $U$ is the Galois group 
of the maximal abelian unramified covering $X_F\to X_0(\fn)_F$ through which $\pi$ 
factorizes. Now \cite[Prop. 6]{LO} implies that $\cS(\fn)$ as a group scheme over $F$ 
is isomorphic to $\Hom(U, \bar{F}^\times)$. 

The Jacobian $J_0(\fn)$ 
has good reduction at $\fp$. Since $\cS(\fn)$ has order coprime to $p$ and is unramified at $\fp$, 
the reduction map injects $\cS(\fn)$ into $J_0(\fn)(\overline{\F}_\fp)$. Let $\Frob_\fp$ 
be the Frobenius endomorphism of $J_0(\fn)_{\F_\fp}$. The Hecke operator $T_\fp$ 
satisfies the Eichler-Shimura relation: 
$$
\Frob_\fp^2-T_\fp\cdot \Frob_\fp+|\fp|=0. 
$$
Since $\cS(\fn)^\ast$ is constant, $\Frob_\fp$ acts on $\cS(\fn)$ 
by multiplication by $|\fp|$. Therefore, the endomorphisms $|\fp|(T_\fp-|\fp|-1)$ 
annihilates $\cS(\fn)$. Since the reduction map commutes with the action of Hecke 
algebra and the multiplication by $|\fp|$ is an automorphism of $\cS(\fn)$,  
we conclude that $T_\fp-|\fp|-1$ annihilates $\cS(\fn)$. 

Note that for $\begin{pmatrix} a & b\\ c & d\end{pmatrix}\in \G$ we have 
$$
\begin{pmatrix} 0 & -1\\ \fn & 0\end{pmatrix}\begin{pmatrix} a & b\\ c & d\end{pmatrix} \begin{pmatrix} 0 & 1/\fn\\ -1 & 0\end{pmatrix}
=\begin{pmatrix} d & -c/\fn\\ -b\fn & a\end{pmatrix}.
$$
Since $d$, up to an element of $\F_q^\times$, 
is the inverse of $a$ modulo $\fn$, this calculation shows that  
$W_\fn$ acts on the group $(\G/\Delta)/\F_q^\times$ by $u\mapsto u^{-1}$, so it acts as $-1$ on $\cS(\fn)$. 
\end{proof}

\begin{thm}\label{thmSSS}
Assume $\fn$ is square-free, and let $\fn=\fp_1\cdots \fp_s$ be its prime decomposition.  As an abelian group, $\cS(\fn)$ 
is isomorphic to 
$$
\begin{cases}
\prod_{i=1}^s\left(\Z \middle/\frac{|\fp_i|-1}{q-1}\Z\right), & \text{if some $\fp_i$ has odd degree}\\
\prod_{i=1}^s\left(\Z \middle/\frac{|\fp_i|-1}{q^2-1}\Z\right), & \text{if $q$ is even and all $\fp_i$ have even degrees}\\
\prod_{i=1}^s\left(\Z \middle/\frac{2(|\fp_i|-1)}{q^2-1}\Z\right)/\mathrm{diag}\left(\Z/2\Z\right), & 
\text{if $q$ is odd and all $\fp_i$ have even degrees}
\end{cases}
$$
where 
$\mathrm{diag}:\Z/2\Z\to \prod_{i=1}^s\left(\Z \middle/\frac{2(|\fp_i|-1)}{q^2-1}\Z\right)$ is 
the diagonal embedding. 
\end{thm}
\begin{proof}
First, we claim that the image of $Q$ under the isomorphism $(A/\fn)^\times \cong \prod_{i=1}^s\F_{\fp_i}^\times$ 
given by 
$$
a \mapsto (a\ \mod\ \fp_1, \cdots, a\ \mod\ \fp_s)
$$
contains the subgroup $\prod_{i=1}^s\F_q^\times$. 
Indeed, let $\alpha\in \F_q^\times$ be arbitrary, and consider 
$(1,\dots, \alpha, \dots, 1)$ with $\alpha$ in the $i$th position. This element 
is in the image of $Q$ since $a\in (A/\fn)^\times$ mapping to it  
satisfies the quadratic equation $(x-\alpha)(x-1)=0$ modulo $\fn$. 

Suppose $a\in Q$ is a zero of  $f(x):=x^2-tx+\kappa$. Then $a$ satisfies the same 
equation modulo $\fp_i$. If $\deg(\fp_i)$ is odd, then $f(x)$ must be reducible over $\F_q$. 
This implies that the image of $a$ in $\F_{\fp_i}^\times$ lies in the subfield $\F_q$. 
On the other hand, since $a\ \mod\ \fp_j$ ($1\leq j\leq s$) satisfies the same reducible quadratic equation, 
the image of $a$ in $\F_{\fp_j}$ also lies in its subfield $\F_q$, and we 
conclude that $Q\cong \prod_{i=1}^s\F_q^\times$. 

Now suppose all $\fp_i$'s have even degrees. 
Let $\alpha\in \F_{q^2}^\times$ be arbitrary. The elements $(\alpha, \dots, \alpha)$ 
and $(\alpha, \dots, \alpha^q, \dots, \alpha)$, with $\alpha^q$ in the $i$th position ($1\leq i\leq s$), are in 
the image of $Q$. Indeed, $a\in (A/\fn)^\times$ which maps to such an element satisfies 
$(x-\alpha)(x-\alpha^q)=0$, and this polynomial has coefficients in $\F_q$. 
Therefore, $(1, \dots, \alpha^{q-1}, \dots, 1)\in Q$. Since $\alpha^{q+1}\in \F_q^\times$, we get that 
$(1, \dots, \alpha^2, \dots, 1)\in Q$. Thus, 
$\mathrm{diag}(\F_{q^2}^\times)\cdot \prod_{i=1}^s(\F_{q^2}^\times)^2$ is a subgroup of $Q$, 
where $(\F_{q^2}^\times)^2=\{\alpha^2\ |\ \alpha\in \F_{q^2}^\times\}$ 
and $\mathrm{diag}: \F_{q^2}^\times\hookrightarrow \prod_{i=1}^s\F_{\fp_i}^\times$ 
is the diagonal embedding. If $f(x)$ is reducible, then $a\in \prod_{i=1}^s\F_q^\times$. 
If $f(x)$ is irreducible, then $a=(\alpha_{i_1}, \dots, \alpha_{i_s})$ with $\alpha_{i_1},\dots, \alpha_{i_s}\in \{\alpha,\alpha^q\}$, where  
$\alpha$, $\alpha^q$ are the roots of $f(x)$. We can write  
$$
a=(\alpha_{i_1}, \dots, \alpha_{i_1})\cdot \left(1, \frac{\alpha_{i_2}}{\alpha_{i_1}},\dots, \frac{\alpha_{i_s}}{\alpha_{i_1}}\right). 
$$
The first element in this product is in $\mathrm{diag}(\F_{q^2}^\times)$ and the second is in $\prod_{i=1}^s(\F_{q^2}^\times)^2$
since each $\alpha_{i_j}/\alpha_{i_1}$ is either $1$ or $\alpha^{\pm(q-1)}$. 
We conclude that $\mathrm{diag}(\F_{q^2}^\times)\cdot \prod_{i=1}^s(\F_{q^2}^\times)^2= Q$. 
Since $\cS(\fn)$ is isomorphic to $(A/\fn)^\times/Q$, the claim follows. 
\end{proof}

\begin{example}
If $\fp\lhd A$ is prime, then $\cS(\fp)$ is cyclic of order $\frac{|\fp|-1}{q-1}$ 
(resp. $\frac{|\fp|-1}{q^2-1}$) is $\deg(\fp)$ if odd (resp. even). This is also proved in \cite{Pal} 
by a different method. 
\end{example}

\begin{example}
$\cS(xy)$ is cyclic of order $q+1$. 
\end{example}

\begin{example}
Assume $\fp$ and $\fq$ are two distinct primes of degree $2$. Then $\cS(\fp\fq)\cong \Z/2\Z$ (resp. $0$) if $q$ is odd (resp. even). 
\end{example}

\begin{defn}
Let $\fn\lhd A$ be a non-zero ideal and denote 
$J=J_0(\fn)$. The \textit{kernel of the Eisenstein ideal} $J[\fE(\fn)]$ is the 
intersection of the kernels of all elements of $\fE(\fn)$ acting on $J(\bar{F})$. Given a prime $\ell\neq p$, 
the action of $\T(\fn)$ on $J$ induces an action on $J[\ell^n]$ ($n\geq 1$) and the $\ell$-divisible group 
$J_\ell:=\underset{\to}{\lim}\ J[\ell^n]$, so one can also define  $J_\ell[\fE(\fn)]$. 
From the proof of Theorem \ref{thm6.1}, we see that $J[\fE(\fn)]$ is a finite group scheme over $F$, as 
$J[\fE(\fn)]\subseteq J[\eta_\fp]$ for any $\fp\nmid \fn$, where $\eta_\fp=T_\fp-|\fp|-1$. 
By Lemma \ref{lem8.1} and Proposition \ref{propSisEis}, if $\fn$ is square-free then 
$\cS(\fn)_\ell$ and $\cC(\fn)_\ell$ are in $J_\ell[\fE(\fn)]$. 
\end{defn}

\begin{defn}
We say that a group scheme $H$ over the base $S$ is \textit{$\mu$-type} if it is finite, flat 
and its Cartier dual is a constant group scheme over $S$; $H$ 
is \textit{pure} if it is the direct sum of a constant and $\mu$-type group schemes; $H$ is \textit{admissible} 
if it is finite, \'etale and has a filtration by groups schemes such that the successive 
quotients are pure. 
\end{defn}

\begin{lem}\label{thm9.2} Assume $\fn=\fp\fq$ is a product of two distinct primes. 
Let $s_{\fp, \fq}=\gcd(|\fp|+1, |\fq|+1)$. 
If $\ell$ does not divide $q(q-1)s_{\fp, \fq}$, then $J_\ell[\fE(\fn)]$ 
is unramified except possibly at $\fp$ and $\fq$, and there is an exact sequence 
of group schemes over $F$ 
$$
0\to \cC(\fn)_\ell\to J_\ell[\fE(\fn)]\to \cM_\ell\to 0,
$$
where $\cM_\ell$ is $\mu$-type and contains $\cS(\fn)_\ell$. 
\end{lem}
\begin{proof} 
Since $\cC(\fn)_\ell$ is fixed by $G_F$ and is invariant under the action of $\T(\fn)$, 
the quotient $\cM_\ell:=J_\ell[\fE(\fn)]/\cC(\fn)_\ell$ is naturally a $\T(\fn)\times G_F$-module . 
From the uniformization sequence (\ref{eqGRs}) and the isomorphism (\ref{eqj}), for any $n\geq 1$
we obtain the exact sequence of $\T(\fn)\times G_{\Fi}$-modules 
$$
0\to D[\ell^n]\to J[\ell^n]\to \cH_{00}(\fn, \Z/\ell^n\Z)\to 0, 
$$
where $D:=\Hom(\bG, \C_\infty^\times)$. Considering the parts annihilated by $\fE(\fn)$, we obtain the exact sequence 
$$
0\to D[\ell^n, \fE(\fn)]\to J[\ell^n, \fE(\fn)]\to \cE_{00}(\fn, \Z/\ell^n\Z).   
$$
From the proof of Theorem \ref{thm8.3} we know that if $\ell$ does not divide $q(q-1)s_{\fp, \fq}$ 
and $n$ is sufficiently large, then $\cC(\fn)_\ell\subseteq J[\ell^n, \fE(\fn)]$ and $\cC(\fn)_\ell$ maps 
isomorphically onto $\cE_{00}(\fn, \Z/\ell^n\Z)$. Hence for $n\gg 0$ the above sequence 
is exact also on the right:
$$
0\to D[\ell^n, \fE(\fn)]\to J[\ell^n, \fE(\fn)]\to \cC(\fn)_\ell\to 0.  
$$
Since the map of $\T(\fn)\times G_{\Fi}$-modules $J[\ell^n, \fE(\fn)]\to \cC(\fn)_\ell$ 
has a retraction, we get the splitting $J_\ell[\fE(\fn)]\cong D_\ell[\fE(\fn)]\oplus \cC(\fn)_\ell$ over $\Fi$. This 
shows that $\cM_\ell\cong D_\ell[\fE(\fn)]$ is $\mu$-type over $\Fi$, and $J_\ell[\fE(\fn)]$ 
is unramified at $\infty$. Since $\cS(\fn)_\ell\subseteq D_\ell[\fE(\fn)]$ (cf. Lemma \ref{lem8.4}), we 
see that $\cS(\fn)_\ell$ is a subgroup scheme of $\cM_\ell$.  

By the N\'eron-Ogg-Shafarevich criterion $J_\ell[\fE(\fn)]$ 
is unramified at all finite places except possibly at $\fp$ and $\fq$. 
Let $\fl\lhd A$ be any prime 
not equal to $\fp$ or $\fq$. Since 
$T_\fl$ acts on $J_\ell[\fE(\fn)]$ by multiplication by $|\fl|+1$, the  
Eichler-Shimura congruence 
relation implies that the action of $\Frob_\fl$ on $J_\ell[\fE(\fn)]$ satisfies the relation 
$(\Frob_\fl-1)(\Frob_\fl-|\fl|)=0$. Now one can use Mazur's argument \cite[Prop. 14.1]{Mazur} to 
show that $J_\ell[\fE(\fn)]$ is admissible over 
$U=\p^1_{\F_q}-\fp-\fq$; cf. the proof of Proposition 10.8 in \cite{Pal}. 
Since the quotient map $J_\ell[\fE(\fn)]\to \cM_\ell$ is compatible with the action of $G_F$ and $\T(\fn)$, $\cM_\ell$ 
is also admissible over $U=\p^1_{\F_q}-\fp-\fq$. On the other hand, $\cM_\ell$ is $\mu$-type 
over $\Fi$, so all Jordan-H\"older components of $\cM_\ell$ over $U$ must be isomorphic to $\mu_\ell$. 
We say that $\fl_1, \fl_2$ is a pair of good primes, if $\fl_i\neq \fp, \fq$, $|\fl_i|-1$ is not 
divisible by $\ell$, and the images of $(\fl_1,\fl_1)$ and $(\fl_2, \fl_2)$ in 
$(\F_\fp^\times/(\F_\fp^\times)^\ell\times \F_\fq^\times/(\F_\fq^\times)^\ell)/\F_q^\times$ 
generate this group; here $(\F_\fp^\times)^\ell$ is the subgroup of $\F_\fp^\times$ 
consisting of $\ell$th powers, and $\F_q^\times$ 
is embedded diagonally into $\F_\fp^\times\times \F_\fq^\times$. The Chebotarev density 
theorem shows that a pair of good primes exists. 
The operator $(\Frob_{\fl_i}-|\fl_i|)(\Frob_{\fl_i}-1)$ annihilates $\cM_\ell$. 
On the other hand, since the semi-simplification 
of $\cM_\ell$ is isomorphic to $(\mu_\ell)^n$ for some $n$ and $\ell$ does not divide $|\fl_i|-1$, 
the operator $\Frob_{\fl_i}-1$ must be invertible on $\cM_\ell$. Therefore, $\cM_\ell$ 
is annihilated by $\Frob_{\fl_i}-|\fl_i|$. This implies that $\cM_\ell$ is $\mu$-type over $F_{\fl_i}$. 
Finally, one can argue as in Proposition 10.7 in \cite{Pal} to conclude that $\cM_\ell$ is 
$\mu$-type over $F$. 
\end{proof}

\begin{rem} 
When $\fn=\fp$ is prime and $\ell$ does not divide $q-1$, then $J_\ell[\fE(\fp)]=\cC(\fp)_\ell\oplus \cS(\fp)_\ell$; 
see \cite{Pal}. As we will see in the next section, for a composite level, $\cM_\ell$ 
can be larger than $\cS(\fn)_\ell$ and the sequence 
in Lemma \ref{thm9.2} need not split over $F$.  
\end{rem}


\subsection{Special case}\label{ssLast} Let $\fp\lhd A$ be a prime of degree $3$. Then 
\begin{enumerate}
\item The rational torsion subgroup $\cT(\fp)$ of $J_0(\fp)$ is equal to the cuspidal divisor group $\cC(\fp)\cong \Z/(q^2+q+1)\Z$; 
\item The maximal $\mu$-type \'etale subgroup scheme of $J_0(\fp)$ is the Shimura subgroup $\cS(\fp)\cong (\Z/(q^2+q+1)\Z)^\ast$;
\item The kernel $J[\fE(\fp)]$ is everywhere unramified and has order $(q^2+q+1)^2$. 
\end{enumerate}
Here (1) and (2) follow from the main results in \cite{Pal}, and (3) follows from \cite{Pal} and \cite{PapikianMRL}. 

We proved that (see Theorem \ref{thmCT})
$$
\cT(xy)=\cC(xy)\cong \Z/(q+1)\Z\times \Z/(q^2+1)\Z,
$$
so the analogue of (1) is true for $\fn=xy$, although $\cC(xy)$ is not cyclic if $q$ is odd. 
Interestingly, even for small composite levels such as $\fn=xy$, the analogues of (2) and (3) 
are no longer true. For simplicity, we will omit the level $xy$ from notation. 
First of all, $J_\ell[\fE]$ can be ramified. To see this, let $q=2^s$ with 
$s$ odd, $x=T+1$, $y=T^2+T+1$. Consider the elliptic curve $E$ over $F$ given by the Weierstrass equation 
$$
E: Y^2+TXY+Y=X^3+X^2+T. 
$$
This curve embeds into $J$ according to \cite[Prop. 7.10]{PapikianJNT}. There 
is an exact sequence of $G_F$-modules 
$$
0\to \Z/5\Z\to E[5]\to \mu_5\to 0. 
$$
The component group of $E$ at $y$ is trivial, so $E[5]$ is ramified at $y$ and the above sequence 
does not split over $F$. On the other hand, from the Eichler-Shimura congruence relations we see 
that $E[5]\subseteq J_5[\fE]$. Thus, $J_5[\fE]$ is ramified at $y$. 

Next, the Shimura subgroup $\cS\cong (\Z/(q+1)\Z)^\ast$ has smaller order than the cuspidal divisor group, 
and $\cS$ is not the maximal $\mu$-type \'etale subgroup scheme of $J$ when $q$ 
is odd. Indeed, $\cC[2]\cong \Z/2\Z\times \Z/2\Z$ is $\mu$-type but is not cyclic. 

In the remainder of this section we will show that even though (2) and (3) fail for the $xy$-level, 
these properties do not fail ``too much''. 

\begin{prop}
Assume $q$ is odd. Then 
$\cS_2+\cC[2]\cong \cS_2\times \Z/2\Z$ 
is the maximal $2$-primary $\mu$-type subgroup of $J(\bar{F})$. 
\end{prop}
\begin{proof}
First, we show that the canonical specialization map $\wp_y: \cS\to \Phi_y$ is an isomorphism. 
Since $\cS$ and $\Phi_y$ have the same order, it is enough to show that the induced morphism 
$\cJ_0(xy)_{\overline{\F}_y}^0\to \cJ_1(xy)_{\overline{\F}_y}^0$ 
on the connected components of the identity of the special fibres 
of the N\'eron models of $J_0(xy)$ and $J_1(xy)$ at $y$ is injective. 
Here we argue as in Proposition 11.9 in \cite{Mazur}. Let $\mathrm{Gr}_0$ 
be the dual graph of $X_0(xy)_{\overline{\F}_y}$; this graph has 
two vertices corresponding to the two irreducible components of $X_0(xy)_{\overline{\F}_y}$ 
joined by $q+1$ edges corresponding to the supersingular points. Let $\mathrm{Gr}_1$ 
be the dual graph of the special fibre of the model of $X_1(xy)$ over $\cO_y$ 
constructed in \cite{Shastry}. The underlying reduced subscheme $X_1(xy)_{\overline{\F}_y}^\mathrm{red}$ 
has two irreducible components intersecting at the supersingular points. 
Since the Jacobian $J_0(xy)$ has toric reduction 
at $y$, there is a canonical isomorphism (cf. \cite[p. 246]{NM})
$$
\cJ_0(xy)_{\overline{\F}_y}^0\cong H^1(\mathrm{Gr}_0, \Z)\otimes \gm_{m,\overline{\F}_y}. 
$$
Similarly, the toric part of $\cJ_1(xy)_{\overline{\F}_y}^0$ is canonically isomorphic to 
$H^1(\mathrm{Gr}_1, \Z)\otimes \gm_{m,\overline{\F}_y}$. There is an obvious map 
$\mathrm{Gr}_1\to \mathrm{Gr}_0$ and our claim reduces to showing that this map 
induces a surjection $H_1(\mathrm{Gr}_1, \Z)\to H_1(\mathrm{Gr}_0, \Z)$. This is clear 
from the description of the graphs $\mathrm{Gr}_0$ and $\mathrm{Gr}_1$. 

Let $\cM_2$ be the maximal $2$-primary $\mu$-type subgroup of $J(\bar{F})$. It is clear that 
$\cS_2+\cT[2]\subseteq\cM_2$ and $\cM_2[2]=\cT[2]$. 
Similar to the notation in the proof of Theorem \ref{thmCT}, let 
$\cM_2^0$ denote the intersection of $\cM_2$ with $\Hom(\bG, \C_\infty^\times)$. 
Note that $\cS_2\subseteq \cM_2^0$. Moreover, we must have $\cS[2]= \cM^0_2[2]$ 
since $\cM_2^0[2]\subseteq \cT^0[2]=\cC^0[2]=\cS[2]$. Thus, if $\cS_2\neq \cM^0_2$, then 
there exists $P\in \cM^0_2$ such that $2P$ generates $\cS_2\cong (\Z/(q+1)\Z)_2^\ast$. 
Since $\wp_y:\cS\to \Phi_y$ is an isomorphism, $2\wp_y(P)$ 
generates $\Z/(q+1)\Z$. On the other hand, $2$ divides $q+1$, so we get a contradiction. We conclude that 
there is an exact sequence 
$$
0\to \cS_2\to \cM_2\to \Z/2^n\Z\to 0  
$$
for some $n\geq 1$, where $\Z/2^n\Z$ is the image of $\cM_2$ in $\Phi_\infty\cong \Z/(q^2+1)(q+1)\Z$ (see Corollary \ref{corPhiInf}). 
Since $\cM_2[2]=\cT[2]\cong \Z/2\Z\times \Z/2\Z$, the above sequence splits as a sequence of abelian groups. 
Moreover, since $\cS_2$ is $G_F$-invariant and $\cM_2^\ast$ is constant, the sequence splits as a sequence of 
$F$-groups schemes: $\cM_2\cong \cS_2\oplus \Z/2^n\Z$.
If $q\equiv 3\ (\mod\ 4)$, then $\Z/2^n\Z$ is $\mu$-type if and only if $n\leq 1$, 
which implies $\cM_2= \cS_2+\cT[2]$. If $q\equiv 1\ (\mod\ 4)$ and $n\geq 2$, then $\cT=\cC$ 
contains a point of order $4$, which is a contradiction. 
\end{proof}

\begin{figure}
\begin{tikzpicture}[semithick, node distance=1.5cm, inner sep=.5mm, vertex/.style={circle, fill=black}]

\node[vertex] (1){};
  \node[vertex] (2) [above of=1] {};
  \node[vertex] (3) [above of=2] {};
  \node (4) [right of=3] {};
  \node[vertex] (5) [right of=4] {};
  \node[vertex] (6) [below of=5] {};
  \node[vertex] (7) [below of=6] {};
  \node (8) [right of=2] {$\vdots$};
  
  \draw (1) -- (2) -- (3) -- node [below] {$\varphi_q$} (5)-- (6)--(7)-- node [above] {$\varphi_1$} (1);
  \draw (2) edge[bend right] node [below] {$\varphi_{2}$} (6) (2) edge[bend right=70]  (6); 
   \draw (2) edge[bend left] (6) (2) edge[bend left=70] node [below] {$\varphi_{q-1}$} (6);
  
\end{tikzpicture}
\caption{$(\G_0(xy)\bs \sT)^0$}\label{Fig5}
\end{figure}

The point of the previous proposition is that even though $\cS_2$ is not the maximal $2$-primary 
$\mu$-type subgroup of $J$, it is not far from it.  
Next, we will show that under a reasonable assumption $\cS_\ell$ 
is the maximal $\ell$-primary $\mu$-type subgroup of $J$ for any odd $\ell\neq p$, and 
the order of $J_\ell[\fE]$ is $(\# \cC_\ell)^2$ for any $\ell\neq p$, which is of course very similar to (2) and (3). 

We start by showing that $\Phi_\infty$ is annihilated by $\fE$. This almost follows 
from Theorem 5.5 in \cite{PapikianJNT}, since according to that theorem 
the canonical specialization map $\wp_\infty: \cC\to \Phi_\infty$ is an isomorphism if $q$ is even and has cokernel $\Z/2\Z$ 
if $q$ is odd. On the other hand, $\wp_\infty$ is $\T$-equivariant and $\cC$ is annihilated by $\fE$. Below 
we give an alternative direct argument which works for any characteristic and does not use 
the cuspidal divisor group $\cC$. 

Choose the cycles $\varphi_1,...,\varphi_q$ as a basis of $H_1(\Gamma_0(xy)\ \bs \sT ,\Z) \cong \cH_0(xy,\Z)$; see Figure \ref{Fig5}. 
We assume that all these cycles are oriented counterclockwise.
Let $\psi_1,...,\psi_q$ be the dual basis of $\text{Hom}(\cH_0(xy,\Z),\Z)$. 
The map $\iota$ in Theorem \ref{thmCGinf} sends $\varphi \in \cH_0(xy,\Z)$ to 
$\sum\limits_{i=1}^q (\varphi,\varphi_i) \psi_i$, where for $1\leq j \leq i \leq q$
$$
(\varphi_i,\varphi_j) = \begin{cases} q+2, & \text{if $i =j= 1$ or $q$,}\\ 
2, & \text{if $1<i =j< q$,}\\
-1, & \text{if $j=i-1$,} \\ 
0, & \text{otherwise.}
\end{cases}
$$
This follows from \cite[Thm. 2.8]{Analytical} and a calculation of the stabilizers of edges 
of $(\G_0(xy)\bs \sT)^0$ (see the proof of Proposition \ref{propT(xy)}). 
Hence the component group $\Phi_{\infty}$ can be explicitly described as follows:
It is the quotient of the free abelian group $\bigoplus_{i=1}^q \Z \psi_i$ by the relations
\begin{eqnarray}
(q+2)\psi_q - \psi_{q-1} &=& 0, \nonumber \\
-\psi_{i+1} + 2 \psi_{i} - \psi_{i-1} & = & 0, \quad \text{ for } 1\leq i \leq q-1, \nonumber \\
-\psi_{2} + (q+2)\psi_1 & = & 0. \nonumber 
\end{eqnarray}
From the last $q-1$ relations we get
$$\psi_{i} = \big((i-1)q + i\big) \psi_1, \quad 1\leq i \leq q.$$
On the other hand, by the first relation, $\psi_{q-1} = (q+2) \psi_q = q^2(q+2)\psi_1$.
Therefore $ q^2(q+2)\psi_1 = (q^2-q-1)\psi_1$ and 
$$(q^2+1)(q+1)\psi_1 = 0.$$
We conclude that

\begin{cor}\label{corPhiInf}
$\Phi_{\infty} \cong \Z/(q^2+1)(q+1)\Z$ is generated by the image of $\psi_1$.
\end{cor}

Multiplication by $N=(q^2+1)(q+1)$ on the exact sequence in Theorem \ref{thmCGinf} and the snake lemma 
give an embedding $\delta_N: \Phi_{\infty} \hookrightarrow \cH_{00}(xy,\Z/N\Z)$.

\begin{lem}\label{lem5.5}
Let $\bar{\psi}_1$ be the image of $\psi_1$ in $\Phi_{\infty}$. Then
$\delta_N(\bar{\psi}_1) = -(q+1)E_x + (q^2+1)E_y + q E_{(x,y)}$, the generator of $\cE_{00}(xy,\Z/N\Z)$.
In particular, the component group $\Phi_{\infty}$ is annihilated by the Eisenstein ideal.
\end{lem}

\begin{proof}
Note that 
$$N \psi_1 = \sum_{i=1}^q \big((i-1)q+i\big) \iota(\varphi_i) \in \text{Hom}(\cH_0(xy,\Z),\Z).$$ 
Hence $\delta_N(\bar{\psi}_1) = \sum_{i=1}^q \big((i-1)q+i\big) \varphi_i \bmod N \in \cH_{00}(xy,\Z/N\Z)$.
Viewing the cycles $\varphi_1,...,\varphi_q$ as harmonic cochains in $\cH_0(xy,\Z)$, it is observed that
\begin{itemize}
\item[(i)]
$$\varphi_1(a_1) = 1,\ \varphi_1(a_4) = -1,\ \varphi_1(a_6) = 1-q,\ \varphi_1(b_1) = 1,$$
$$\varphi_1(a_2) = \varphi_1(a_3) = \varphi_1(a_5) = \varphi_1(b_u) = 0,\ 2\leq u \leq q-1;$$
\item[(ii)]
$$\varphi_q(a_1) = \varphi_q(a_4) = \varphi_q(a_6) = \varphi_q(b_u) = 0, \ 1\leq u \leq q-2,$$
$$\varphi_q(a_2) = 1,\ \varphi_q(a_3) = -1,\ \varphi_q(a_5) = q-1,\ \varphi_1(b_{q-1}) = -1.$$
\item[(iii)] for $2\leq j \leq q-1$,
$$\varphi_j(a_i) = \varphi_j(b_u) = 0,\ 1\leq i \leq 6,\ u \neq j-1, j;\quad  \varphi_j(b_j) = -\varphi_j(b_{j-1}) = 1.$$
\end{itemize}
Here $a_i,\ b_u$ are the oriented edges of the graph $(\G_0(xy) \bs \sT)^0$ in Figure \ref{Fig4}.
Therefore $\delta_N(\bar{\psi}_1)(a_1) = 1$ and $\delta_N(\bar{\psi}_1)(b_u) = -(q+1)$, $1\leq u \leq q-1$.

On the other hand, let $f = -(q+1)E_x+(q^2+1)E_y+qE_{(x,y)} \in \cE_{00}(xy, \Z/N\Z)$.
From the Fourier expansion of $E_x$, $E_y$, and $E_{(x,y)}$, we also get
$$f(a_1) = 1 \quad \text{and}\quad f(b_u) = -(q+1) \ \text{ for } 1\leq u \leq q-1.$$
Since every harmonic cochain in $\cH_{00}(xy,\Z/N\Z)$ is determined uniquely by its values at $a_1$ and $b_u$ for $1\leq u \leq q-1$, we get $\delta_N(\bar{\psi}_1) = f$ and the proof is complete.
\end{proof}

\begin{defn} Let $\fn\lhd A$ be a non-zero ideal and $\ell$ a prime number. Let 
$\T(\fn)_\ell:=\T(\fn)\otimes_\Z \Z_\ell$. Given a maximal ideal $\fM\lhd \T(\fn)_\ell$, let 
$\T(\fn)_\fM$ be the completion of $\T(\fn)_\ell$ at $\fM$. 
We say that $\T(\fn)_\fM$ is a \textit{Gorenstein} $\Z_\ell$-algebra if $\T(\fn)_\fM^\vee:=\Hom_{\Z_\ell}(\T(\fn)_\fM, \Z_\ell)$ 
is a free $\T(\fn)_\fM$ module of rank $1$. 
In \cite{Pal}, following Mazur \cite{Mazur}, P\'al proved that $\T(\fp)_\fM$ is 
Gorenstein for the maximal ideals containing $\fE(\fp)$. 
\end{defn}

Let $\fE_\ell$ be the ideal generated by $\fE(xy)$ in $\T_\ell$. We know that 
$\T_\ell/\fE_\ell\cong \Z_\ell/N\Z_\ell$; see Corollary \ref{corT/E}. Thus, $N$ annihilates $J_\ell[\fE]$. 

\begin{prop}\label{prop9.5} Assume $\ell\mid N$. The finite group-scheme $J_\ell[\fE]$ is unramified at $\infty$.  
If $\T_\fM$ is Gorenstein for $\fM=(\fE_\ell, \ell)$, then there is an exact sequence of $G_{\Fi}$-modules
$$
0\to (\Z_\ell/N\Z_\ell)^\ast\to J_\ell[\fE]\to \Z_\ell/N\Z_\ell\to 0. 
$$
\end{prop}
\begin{proof} The argument in the proof of Lemma \ref{thm9.2} 
and the isomorphism of Hecke modules $(\Phi_\infty)_\ell\cong \cE_{00}(xy, \Z/\ell^n\Z)$  
that we just proved ($n\gg 0$) give the exact sequence 
$$
0\to D_\ell[\fE]\to J_\ell[\fE]\to (\Phi_\infty)_\ell\to 0.
$$
Since $J[\ell^n]^{I_\infty}\cong D[\ell^n]\times (\Phi_\infty)[\ell^n]$, we see that $J_\ell[\fE]$ 
is unramified at $\infty$. Next, using Proposition \ref{propT(xy)}, (\ref{eqGRs}) and (\ref{eqj}), we have 
\begin{align*}
D_\ell[\fE] &\cong \Hom(\cH_0(xy, \Z_\ell)/\fE_\ell\cH_0(xy, \Z_\ell), \C_\infty^\times)\\ 
&\cong \Hom(\T_\ell^\vee/\fE_\ell \T_\ell^\vee, \C_\infty^\times)\cong \Hom(\T_\fM^\vee/\fE_\ell \T_\fM^\vee, \C_\infty^\times). 
\end{align*}
If $\T_\fM$ is Gorenstein, then 
$$
\Hom(\T_\fM^\vee/\fE_\ell\T_\fM^\vee, \C_\infty^\times)\cong \Hom(\T_\fM/\fE_\ell, \C_\infty^\times)\cong 
\Hom(\Z_\ell/N\Z_\ell, \C_\infty^\times). 
$$
\end{proof}

 If $\ell$ is odd and divides $q+1$ then, under the assumption that $\T_\fM$ 
 is Gorenstein, Proposition \ref{prop9.5} implies that $\cS_\ell$ is the maximal $\mu$-type subgroup 
 scheme of $J_\ell$ and $J_\ell[\fE]\cong \cC_\ell\oplus \cS_\ell$ is pure. 
 Indeed, since $\Phi_\infty$ is constant and $\mathrm{gcd}(q+1, q-1)$ divides $2$, 
 the maximal $\mu$-type subgroup scheme of $J_\ell$ specializes to the connected component 
 $\cJ^0_{\F_\infty}$ of the N\'eron model, so must be isomorphic to $(\Z_\ell/(q+1)\Z_\ell)^\ast$. 
 Since $\cS_\ell$ is $\mu$-type, it must be maximal by comparing the orders. The fact that $\T_\fM$ 
 is Gorenstein in this case can be proved by Mazur's Eisenstein descent discussed in Section 10 of \cite{Pal}. 
 Since this fact is not central for our paper, and the proof closely follows the argument in \textit{loc. cit.} 
 we omit the details. 
 
 Now assume $\ell$ is odd and divides $q^2+1$. Assume $\T_\fM$ is Gorenstein. Since $s_{x,y}$ divides $2$, 
 Lemma \ref{thm9.2} and Proposition \ref{prop9.5} imply that there is an exact sequence of $G_F$-modules 
 \begin{equation}\label{eqCEM}
0\to \cC_\ell\to J_\ell[\fE]\to \cM_\ell\to 0, 
\end{equation}
where $\cM_\ell\cong (\Z_\ell/(q^2+1)\Z_\ell)^\ast$. If $J_\ell[\fE]^{I_y}$ is strictly 
larger than $\cC_\ell$, then the Galois representation $\rho_{_\fM}: G_F\to \Aut(J_\ell[\fM])\cong \GL_2(\F_\ell)$ 
is unramified at $y$. On the other hand, $(\Phi_y)_\ell=0$, so $J_\ell[\fM]$ is isomorphic to a $\T_\ell\times G_{\F_y}$-submodule 
of the torus $\cJ^0_{\F_y}$. Now the same argument that proves Proposition 3.8 in \cite{RibetLL} implies 
that $\Frob_y$ acts on $J_\ell[\fM]$ as a scalar $\pm |y|^2$. Hence $\det(\rho_{_\fM})=q^4$. 
On the other hand, the sequence (\ref{eqCEM}) shows that the eigenvalues of $\Frob_y$ are $1$ and $|y|=q^2$. 
Thus, $\det(\rho_{_\fM})=q^2$. This forces $q^4\equiv q^2\ (\mod\ \ell)$, which is a 
contradiction since $\gcd(q^2+1, q^2(q^2-1))$ divides $2$. We conclude that if $\T_\fM$ is Gorenstein 
then the $\ell$-primary $\mu$-type subgroup scheme of $J$ is trivial. 
This ``explains'' why $\cS$ is smaller than $\cC$ -- the missing part corresponds to the ramified 
portion of $J[\fE]$. 

 
\section{Jacquet-Langlands isogeny}\label{sJL}

\subsection{Modular curves of $\sD$-elliptic sheaves}\label{ssDES}
Let $D$ be a division quaternion algebra over $F$. Assume $D$ is split at $\infty$, i.e., $D\otimes_F \Fi\cong \Mat_2(\Fi)$; 
this is the analogue of ``indefinite'' over $\Q$. Let $\fn\lhd A$ be the 
discriminant of $D$. It is known that $\fn$ is square-free with an even number of prime factors. 
Moreover, any $\fn$ with this property is a discriminant of some $D$, and, up to isomorphism, 
$\fn$ determines $D$; cf. \cite{Vigneras}. 

Let $C:=\p^1_{\F_q}$. Fix a
locally free sheaf $\sD$ of $\cO_C$-algebras with stalk at the
generic point equal to $D$ and such that
$\sD_v:=\sD\otimes_{\cO_C}\cO_v$ is a maximal order in
$D_v:=D\otimes_F F_v$ for every place $v$.
Let $S$ be an $\F_q$-scheme. Denote by $\Frob_S$ its Frobenius
endomorphism, which is the identity on the points and the $q$th
power map on the functions. Denote by $C\times S$ the fiber product
$C\times_{\Spec(\F_q)}S$. Let $z:S\to \Spec(A[\fn^{-1}])$ be a morphism of
$\F_q$-schemes. A \textit{$\sD$-elliptic sheaf over $S$}, with pole
$\infty$ and zero $z$, is a sequence $\mathbb{E}=(\cE_i,j_i,t_i)_{i\in \Z}$,
where each $\cE_i$ is a locally free sheaf of $\cO_{C\times
S}$-modules of rank $4$ equipped with a right action of $\sD$
compatible with the $\cO_C$-action, and where
\begin{align*}
j_i &:\cE_i\to \cE_{i+1}\\
t_i &:{^\tau}{\cE}_{i}:=(\mathrm{Id}_C\times \Frob_{S})^\ast \cE_i\to
\cE_{i+1}
\end{align*}
are injective $\cO_{C\times S}$-linear homomorphisms compatible with
the $\sD$-action. The maps $j_i$ and $t_i$ are sheaf modifications
at $\infty$ and $z$, respectively, which satisfy certain conditions,
and it is assumed that for each closed point $P$ of $S$, the
Euler-Poincar\'e characteristic $\chi(\cE_0|_{C\times P})$ is in the
interval $[0,2)$; we refer to \cite[$\S$2]{LRS} for the precise definition. 
Denote by $\Ell^\sD(S)$ the set of isomorphism classes of
$\sD$-elliptic sheaves over $S$. The following theorem can be
deduced from (4.1), (5.1) and (6.2) in \cite{LRS}:

\begin{thm}\label{thmMC}
The functor $S\mapsto \Ell^{\sD}(S)$ has a coarse moduli scheme $X^\fn$, which is proper and smooth 
of pure relative dimension $1$ over $\Spec(A[\fn^{-1}])$. 
\end{thm}

For each prime $\fp\lhd A$ not in $R$, there is a Hecke correspondence $T_\fp$ on $X^\fn$; 
we refer to \cite[$\S$7]{LRS} for the definition. 

\subsection{Rigid-analytic uniformization}\label{ssRAUJ'} Let $\cD$ be a maximal $A$-order in $D$. 
Let $\G^\fn:=\cD^\times$ be the group of units of $\cD$. 
By fixing an isomorphism $D\otimes_F\Fi\cong \Mat_2(\Fi)$, we get an embedding $\G^\fn\hookrightarrow \GL_2(\Fi)$. 
Hence $\G^\fn$ acts on the Drinfeld half-plane $\Omega=\C_\infty-\Fi$. 
We have the following uniformization theorem 
\begin{equation}\label{eqMum}
\G^\fn\bs \Omega\cong (X^\fn_{\Fi})^\an, 
\end{equation}
which follows by applying Raynaud's ``generic fibre'' functor to Theorem 4.4.11 in \cite{BS}. 
The proof of this theorem is only 
outlined in \cite{BS}. Nevertheless, as is shown in \cite[Prop. 4.28]{Spiess}, (\ref{eqMum}) 
can be deduced from Hausberger's version \cite{Hausberger} of the Cherednik-Drinfeld theorem for $X^\fn$. 
An alternative proof of (\ref{eqMum}) is given in \cite[Thm. 4.6]{TaelmanArXiv}. 

Assume for simplicity that $\fn$ is divisible by a prime of even degree. In this case, the normal 
subgroup $\G^\fn_f$ of $\G^\fn$ generated by torsion elements is just the 
center $\F_q^\times$ of $\G^\fn$, cf. \cite[Thm. 5.6]{PapLocProp}. 
This implies that the image of $\G^\fn$ in $\PGL_2(\Fi)$ is a discrete, finitely generated free group, 
and (\ref{eqMum}) is a Mumford uniformization \cite{GvdP}, \cite{Mumford}. 

Denote $\G=\G^\fn/\F_q^\times$ 
and let $\bG$ be the abelianization of $\G$. The group $\G$ acts without inversions on the Bruhat-Tits tree $\sT$, 
and the quotient graph $\G^\fn\bs \sT=\G\bs \sT$ is finite; cf. \cite[Lem. 5.1]{PapLocProp}. Fix a vertex $v_0\in \sT$. 
For any $\gamma\in \G$ there is a unique path  
in $\sT$ without backtracking from $v_0$ to $\gamma v_0$. 
The map $\G\to H_1(\G\bs \sT, \Z)$ which sends $\gamma$ to the homology class of the image of this path in $\G\bs \sT$ 
does not depend on the choice of $v_0$, and induces an isomorphism
\begin{equation}\label{eq_piH}
\bG \cong H_1(\G\bs \sT, \Z).
\end{equation}
Since $\G$ is a free group and $\G\bs \sT$ is a finite graph, 
\begin{align*}
\cH(\fn, \Z)' &:=\cH_0(\sT, \Z)^{\G^\fn}=\cH_0(\sT, \Z)^\G=\cH(\sT, \Z)^\G \\ 
&\cong \cH(\G\bs \sT, \Z)\cong H_1(\G\bs \sT, \Z),   
\end{align*}
cf. Definitions \ref{defnHarmG} and \ref{defnHarmG0}. The space $\cH(\fn, \Z)'$ 
is equipped with a natural action of Hecke operators $T_\fm$ ($\fm\lhd A$), which generate 
a commutative $\Z$-algebra; cf. \cite[$\S$5.3]{Miyake}. 

The map $\langle \cdot, \cdot \rangle: E(\sT)\times E(\sT) \to \Z$ 
$$
\langle e, f\rangle= \left\{
         \begin{array}{ll}
           1 & \hbox{if $f=e$} \\
           -1 & \hbox{if $f=\bar{e}$}\\
           0 & \hbox{otherwise}
         \end{array}
       \right.
$$
induces a $\Z$-valued bilinear symmetric positive-definite pairing on $H_1(\G\bs \sT, \Z)$, 
which we denote by the same symbol. Using (\ref{eq_piH}) we get a pairing 
\begin{align}\label{eqMoPaQ}
\bG \times \bG &\to \Z\\ 
\nonumber \gamma, \delta &\mapsto \langle \gamma, \delta \rangle. 
\end{align}

Let $J^\fn$ denote the Jacobian variety of $X^\fn_F$, and $\cJ^\fn$ denote the N\'eron model of $J^\fn$ over $\cO_\infty$.  
Since $X^\fn$ is a Mumford curve over $\Fi$, $(\cJ^\fn)^0_{\F_\infty}$ is a split algebraic torus. Let 
$M:=\Hom((\cJ^\fn)^0_{\overline{\F}_\infty}, \gm_{m, \overline{\F}_\infty})$ be the 
character group of this torus. By the mapping property of N\'eron models, each endomorphism of $J^\fn$ 
defined over $\Fi$ canonically extends to $\cJ^\fn$, hence acts functorially on $M$. Since $J^\fn$ 
has purely toric reduction, $\End_{\Fi}(J^\fn)$ operates faithfully on $M$. 
By \cite[p. 132]{Mumford}, the graph $\G\bs \sT$ is the dual graph of the special fibre of the minimal regular model of $X^\fn$ 
over $\cO_\infty$. On the other hand, by \cite[p. 246]{NM}, there is a canonical isomorphism $M\cong H_1(\G\bs \sT, \Z)$. 
The Hecke correspondence $T_\fp$ induces an endomorphism of $J^\fn$ by Picard functoriality, 
which we denote by the same symbol. Via the isomorphisms mentioned in this paragraph 
and (\ref{eq_piH}), we get a canonical action of $T_\fp$ on $\bG$. This action agrees with the 
action of $T_\fp$ on $\cH(\fn, \Z)'$.

Using rigid-analytic theta functions, one constructs a symmetric bilinear pairing 
\begin{align*}
\bG\times \bG &\to \Fi^\times\\
\gamma, \delta &\mapsto [\gamma, \delta],
\end{align*}
such that $\ord_\infty[\gamma, \delta]=\langle \gamma, \delta\rangle$; see \cite[Thm. 5]{MD}. 
The main result of \cite{MD} gives a rigid-analytic uniformization of $J^\fn$:
\begin{equation}
0\to\bG\xrightarrow{\gamma\mapsto [\gamma, \cdot]}\Hom(\bG, \C_\infty)\to J^\fn(\C_\infty)\to 0. 
\end{equation}
This sequence is equivariant with respect to the action of $T_\fp$ ($\fp\not \in R$); this was proven by Ryan Flynn \cite{Flynn} 
following the method in \cite{GR}. 

Finally, we have the following quaternionic analogue of Theorem \ref{thmCGinf}. 
Let $\Phi_\infty'$ be the component group of $J^\fn$ at $\infty$.
After identifying $\bG$ with the character group $M$, the pairing (\ref{eqMoPaQ}) 
becomes Grothendieck's monodromy pairing on $M$, so there is an exact sequence (see \cite[$\S\S$11-12]{SGA7} and \cite{PapikianCM})
\begin{equation}\label{eqPhiJ'inf}
0\to \bG\xrightarrow{\gamma\mapsto \langle \gamma, \cdot\rangle}\Hom(\bG, \Z)\to \Phi_\infty'\to 0. 
\end{equation}
This sequence is equivariant with respect to the action of $T_\fp$. 


\subsection{Explicit Jacquet-Langlands isogeny conjecture}\label{ssGJL} Let 
$\fn\lhd A$ be a product of an even number of distinct primes. 
The space $\cH_0(\fn, \Q)$ can be interpreted as a space of automorphic forms on $\G_0(\fn)$; see \cite[$\S$4]{GR}. 
A similar argument can be used to show that $\cH(\fn, \Q)'$ has an interpretation as a space of automorphic 
forms on $\G^{\fn}$; cf. \cite[$\S$8]{PapikianCM}. The Jacquet-Langlands correspondence over $F$ implies that there is an isomorphism 
$\JL: \cH_0(\fn, \Q)^\new\cong \cH(\fn, \Q)'$ which is compatible with the action of Hecke operators $T_\fp$, $\fp\nmid \fn$; 
cf. \cite[Ch. III]{JL}. 
\begin{rem}\label{remJLpbad}
In fact, $\JL$ is $U_\fp$-equivariant also for $\fp\mid \fn$. The analogous statement over $\Q$ 
follows from the results of Ribet \cite[$\S$4]{RibetLL} (cf. \cite[Thm. 2.3]{Helm}), 
which rely on the geometry of the Jacobians of modular curves. The N\'eron models of $J_0(\fn)$ and $J^\fn$  
have the necessary properties for Ribet's arguments to apply. Since this fact will not be essential for our purposes, we do not discuss the details further. 
The $U_\fp$-equivariance of $\JL$ can also be deduced from \cite{JL}; the necessary input is 
contained in \cite[Thm. 4.2]{JL} and its proof. 
\end{rem}

Let $J_0(\fn)^\old$ be the abelian subvariety of $J_0(\fn)$ generated by the images of $J_0(\fm)$ 
under the maps $J_0(\fm)\to J_0(\fn)$ induced by the degeneracy morphisms $X_0(\fn)_F\to X_0(\fm)_F$ 
for all $\fm\supsetneq \fn$. Let $J_0(\fn)^\new$ be the quotient of $J_0(\fn)$ 
by $J_0(\fn)^\old$. 
Combining $\JL$ with the main results in \cite{Drinfeld} and \cite{LRS}, 
and Zarhin's isogeny theorem, one concludes that there is a $\T(\fn)^0$-equivariant isogeny $J_0(\fn)^\new\to J^\fn$ defined over $F$, 
which we call a \textit{Jacquet-Langlands isogeny}; 
cf. \cite[Cor. 7.2]{PapikianJNT}.  
(This isogeny is in fact $\T(\fn)$-equivariant by the previous remark; see \cite[Cor. 2.4]{Helm} 
for the corresponding statement over $\Q$.) Zarhin's isogeny theorem provides no information 
about the Jacquet-Langlands isogenies beyond their existence. One possible 
approach to making these isogenies more explicit is via the study of component groups. 
More precisely, since $J_0(\fn)^\new$ and $J^\fn$ 
have purely toric reduction at the primes dividing $\fn$ and at $\infty$, one can deduce restrictions on 
possible kernels of isogenies  $J_0(\fn)^\new\to J^\fn$ from the component groups of these abelian varieties 
using \cite[Thm. 4.3]{PapikianJNT}. Unfortunately, the explicit structure of these component groups 
is not known in general. 

From now on we restrict to the case when $\fn=\fp\fq$ is a product of two distinct primes and $\deg(\fp)\leq 2$.  
In this case, the structure of $X^{\fn}_{\F_\fq}$ 
is fairly simple and can be deduced from \cite[Prop. 6.2]{PapikianJNT}. 
Using this and Raynaud's theorem, one computes that the component group $\Phi_\fq'$ of $J^{\fn}$ at $\fq$ is  
given by Table \ref{tableCG1}, where 
$$
M(\fq)=\begin{cases} 
|\fq|+1, & \text{if $\deg(\fq)$ is even;}\\
\frac{|\fq|+1}{q+1},& \text{if $\deg(\fq)$ is odd.}
\end{cases}
$$
\begin{table}
\begin{tabular}{ |c | c | c | }
\hline
& $\deg(\fq)$ is even & $\deg(\fq)$ is odd \\
\hline
$\deg(\fp)=1$ &  $\Z/(q+1)M(\fq)\Z$ & $\Z/M(\fq)\Z$\\
\hline
$\deg(\fp)=2$ &  $\Z/M(\fq)\Z$ & $\Z/(q+1)M(\fq)\Z$\\
\hline
\end{tabular}
\caption{Component group $\Phi_\fq'$ of $J^{\fp\fq}$}\label{tableCG1}
\end{table}

The cuspidal divisor group of $J_0(\fq)$ is generated by $[1]-[\infty]$, which has order 
$$
N(\fq)=\begin{cases} 
\frac{|\fq|-1}{q^2-1}, & \text{if $\deg(\fq)$ is even;}\\
\frac{|\fq|-1}{q-1},& \text{if $\deg(\fq)$ is odd.}
\end{cases}
$$
Let $\alpha: X_0(\fp\fq)_F\to X_0(\fq)_F$ be the degeneracy morphism discussed in Section \ref{sDMC}. 
The image of $\cC(\fq)$ in $J_0(\fp\fq)$ under the induced map $J_0(\fq)\to J_0(\fp\fq)$ is generated by 
$$
c:=\alpha^\ast([1]-[\infty])=|\fp|[1]+[\fp]-|\fp|[\fq]-[\infty]. 
$$
By examining the specializations of the cusps in $X_0(\fp\fq)_{\F_\fq}$, we see that 
$$
\wp_\fq(c)=(|\fp|+1)z,
$$
where $z\in \Phi_\fq$ is the element from the proof of Proposition \ref{prop5.3}. Let $\tilde{\Phi}_\fq:=\Phi_\fq/\wp_\fq(c)$. 
The order of $z$ in $\Phi_\fq$ is given in \cite[Thm. 4.1]{PapikianJNT}, and the order of $\Phi_\fq$ itself 
is given in Proposition \ref{prop5.3}. From this one easily computes that 
$\tilde{\Phi}_\fq$ is the group in Table \ref{tableCG2}.  
\begin{table}
\begin{tabular}{ |c | c | c | }
\hline
& $\deg(\fq)$ is even & $\deg(\fq)$ is odd \\
\hline
$\deg(\fp)=1$ &  $\Z/(q+1)\Z$ & $0$\\
\hline
$\deg(\fp)=2$ &  $\Z/(q^2+1)\Z$ & $\Z/(q^2+1)(q+1)\Z$\\
\hline
\end{tabular}
\caption{$\tilde{\Phi}_\fq$}\label{tableCG2}
\end{table}

Since $c\in J_0(\fn)^\old$, the map of component groups $\Phi_\fq\to \Phi_\fq^{\new}$ induced by the quotient $J_0(\fn)\to J_0(\fn)^\new$ 
must factor through $\tilde{\Phi}_\fq$ (here $\Phi_\fq^{\new}$ denotes the component group of $J_0(\fn)^\new$ at $\fq$).

Assume $\deg(\fp)=1$. Then the cuspidal divisor $c_\fq:=[\fq]-[\infty]\in \cC(\fp\fq)$ has 
order $N(\fq)M(\fq)$ (see Theorem \ref{thm6.10}) and specializes to the connected 
component of identity $\cJ^0_{\F_\fq}$ of the N\'eron model of $J_0(\fn)$. 
Theorem 4.3 in \cite{PapikianJNT} describes how the component groups of abelian varieties 
with toric reduction over a local field change under isogenies, depending on the specialization of the 
kernel of the isogeny in the closed fibre. After comparing the groups $\tilde{\Phi}_\fq$ and $\Phi'_\fq$, 
and the orders of $c$ and $c_\fq$, this theorem suggests that there is an isogeny $J_0(\fn)^\new\to J^\fn$ 
whose kernel is isomorphic to $\Z/M(\fq)\Z$ and is generated by the image of $c_\fq$ in $J_0(\fn)^\new$. 

The case $\deg(\fp)=2$ can be analysed similarly. The order of $c_\fq$ is $(q+1)N(\fq)M(\fq)$. The 
image of $c_\fp$ in $\tilde{\Phi}_\fq$ generates its cyclic subgroup of order $q^2+1$. Hence 
there might be an isogeny $J_0(\fn)^\new\to J^\fn$  whose kernel is isomorphic to $\frac{\Z}{M(\fq)\Z}\times \frac{\Z}{(q^2+1)\Z}$. 

\begin{conj}\label{conjJL} Assume $\fn=\fp\fq$, where $\fp, \fq$ are distinct primes and 
$\deg(\fp)\leq 2$. There is an isogeny $J_0(\fp\fq)^\new\to J^{\fp\fq}$ whose 
kernel $K$ is generated by cuspidal divisors and 
$$
K\cong 
\begin{cases}
\frac{\Z}{M(\fq)\Z} & \text{if $\deg(\fp)=1$}; \\
\frac{\Z}{M(\fq)\Z}\times \frac{\Z}{(q^2+1)\Z} & \text{if $\deg(\fp)=2$}. 
\end{cases}
$$
\end{conj}

This conjecture is the function field analogue of Ogg's conjectures about Jacquet-Langlands isogenies over $\Q$; see 
\cite[pp. 212-216]{Ogg}. 

There are only two cases when $\fn$ is square-free, $J_0(\fn)$ is non-trivial, and $J_0(\fn)^\new=J_0(\fn)$. 
The first case is $\fn=xy$; Conjecture \ref{conjJL} then specializes to the conjecture in \cite{PapikianJNT}. 
We will prove this conjecture in $\S$\ref{ssJLmain} under certain assumptions, and unconditionally for some small $q$. 
The second case is $\fn=\fp\fq$, where $\fp\neq \fq$ are primes of degree $2$. In this case the conjecture 
predicts that there is a Jacquet-Langlands isogeny whose kernel is isomorphic to $\Z/(q^2+1)\Z\times \Z/(q^2+1)\Z$; cf. 
Example \ref{exampleC22}. 
The method of this paper should be possible to adapt to this latter case, and prove the conjecture 
for some small $q$. 


\subsection{Special case}\label{ssJLmain} Assume $\fn=xy$. To simplify the notation we 
put $$J:=J_0(xy),\ J':=J^{xy},\ \cH:=\cH_0(xy, \Z),\ \cH':=\cH(xy, \Z)',\ \T:=\T(xy).$$ 
First, we prove the analogue of Lemma \ref{lem5.5} for $J'$. 

\begin{lem}\label{lemEisPhi'}
Let $\Phi_\infty'$ be the component group of $J'$ at $\infty$. Let $\fp\lhd A$ be any prime not dividing $xy$. 
Then $T_\fp-|\fp|-1$ annihilates $\Phi_\infty'\cong \Z/(q+1)\Z$. 
\end{lem}
\begin{proof}
\begin{figure}
\begin{tikzpicture}[->, >=stealth', semithick, node distance=1.5cm, inner sep=.5mm, vertex/.style={circle, fill=black}]
\node[vertex] [label=left:$v_{+}$](1){};
  \node (2) [right of=1] {$\vdots$};
  \node[vertex] (3) [right of=2, label=right:$v_{-}$] {};
  \draw (1) edge[bend right]  (3) (1) edge[bend right=70]   (3) (1) edge[bend left]  (3) (1) edge[bend left=70] node [above] {$e_u$} (3); 
\end{tikzpicture}
\caption{$\G^{xy}\bs \sT$}\label{Fig6}
\end{figure}

The quotient graph $\G^{xy}\bs \sT$ has two vertices joined by $q+1$ edges; see \cite[Prop. 6.5]{PapikianJNT}. 
We label the vertices by $v_+$ and $v_{-}$, and label the edges by the elements of $\p^1(\F_q)$; see Figure \ref{Fig6}. 
Let $\gamma\in \GL_2(\Fi)$ be an arbitrary 
element. Then $\gamma v_{\pm}\equiv v_{\pm}\ \mod\ \G^{xy}$ if $\ord_\infty(\det \gamma)$ is even, and 
$\gamma v_{\pm}\equiv v_{\mp}\ \mod\ \G^{xy}$ if $\ord_\infty(\det \gamma)$ is odd. 
Consider the free $\Z$-module with generators $\{e_u, \overline{e_u}\ |\ u\in \p^1(\F_q)\}$, modulo the relations 
$\overline{e_u}=-e_u$. The action of $T_\fp$ on this module satisfies 
$$
T_\fp \sum_{u\in \p^1(\F_q)} e_u = (|\fp|+1)
\begin{cases} \sum_{u\in \p^1(\F_q)} e_u & \text{if $\deg(\fp)$ is even,} \\ 
-\sum_{u\in \p^1(\F_q)}e_u & \text{if $\deg(\fp)$ is odd.}\end{cases}
$$

$\cH'$ is generated by the cycles $\varphi_u=e_u-e_\infty$, $u\in \F_q$. Let $\varphi^\ast_u$ be the dual basis 
of $\Hom(\cH', \Z)$. The map in (\ref{eqPhiJ'inf}) sends $\varphi_u$ to 
$\varphi^\ast_u + \sum_{w\in \F_q} \varphi_w^\ast$. It is easy to see from this that $\Phi_\infty'$ 
is cyclic of order $q+1$ and is generated by $\sum_{w\in \F_q} \varphi_w^\ast$. Note that $|\fp|+1\equiv 0\ \mod\ (q+1)$ 
if $\deg(\fp)$ is odd. Hence 
$$
T_\fp\sum_{w\in \F_q} \varphi_w^\ast =T_\fp \left(\sum_{u\in \p^1(\F_q)}e_u^\ast-(q+1)e_\infty^\ast\right)
=\pm (|\fp|+1)\sum_{u\in \p^1(\F_q)}e_u^\ast-(q+1)T_\fp e_\infty^\ast 
$$
$$
\equiv 
(|\fp|+1)\sum_{u\in \p^1(\F_q)}e_u^\ast-(q+1)(|\fp|+1) e_\infty^\ast= (|\fp|+1) \sum_{w\in \F_q} \varphi_w^\ast\ \mod\ (q+1). 
$$
This implies that $T_\fp$ acts by multiplication by $|\fp|+1$ on $\Phi_\infty'$. 
\end{proof}

\begin{thm}\label{thmJL1}
Assume $\cH\cong \cH'$ as $\T$-modules. There is an isogeny $J\to J'$ defined over $F$ whose 
kernel is cyclic of order $q^2+1$ and annihilated by the Eisenstein ideal.  
\end{thm}
\begin{proof} By (\ref{eqGRs}) and (\ref{eqj}), the rigid-analytic uniformization of $J$ over $\Fi$ is given by 
the $\T^0$-equivariant sequence 
$$
0\to \cH\to \Hom(\cH, \C_\infty^\times)\to J(\C_\infty)\to 0. 
$$
By Proposition \ref{propT(xy)}, $\Hom(\cH, \Z)\cong \T=\T^0$, so we can write this sequence as the 
$\T$-equivariant sequence 
$$
0\to \cH\to \T\otimes \C_\infty^\times \to J(\C_\infty)\to 0. 
$$
By Theorem \ref{thmCGinf}, the sequence derived from this using the valuation homomorphism $\ord_\infty$ is
$$
0\to \cH\to \T  \to \Phi_\infty\to 0,
$$
where $\Phi_\infty$ is the component group of the N\'eron model of $J$ at $\infty$. 
Now we can consider $\cH$ as ideal in $\T$. We know from Lemma \ref{lem5.5} 
that the Eisenstein ideal $\fE$ annihilates $\Phi_\infty\cong \Z/(q^2+1)(q+1)\Z$. Hence 
$\Phi_\infty$ is a quotient of $\T/\fE$. On the other hand, by Corollary \ref{corT/E}, $\T/\fE\cong \Z/(q^2+1)(q+1)\Z$. 
Comparing the orders of $\Phi_\infty$ and $\T/\fE$, we conclude that $\Phi_\infty\cong \T/\fE$ and $\cH\cong \fE$. 

From the discussion in $\S$\ref{ssRAUJ'}, if we assume $\cH\cong \cH'$ as $\T$-modules, the 
rigid-analytic uniformization of $J'$ over $\Fi$ is given by the $\T$-equivariant sequence 
$$
0\to \cH'\to \T\otimes \C_\infty^\times \to J'(\C_\infty)\to 0. 
$$
The argument in the previous paragraph allows us to 
identify $\cH'$ in the above sequence with the annihilator $\fE'\lhd \T$ of 
$\Phi_\infty'$ in $\T$. On the other hand, by Lemma \ref{lemEisPhi'}, $T_\fp-|\fp|-1\in \fE'$ 
for any $\fp\nmid xy$. Therefore, $\fE\subset \fE'$ and $$\T/\fE'\cong \Phi_\infty'\cong \Z/(q+1)\Z.$$ 
Hence $\fE'/\fE\cong \Z/(q^2+1)\Z$. 

We have identified the uniformizing tori of $J$ and $J'$ with $\T\otimes \C_\infty^\times$, 
and the uniformizing lattices $\cH$ and $\cH'$ with $\fE$ and $\fE'$, respectively. Now 
specializing a theorem of Gerritzen \cite{Gerritzen} to this situation, we get a natural 
bijection 
$$
\Hom_{\T}(\T\otimes\C_\infty^\times, \fE; \T\otimes\C_\infty^\times, \fE')\xrightarrow{\sim}\Hom_{\T}(J_{F_\infty}, J'_{F_\infty}),
$$
where on the left hand-side is the group of homomorphisms $\T\otimes\C_\infty^\times \to \T\otimes\C_\infty^\times$ 
which map $\fE$ to $\fE'$ and are compatible with the action of $\T$.  
It is clear that identity map on $\T\otimes\C_\infty^\times$
is in this set. The snake lemma applied to the resulting diagram 
$$
\xymatrix{
0\ar[r] & \fE \ar[r] \ar@{^{(}->}[d] &  \T\otimes \C_\infty^\times \ar[r]\ar@{=}[d] & J(\C_\infty)\ar[r]\ar[d]^{\pi} & 0 \\
0\ar[r] & \fE' \ar[r] & \T\otimes \C_\infty^\times \ar[r] & J'(\C_\infty) \ar[r]& 0 
}
$$
shows that there is an isogeny $\pi: J\to J'$ with $\ker(\pi)\cong \fE'/\fE$. Moreover, 
since $\Hom_{\T}(\T, \T)\cong\T$, every $\T$-equivariant homomorphism $J\to J'$ can be obtained as a composition 
of $\pi$ with an element of $\T$. We know that there is an isogeny $J\to J'$ defined over $F$. 
Since the endomorphisms of $J$ induced by the Hecke operators are also defined over $F$, we 
conclude that $\pi$ is defined over $F$.  
\end{proof}

\begin{thm}\label{thmJL2} In addition to the assumption in Theorem \ref{thmJL1}, 
assume $\T_\fM$ is Gorenstein for all maximal Eisenstein ideals $\fM$ of residue characteristic 
dividing $q^2+1$. Then there is an isogeny $J\to J'$ whose kernel is $\langle c_y\rangle$. 
\end{thm}
\begin{proof} By Theorem \ref{thmJL1}, there is an isogeny $J\to J'$ defined over $F$ 
whose kernel $H\subset J[\fE]$ is cyclic of order $q^2+1$. 
Assume $\T_\fM$ is Gorenstein for all maximal Eisenstein ideals $\fM$ of residue characteristic 
dividing $q^2+1$. First, by Theorem \ref{thmCT} and Proposition \ref{prop9.5}, this implies $J[2, \fE]=\cC[2]$. 
Next, let $\ell|(q^2+1)$ be an odd prime. From the discussion after Proposition \ref{prop9.5} 
we know that the action of inertia at $y$ on $J_\ell[\fE]$ is unipotent and the $I_y$-invariant subgroup of $J_\ell[\fE]$ 
is $\cC_\ell=\langle c_y\rangle_\ell$.  Since $4$ does not divide $q^2+1$, these two 
observations imply that there is an isogeny $J\to J'$ whose kernel is cyclic of order $q^2+1$ and is contained in $\cC[2]+\langle c_y\rangle$. 

If $q$ is even, then $\cC[2]+\langle c_y\rangle=\langle c_y\rangle$, and we are done. 
If $q$ is odd, then the kernel is generated by $c_y$, or $c_y+\frac{q+1}{2}c_x$, or $2c_y+\frac{q+1}{2}c_x$. 
We know that $\wp_y(c_y)=0$ and $\wp_y(c_x)=1$ (cf. Proposition \ref{prop6.2}). 
If $H$ is generated by $c_y+\frac{q+1}{2}c_x$ or $2c_y+\frac{q+1}{2}c_x$, 
then the specialization map $\wp_y$ gives the exact sequence 
$$
0\to \frac{\Z}{\frac{q^2+1}{2}\Z}\to H\xrightarrow{\wp_y} \Z/2\Z\to 0. 
$$ 

It is a consequence of the uniformization theorem in \cite{Hausberger} that $X^\fn_{F_\fp}$ 
is a twisted Mumford curve for any $\fp | \fn$ (here $\fn$ is arbitrary). This implies that $J^\fn$ 
has toric reduction at $\fp$. In particular, $J'$ has toric reduction at $y$. Now we can apply  
Theorem 4.3 in \cite{PapikianJNT} to get an exact sequence 
$$
0\to \Z/2\Z\to \Phi_y\to \Phi_y'\to \frac{\Z}{\frac{q^2+1}{2}\Z}\to 0,  
$$
where $\Phi_y'$ is the component group of $J'$ at $y$. 
This implies that the order of $\Phi_y'$ is $(q+1)(q^2+1)/4$. But according to \cite[Thm. 6.4]{PapikianJNT}, 
$\Phi_y'\cong \Z/(q^2+1)(q+1)\Z$, which leads to a contradiction. 
\end{proof}

To be able to verify the assumptions in Theorem \ref{thmJL2} computationally, it is crucial to be able to 
compute the action of $\T$ on $\cH$ and $\cH'$. The methods for doing this will be discussed in Section \ref{sComputations}. 
Our calculations lead to the following:

\begin{prop}
The assumptions of Theorem \ref{thmJL2} hold for $q=2$, and for the 12 cases listed 
in Table \ref{table2}. In particular, in these cases 
there is an isogeny $J\to J'$ whose kernel is $\langle c_y\rangle$. 
\end{prop}

\begin{rem}
We believe that the assumptions in Theorem \ref{thmJL2} hold in general. 
Our method for verifying that $\cH$ and $\cH'$ are isomorphic $\T$-modules relies on finding a 
perfect $\T$-equivariant pairing $\T\times \cH'\to \Z$. Unlike the case of $\cH$, 
there is no natural pairing between these modules, so our method is by trial-and-error. 
We essentially construct some $\T$-equivariant pairings, and check if one of those 
is perfect (see the discussion after (\ref{eqPairH'T})). This method is very inefficient, and 
our computer calculations terminated in a reasonable time only in the cases listed in Table~\ref{table2}.
\end{rem}


\section{Computing the action of Hecke operators}\label{sComputations}

\subsection{Action on $\cH$} Assume $\fn=xy$. 
To simplify the notation, we denote $\cH=\cH_0(xy, \Z)$, $\cH'=\cH(xy, \Z)'$, $\T=\T(xy)$. 
Assume $x=T$. Theorem 6.8 in \cite{GekelerKleinem} gives a recipe for computing a  
matrix by which $T_{x-s}$ acts on $\cH$ for any $s\in \F_q^\times$. 
Since by Proposition \ref{propT(xy)} the operators $\{1, T_{x-s}\ |\ s\in \F_q^\times\}$ 
form a $\Z$-basis of $\T$, this essentially gives a complete description of the 
action of $\T$ on $\cH$. This also allows us to compute the discriminant of $\T$, an interesting invariant 
measuring the congruences between Hecke eigenforms. 
(Recall that $\disc(\T)$ is the determinant of the $q\times q$ matrix $\big(\text{Trace}(t_it_j)\big)_{1\leq i, j\leq q}$, 
where $\{t_1, \dots, t_q\}$ is a $\Z$-basis of $\T$.)  Table \ref{table1} lists the values of $\disc(\T)$ in some cases. 

\begin{table}
\begin{tabular}{|c|c|c|}
\hline
$q$ & $y$ & $\disc(\T)$\\
\hline
$2$ & $T^2+T+1$ & $4$ \\
\hline
$3$ & $T^2+1$ & $80$ \\
\hline
$3$ & $T^2+T+2$ & $68$ \\
\hline
$3$ & $T^2+2T+2$ & $68$ \\
\hline
$5$ & $T^2+T+1$ & $265216$ \\
\hline
$5$ & $T^2+T+2$ & $278800$ \\
\hline
$7$ & $T^2+1$ & $7372800000$ \\
\hline
$7$ & $T^2+T+4$ & $6567981056$ \\
\hline
\end{tabular}
\caption{Discriminant of $\T(xy)$}\label{table1}
\end{table}

\begin{rem}
The algebra $\T(\fn)\otimes\C_\infty$ 
is isomorphic to the Hecke algebra acting on doubly cuspidal Drinfeld 
modular forms of weight $2$ and type $1$ on $\G_0(\fn)$; see \cite[(6.5)]{GR}. The algebra  
$\T(\fn)\otimes \C_\infty$ has no nilpotent elements if and only if $p\nmid \disc(\T)$. 
Table \ref{table1} indicates that $p\nmid \disc(\T)$ for $q=3$ and arbitrary $y$, but 
for $q = 5$ or $7$, there exist $y_1$ and $y_2$ such that $p \mid \disc(\T(xy_1))$ and $p \nmid \disc(\T (xy_2))$.
It seems like an interesting problem to investigate the frequency with which $p$ 
divides $\disc(\T)$. 
\end{rem}

\begin{rem} \label{rem10.2}
For the sake of completeness, and also because \cite{GekelerKleinem} is in German, 
we give Gekeler's method for computing a matrix $G(x-s)\in \Mat_q(\Z)$, $s\in \F_q^\times$, 
representing the action of $T_{x-s}$ on $\cH$. 
Let $y=T^2+aT+b$. Label the rows and columns of 
$G(x-s)$ by $u, w\in \F_q$. Then the $(u,w)$ entry of $G(x-s)$ is equal to 
$$
2-Q(u,w)-(q+1)\delta_{w, s}+ q\delta_{u, 0}\delta_{w, b/s},
$$ 
where $\delta$ is Kroneker's delta, and $Q(u,w)$ is the number of solutions $\beta\in \F_q$ (without multiplicities) 
of the equation 
$$
(u-\beta)(w-\beta)(s-\beta)+\beta(\beta^2+a\beta+b)=0
$$
plus $1$ if $u+w+s+a=0$.
\end{rem}

By comparing the discriminant of the characteristic polynomial of $T_{x-s}$ with $\disc(\T)$, 
one can deduce that in some cases $\T$ is monogenic, i.e., is generated by a single element as a $\Z$-algebra:
\begin{example}
For $q=2$ and $y=T^2+T+1$ 
$$
\T\cong \Z[T_{x-1}]\cong \Z[X]/X(X+2).
$$
For $q=3$ and $y=T^2 + T + 2$ 
$$
\T\cong \Z[T_{x-2}]\cong \Z[X]/(X+1)(X^2 - X - 4). 
$$
\end{example}

If $\T$ is monogenic, then its localization at any maximal ideal is Gorenstein; see \cite[Thm. 23.5]{Matsumura}. 
This is stronger than what we need for Theorem \ref{thmJL2}, but it is also computationally harder to 
establish. A simpler test is based on the following lemma: 

\begin{lem}\label{lemeta} 
Let $\ell$ be a prime number dividing $(q+1)(q^2+1)$. 
Suppose there is an element $\eta$ in $\fE$ such that
$$
\dim_{\F_\ell}\cH_{00}(xy, \F_\ell)[\eta]=1. 
$$ 
Then $\T_\fM$ is Gorenstein, where $\fM=(\fE, \ell)$.   
\end{lem}
\begin{proof} 
Note that the dimension of $\cH_{00}(xy, \F_\ell)[\eta]$ is at least $1$ 
since $\cE_{00}(xy, \F_\ell)\cong \F_\ell$ is a subspace. Now 
consider the ideal $\fI=(\ell, \eta)\T_\ell$. We 
have $\cH_{00}(xy, \F_\ell)[\fI]=\cH_{00}(xy, \F_\ell)[\eta]$. On the other hand, 
$\cH_{00}(xy, \F_\ell)[\fI]$ is $\F_\ell$-dual to $\T_\ell/\fI$, cf. Proposition \ref{propT(xy)}. 
Hence, if $\dim_{\F_\ell}\cH_{00}(xy, \F_\ell)[\eta]=1$, 
then $\T_\ell/\fI\cong \F_\ell$. This implies $\T_\ell/\fM\cong \T_\ell/\fI$. Since 
$\fI\subset \fM$, we get $\fM=(\ell, \eta)$, which implies that $\T_\fM$ 
is Gorenstein by Proposition 15.3 in \cite{Mazur}.  
\end{proof}

Any $\Z$-linear combination $\eta$ of the operators $\{T_{x-s} - (q+1)\}_{s \in \F_q^{\times}}$ is in $\fE$. 
We can compute the characteristic polynomial of such $\eta$ acting on $\cH$ using 
Gekeler's method. Fix $\ell$ dividing $(q+1)(q^2+1)$. If we find $\eta$ 
whose characteristic polynomial modulo $\ell$ does not have $0$ as a multiple root, 
then we can apply Lemma \ref{lemeta} to conclude that $\T_\fM$ is Gorenstein for Eisenstein $\fM$ 
of residue characteristic $\ell$. Using this strategy, for each 
prime $q\leq 7$ we found by computer calculations an appropriate $\eta$ for any $\ell$ dividing $(q+1)(q^2+1)$. 

\subsection{Action on $\cH'$}

Suppose $y = T^2+aT+b$. We denote the place $x$ by $\infty'$. 
Let $T' = T^{-1}$, $A' = \F_q[T']$, and $y' = T'^2+ab^{-1} T' + b^{-1}$. We have the correspondence of places of $F$:
$$\begin{tabular}{ccc}
$\infty$ & $\longleftrightarrow$ & $T'$ \\
$T$ & $\longleftrightarrow$ & $\infty'$ \\
$y$ & $\longleftrightarrow$ & $y'$
\end{tabular}$$

Let $D$ be the quaternion algebra over $F$ ramified precisely at $\infty'$ and $y'$ (i.e.\ ramified at $x$ and $y$).
Take an Eichler $A'$-order $\cD'$ in $D$ of level $T'$. More precisely, $\cD'_{\fp'} :=\cD' \otimes_{A'} \cO_{\fp'}$ is a maximal $\cO_{\fp'}$-order in $D_{\fp'}:= D\otimes_F F_{\fp'}$ for each prime $\fp'$ of $A'$ with $\fp' \neq T'$, and there exists an isomorphism $\iota_{T'}: D_{T'}:= D \otimes_F F_{T'} \cong \Mat_2(F_{T'})$ such that
$$\iota_{T'}(\cD'_{T'}) = \left\{ \begin{pmatrix}a&b\\c&d\end{pmatrix} \in \Mat_2 (\cO_{T'}) \ \bigg| \ c \equiv 0 \bmod T'\right\}.$$
Let $\cO_{D_{\infty'}}$ be the maximal $\cO_{\infty'}$-order in $D_{\infty'}$.
Consider the double coset spaces
$$\cG_{xy}':= D^{\times} \backslash D^{\times} (\A_F)/ \left(\widehat{\cD'}^{\times} \cdot (\cO_{D_{\infty'}}^{\times} F_{\infty'}^{\times})\right) \quad \text{ and } \quad  \text{Cl}(\cD'):= D^{\times} \backslash D^{\times}(\A_F^{\infty'})/\widehat{\cD'}^{\times},$$
where:
\begin{itemize}
\item $\A_F$ is the adele ring of $F$, i.e.\ $\A_F$ is the restricted direct product $\prod_v' F_v$;
\item $\A_F^{\infty'}$ is the finite (with respect to $\infty'$) adele ring of $F$, i.e.\ $\A_F^{\infty'} = \prod_{v \neq \infty'}F_v$;
\item $D^{\times}(\A_F)$ (resp.\ $D^{\times} (\A_F^{\infty'})$) denotes $(D\otimes_F \A_F)^{\times}$ (resp.\ $(D\otimes_F \A_F^{\infty'})^{\times}$);
\item $\widehat{\cD'} = \prod_{\fp' \lhd A'} \cD'_{\fp'}$.
\end{itemize}
Then the strong approximation theorem (with respect to $\{\infty\}$) shows that

\begin{lem}
The double coset space $\cG'_{xy}$ can be identified with the set of the oriented edges of the 
quotient graph $\Gamma^{xy}\backslash \sT$ in Figure \ref{Fig6}.
\end{lem}

Note that $\text{Cl}(\cD')$ can be identified with the locally-principal right ideals of $\cD'$ in $D$, and $\#\text{Cl}(\cD') = q+1$.
Moreover, if we take $i_u \in \GL_2(F_{T'})$, $u \in \mathbb{P}^1(\F_q)$, to be 
$$i_u = \begin{cases} \begin{pmatrix} u & 1 \\ 1 &0 \end{pmatrix}, & \text{ if $u \in \F_q$,} \\
\begin{pmatrix}1&0 \\ 0 &1\end{pmatrix}, & \text{if $u = \infty$,}\end{cases}$$
then, via the natural embedding $\GL_2(F_{T'}) \cong D_{T'}^{\times} \hookrightarrow D^{\times} (\A_F^{\infty'})$, 
$\{i_u \mid u \in \mathbb{P}^1(\F_q)\}$ is a set of representatives of $\text{Cl}(\cD')$.
Take an element $\varpi_{\infty'} \in D_{\infty'}$ such that its reduced norm $\text{Nr}(\varpi_{\infty'}) = T$.
From the natural surjection from $\cG_{xy}'$ to $\text{Cl}(\cD')$, one observes that
$$\left\{(i_u, \varpi_{\infty'}^c) \in D^{\times}(\A_F^{\infty'}) \times D_{\infty'}^{\times}  = D^{\times} (\A_F) \mid u \in \mathbb{P}^1(\F_q),\ c = 0,1 \right\}$$
is a set of representatives of $\cG_{xy}'$.
We may take $e_u := [i_u,1] \in \cG_{xy}'$.
Then there exists a unique permutation $\gamma : \mathbb{P}^1(\F_q) \rightarrow \mathbb{P}^1(\F_q)$ of order $2$ so that
$$\overline{e_u}:= \left[i_u \begin{pmatrix}0&1\\ T'&0\end{pmatrix},1\right] = [i_{\gamma(u)},\varpi_{\infty'}].$$
Moreover, $\cH'$ can be viewed as the set of $\Z$-valued functions $f$ on $\cG'_{x,y}$ satisfying
$$f(e_u) + f(\overline{e_u}) = 0,\ \forall u \in \mathbb{P}^1(\F_q) \quad \text{ and } \quad
\sum_{u \in \mathbb{P}^1(\F_q)} f(e_u) = 0.$$

Let $I_u$ be the right ideal of $\cD'$ in $D$ corresponding to $i_u$, i.e., 
$$I_u:= D \cap i_u \cdot \widehat{\cD'}.$$
Then the reduced norm of $I_u$ is trivial, i.e., 
the ideal of $A'$ generated by the reduced norms of all elements in $I_u$ is $A'$.
For each ideal $\fm' \lhd A'$, the \text{\it $\fm'$-th Brandt matrix} $B(\fm') = \big(B_{u,u'}(\fm')\big)_{u,u'} \in \Mat_{q+1}(\Z)$ is defined by
$$B_{u,u'}(\fm'):= \frac{\displaystyle \#\{b \in I_u I_{u'}^{-1}: \text{Nr}(b) \cdot A' = \fm' \}}{\displaystyle q-1}.$$

In Section \ref{Brandt}, we give an explicit recipe of computing $B_{u,u'}(\fm')$ for $\deg(\fm') = 1$ 
when $q$ is odd and the constant term $b$ of $y$ is not a square in $\F_q^{\times}$, and also when $q=2$. 
By an analogue of Hecke's theory (cf. \cite{CLWY}) we obtain the following result:

\begin{prop}\label{prop10.6}\hfill
\begin{enumerate}
\item  Viewing $\gamma$ as a permutation matrix in $\Mat_{q+1}(\Z)$, one has
$$B(T') = 2\cdot J - \gamma,$$
where every entry of $J$ is $1$.\\
\item Identifying the place $x-s$ and $T'-s^{-1}$ of $F$, one has that for $s \in \F_q^{\times}$,
$$T_{x-s} e_u = \sum_{u' \in \mathbb{P}^1(\F_q)} B_{u,\gamma(u')}(T'-s^{-1}) \overline{e_{u'}}.$$
\end{enumerate}
\end{prop}

\begin{proof}
By Lemma II.5 in \cite{CLWY}, we observe that for every $u \in \F_q$,
$$[i_{\gamma(u)}] + \sum_{u' \in \F_q}B_{u,u'}(T') [i_{u'}]  = 2 \sum_{u'' \in \F_q} [i_{u''}] \in \Z[\text{Cl}(\cD')].$$
This shows $(1)$.
To prove $(2)$, notice that for every $g_v \in \GL_2(F_v)$ with $v \neq \infty'$ and $u \in \F_q$, there exists $u' \in \F_q$ such that
$$ [i_u g_v, 1] = [i_{u'}, \varpi_{\infty'}^c],$$
where $c = \ord_v(\det g_v) \cdot \deg v$.
By Proposition II.4 in \cite{CLWY} we have that for $s \in \F_q^{\times}$,
\begin{eqnarray}
T_{x-s} e_u &=& \sum_{u' \in \mathbb{P}^1(\F_q)} B_{u,u'}(T'-s^{-1}) [i_{u'}, \varpi_{\infty'}] \nonumber \\
&=& \sum_{u' \in \mathbb{P}^1(\F_q)} B_{u,u'}(T'-s^{-1}) \overline{e_{\gamma(u')}} \nonumber \\
&=& \sum_{u' \in \mathbb{P}^1(\F_q)} B_{u,\gamma(u')}(T'-s^{-1}) \overline{e_{u'}}. \nonumber
\end{eqnarray}
\end{proof}

For $u \in \F_q$, let $f_u \in \cH'$ such that $f_u(e_u') = \delta_{u,u'}$ for $u' \in \F_q$ and $f_u(e_{\infty}) = -1$. We immediately get the following.

\begin{cor}\label{cor10.7}
For $u,u' \in \F_q$ and $s \in \F_q^{\times}$, set
$$B_{u',u}'(x-s) := B_{u',\gamma(\infty)}(T'-s^{-1}) - B_{u',\gamma(u)}(T'-s^{-1}).$$
Then
$$ T_{x-s} f_u = \sum_{u' \in \F_q} B'_{u',u} (x-s) f_{u'}.$$
In other words, $B'(x-s) := \left(B'_{u',u}(x-s)\right)_{u',u} \in \Mat_q(\Z)$ is a 
matrix representation of $T_{x-s}$ acting on $\cH'$ with respect to the basis $\{ f_u \mid u \in \F_q\}$.
\end{cor}

\begin{rem} From (\ref{eqS=-1}) we know that  $\sum_{s \in \F_q} T_{x-s} = -1$, as endomorphisms of $\cH$. 
On the other hand, by Remark \ref{remJLpbad}, $\JL: \cH\otimes \Q\cong\cH'\otimes \Q$ is $T_x$-equivariant. 
Hence the previous corollary allows us to obtain also  
the matrix representation of $T_x$ acting on $\cH'$. 
\end{rem}

Remark \ref{rem10.2} and Corollary \ref{cor10.7} give the matrices by which 
$T_{x-s}$ ($s \in \F_q^{\times}$) acts on $\cH$ and $\cH'$. Since $T_{x-s}$ 
generate $\T$, the condition that the $\T$-modules $\cH$ and $\cH'$ are isomorphic in Theorem \ref{thmJL1} 
is equivalent to the matrices $B'(x-s)$ being simultaneously $\Z$-conjugate to Gekeler's matrices $G(x-s)$, i.e.,  
to the existence of a single matrix $C \in \GL_q(\Z)$ such that 
\begin{equation}\label{eqMatC}
C^{-1} \cdot B'(x-s) \cdot C = G(x-s) \qquad \forall s \in \F_q^{\times}.
\end{equation}

\begin{rem}Due to Jacquet-Langlands correspondence, there does exist a matrix $C\in \GL_q(\Q)$ satisfying (\ref{eqMatC}), 
but the existence of an integral matrix is more subtle; cf. \cite{LM}.
\end{rem}

\begin{example}
Let $q = 2$ and $y = T^2+T+1$. By Remark \ref{rem10.2}
$$G(x-1) = \begin{pmatrix}0&0\\1&-2\end{pmatrix}.$$
On the other hand, with respect to the basis $\{i_{\infty}, i_0,i_1\}$ of $\Z[\text{Cl}(\cD')]$ we calculate that (see Remark \ref{10.1.1})
$$\gamma = \begin{pmatrix}0&1&0\\1&0&0\\0&0&1\end{pmatrix} \quad \text{ and } \quad B(T'-1) =  
\begin{pmatrix} 0&2&1\\2&0&1\\1&1&1\end{pmatrix}.$$
Therefore
$$B'(x-1) = \begin{pmatrix} -2&-1\\ 0&0\end{pmatrix}.$$
We can take $C = \begin{pmatrix}0&-1\\ 1&0\end{pmatrix}$. 

Note that the matrix 
$\begin{pmatrix}0&0\\0&-2\end{pmatrix}$ is conjugate to $G(x-1)$ over $\Q$, but not over $\Z$; in fact, 
there are exactly two conjugacy classes of matrices in $\Mat_2(\Z)$ with characteristic polynomial $X(X+2)$. 
\end{example}

There exists an algorithm for deciding whether, for two collections of integral matrices 
$\{X_1, \dots, X_m\}$ and $\{Y_1, \dots , Y_m\}$, there exists an integral and integrally
invertible matrix $C$ relating them via conjugation, i.e., such that $C^{-1}X_iC=Y_i$ for all $i$; see \cite{Sarkisyan}. 
Unfortunately, this algorithm is complicated and does not seem to 
have been implemented into the standard computational programs, such as \texttt{Magma}. 
Instead of trying to find $C$, we take a different approach to proving that the $\T$-modules $\cH$ and $\cH'$ are isomorphic. 
Note that the pairing in Proposition \ref{propT(xy)} gives an isomorphism $\cH\cong \Hom(\T,\Z)$ of $\T$-modules. 
If one constructs a perfect $\T$-equivariant pairing 
\begin{equation}\label{eqPairH'T}
\cH'\times \T\to \Z,
\end{equation}
then the desired isomorphism $\cH'\cong \Hom(\T,\Z) \cong \cH$ follows. 
The absence of Fourier expansion in the quaternionic setting makes the construction of such a pairing ad hoc. 

Note that $\Hom(\cH',\Z) = \bigoplus_{u \in \F_q} \Z e_u^*$, where $\langle f_u,e_{u'}^*\rangle:= f_u(e_{u'}) = \delta_{u,u'}$.
One way to construct (\ref{eqPairH'T}) is to find a $\Z$-linear combination $\alpha = \sum_u a_u e_u^*$ such that
$$\det \left( \langle f_u, T_{x-s} \alpha\rangle \right)_{u,s \in \F_q}= \pm 1.$$
We were able to find such $\alpha$ in several cases; see Table \ref{table2} where $\alpha$ is given as $[a_0, a_1, \dots, a_{p-1}]$.

\begin{table}
\begin{tabular}{|c|c|c|}
\hline
$q$ & $y$ & $\alpha$\\
\hline
$3$ & $T^2+T+2$ & $[0,1,0]$ \\
\hline
$3$ & $T^2+2T+2$ & $[0,0,1]$ \\
\hline
$5$ & $T^2+T+2$ & $[-1, 1, 4, 5, 2]$ \\
\hline
$5$ & $T^2+2T+3$ & $[-1, -3, -6, -5, -2]$ \\
\hline
$5$ & $T^2+3T+3$ & $[-1, -6, -3, -2, -5]$ \\
\hline
$5$ & $T^2+4T+2$ & $[-1, 4, 1, 2, 5]$ \\
\hline
$7$ & $T^2+T+6$ & $[-8, 0, -6, -5, -8, -7, 5]$ \\
\hline
$7$ & $T^2+2T+3$ & $[-8, -7, -7, 2, 3, -6, -6]$ \\
\hline
$7$ & $T^2+3T+5$ & $[-5, -6, -6, -4, 2, 3, -5]$ \\
\hline
$7$ & $T^2+4T+5$ & $[-8, -8, 0, -7, -6, 5, -5]$ \\
\hline
$7$ & $T^2+5T+3$ & $[-5, -4, -5, -6, 3, -6, 2]$ \\
\hline
$7$ & $T^2+6T+6$ & $[-5, -6, 2, -5, -6, -4, 3]$ \\
\hline
\end{tabular}
\caption{Available choice of $\alpha$}\label{table2}
\end{table}

\subsection{Computation of Brandt matrices}\label{Brandt}
Recall that we denote $y = T^2+aT+b \in A$ and $y' = T'^2+ab^{-1}T+b^{-1}$, where  
$T' = 1/T$. 
Assume $q$ is odd and $b$ is not a square in $\F_q^{\times}$. Set $y_0 := bT'^2+aT'+1$. Let 
$$D:= F + Fi + Fj +Fij$$ 
where $i^2 = T'$, $j^2 = y_0$, $ij= -ji$.
Since we have the Hilbert quadratic symbols $(T',y_0)_{T} = (T',y_0)_{y} = -1$ and $(T',y_0)_v =1$ for every places $v \neq T$ or $y$,
the quaternion algebra $D$ is ramified precisely at $T$ and $y$.

Let $\cD' = A'+A'i+A'j+A'ij$, which is an Eichler $A'$-order in $D$ of level $T'y'$.
We choose the representatives of locally principal right ideal classes of $\cD'$ as follows.
\begin{eqnarray}
I_{\infty} & := & \cD' = A'+A'i + A'j +A'ij, \nonumber \\
I_u & := & A'(1-j) + A'(\frac{i+ij}{T'}+2uj) + A'T'j + A'ij, \quad u \in \F_q. \nonumber
\end{eqnarray}
Then the reduced norm of $I_u$ is trivial for every $u \in \mathbb{P}^1(\F_q)$. 
Moreover, for $u, u' \in \F_q$, $I_u I_{u'}^{-1}$ is equal to
$$ 
\begin{cases}
A'+A'T'i+A'(j-2ui) + A\left(\frac{i(1+j)^2}{T'}+4u^2 i\right), & \text{ if $u = u'$,}\\
A'T' + A\big(i- (u'-u)^{-1}\big) + A\left(j-\frac{u'+u}{u'-u}\right) + A\left(\frac{i(1+j)^2}{T'}+\frac{4uu'}{u'-u}\right), & \text{ if $u \neq u' $.}
\end{cases}$$
After tedious calculations, we obtain the following:

\begin{lem}\label{lem10.8}\hfill 
\begin{enumerate}
\item For $s \in \F_q$,
$$\frac{\#\{z \in I_{\infty} \mid \Nr(z)A' = (x-s)A'\}}{q-1} = 
\begin{cases}
2 & \text{ if $s \in (\F_q^{\times})^2$,}\\
1 & \text{ if $s = 0$,} \\
0 & \text{ otherwise.}
\end{cases}$$
\item For $s,u \in \F_q$, set 
$$\alpha_y(s,u):= 1+s(4u^2+a+sb) \quad \text{ and } \quad
\beta_y(s,u) = \big((a+4u^2)^2-4b\big)s+16u^2.$$
Then
$$\frac{\#\{z \in I_u \mid \Nr(z)A' = (x-s)A'\}}{q-1} = 
\begin{cases}
2 & \text{ if $\alpha_y(s,u) \in (\F_q^{\times})^2$,}\\
1 & \text{ if $\alpha_y(s,u) = 0$,} \\
0 & \text{ otherwise.}
\end{cases}$$
$$\frac{\#\{z \in I_u I_u^{-1} \mid \Nr(z)A' = (x-s)A'\}}{q-1} = 
\begin{cases}
2 & \text{ if $\beta_y(s,u) \in (\F_q^{\times})^2$,}\\
1 & \text{ if $\beta_y(s,u) = 0$,} \\
0 & \text{ otherwise.}
\end{cases}$$
\item For $s,u,u' \in \F_q$ with $u \neq u'$, set
\begin{align*} 
\xi_y(s,u,u'):=&  \big(2u^2+2u'^2+s(u'-u)a\big)^2 \\ 
&-\big(1-s(u'-u)^2\big)\big(16u^2u'^2-s(u'-u)^2(a^2-4b)\big).
\end{align*}
Then
$$\frac{\#\{z \in I_uI_{u'}^{-1} \mid \Nr(z)A' = (x-s)A'\}}{q-1} = 
\begin{cases}
2 & \text{ if $\xi_y(s,u) \in (\F_q^{\times})^2$,}\\
1 & \text{ if $\xi_y(s,u) = 0$,} \\
0 & \text{ otherwise;}
\end{cases}$$
\end{enumerate}
\end{lem}

Since the Brandt matrices are symmetric under our settings, the above lemma gives us the 
recipe of computing the Brandt matrices $B(T'-s)$ for $s \in \F_q$. In particular, we can get
$$\gamma(\infty) = \infty,\quad \gamma(0) = 0, \quad \text{ and } \quad \gamma(u) = -u \quad \forall u \in \F_q^{\times}.$$

\begin{rem}\label{10.1.1} In characteristic $2$ the presentation of quaternion algebras 
has to be modified.  Consider the case when $q=2$ and $y=T^2+T+1$. 
In this case, we let $D = F+Fi+Fj+Fij$, where
$$i^2+i = T', \quad j^2 = y_0 = T'^2+T'+1, \quad ji = (i+1)j.$$
Let $\cD'_{\text{max}} = A+Ai+Aj+Aij$, which is a maximal $A'$-order in $D$.
We take the Eichler $A'$-order $\cD' = A'+A'i+A'T'j+A'ij$ and the representatives of ideal classes of $\cD'$ in the following:
\begin{eqnarray}
I_{\infty} &=& \cD' = A'+A'i+A'T'j+A'ij, \nonumber \\
I_0 & =& A'T'+A'(1+i)+A'j+A'ij,\nonumber \\
I_1 & = & A'T'+A'(1+i)+A'(1+j)+A'ij. \nonumber
\end{eqnarray}
Then
\begin{eqnarray}
I_0I_0^{-1} & = & A'+A'i+A'xj+A'(i+1)j, \nonumber\\
I_0I_1^{-1} & = & A'T'+A'i+A'(1+j)+A'(1+ij), \nonumber\\
I_1I_1^{-1} & = & A'+A'T'i+A'j+A'(i+ij). \nonumber
\end{eqnarray}
Since
$$\{z \in \cD'_{\text{max}} \mid \Nr(z) = T'\} = \{i,1+i,1+T'+j\}$$
and
$$\{z \in \cD'_{\text{max}} \mid \Nr(z) = 1+T'\} = \{T'+j,T'+i+j,1+T'+i+j\},$$
we can get
$$B(T') = \begin{pmatrix}2&1&2\\1&2&2\\2&2&1\end{pmatrix} \quad \text{ and } \quad B(T'-1) = \begin{pmatrix} 0&2&1\\2&0&1\\1&1&1\end{pmatrix}.$$
\end{rem}

\subsection*{Acknowledgements} Part of this work was carried out while the first 
 author was visiting Taida Institute for Mathematical Sciences in Taipei and 
 National Center for Theoretical Sciences in Hsinchu. He heartily thanks 
Professor Winnie Li and Professor Jing Yu for their invitation, and  
 the members at these institutions for creating a welcoming and productive atmosphere. 
 
 This research was completed while the second author was visiting Max Planck Institute for Mathematics. 
 He wishes to thank the institute for kind hospitality and good working conditions.



\end{document}